\documentclass[11pt]{amsart}
\usepackage{amsmath,amssymb}

\theoremstyle{plain}
\newtheorem{thm}{Theorem}[section]
\newtheorem{theorem}[thm]{Theorem}

\newtheorem{lemma}[thm]{Lemma}
\newtheorem{corollary}[thm]{Corollary}
\newtheorem{proposition}[thm]{Proposition}
\theoremstyle{definition}
\newtheorem{remark}[thm]{Remark}

\newtheorem{notation}[thm]{Notation}
\newtheorem{definition}[thm]{Definition}

\newtheorem{assumption}[thm]{Assumption}

\newtheorem{example}[thm]{Example}

\newtheorem{question}[thm]{Question}

\numberwithin{equation}{section}

\newcommand{\sA}{{\mathcal A}}
\newcommand{\sB}{{\mathcal B}}
\newcommand{\sC}{{\mathcal C}}
\newcommand{\sD}{{\mathcal D}}
\newcommand{\sE}{{\mathcal E}}
\newcommand{\sF}{{\mathcal F}}

\newcommand{\sJ}{{\mathcal J}}
\newcommand{\sK}{{\mathcal K}}
\newcommand{\sL}{{\mathcal L}}

\newcommand{\sN}{{\mathcal N}}
\newcommand{\sO}{{\mathcal O}}
\newcommand{\sP}{{\mathcal P}}
\newcommand{\sQ}{{\mathcal Q}}
\newcommand{\sR}{{\mathcal R}}
\newcommand{\sS}{{\mathcal S}}
\newcommand{\sT}{{\mathcal T}}
\newcommand{\sU}{{\mathcal U}}
\newcommand{\sV}{{\mathcal V}}
\newcommand{\sW}{{\mathcal W}}
\newcommand{\sX}{{\mathcal X}}
\newcommand{\sY}{{\mathcal Y}}
\newcommand{\sZ}{{\mathcal Z}}

\newcommand{\C}{{\mathbb C}}

\newcommand{\F}{{\mathbb F}}
\newcommand{\G}{{\mathbb G}}

\newcommand{\N}{{\mathbb N}}
\newcommand{\BP}{{\mathbb P}}

\newcommand{\Z}{{\mathbb Z}}

\newcommand{\Sp}{{\rm Sp}}

\newcommand{\End}{{\rm End}}

\newcommand{\fg}{{\mathfrak g}}

\newcommand{\fgl}{{\mathfrak g}{\mathfrak l}}

\newcommand{\fsp}{{\mathfrak s}{\mathfrak p}}
\newcommand{\spo}{{\mathfrak s}{\mathfrak p}_{\omega}}
\newcommand{\aut}{{\mathfrak a}{\mathfrak u}{\mathfrak t}}

\newcommand\Aut{\rm Aut}

\newcommand\sd{\rtimes}

\def\Gr{\mathop{\rm Gr}\nolimits}
\def\Go{\G_{\omega}}
\def\no{{\rm \bf n}_{\omega}}
\def\GL{\mathop{\rm GL}\nolimits}

\def\Sym{\mathop{\rm Sym}\nolimits}

\def\Hom{\mathop{\rm Hom}\nolimits}

\def\min{\mathop{\rm min}\nolimits}

\title[Characterizing symplectic Grassmannians by VMRT]{Characterizing symplectic Grassmannians by varieties of minimal rational tangents}
\author{Jun-Muk Hwang and Qifeng Li}
\thanks{The authors are supported
by National Researcher Program 2010-0020413 of NRF}

\begin{document}

\maketitle

\begin{abstract} We show that
if the variety of minimal rational tangents (VMRT) of a uniruled projective manifold at a general point is projectively equivalent to that of a symplectic or an odd-symplectic Grassmannian, the germ of a general minimal rational curve is biholomorphic to the germ of a general line in a presymplectic Grassmannian. As an application, we characterize symplectic and odd-symplectic Grassmannians, among Fano manifolds of Picard number 1, by their VMRT at a general point and prove their rigidity under global  K\"ahler deformation.  Analogous results for $G/P$ associated with a long root were obtained by Mok and Hong-Hwang a decade ago by using Tanaka theory for parabolic geometries. When $G/P$ is associated with a short root, for which symplectic Grassmannians are most prominent examples, the associated local differential geometric structure is no longer a parabolic geometry and standard machinery of Tanaka theory cannot be applied because of several degenerate features. To overcome the difficulty, we show that Tanaka's method can be generalized to a setting much broader than parabolic geometries, by assuming a pseudo-concavity type condition that certain vector bundles arising from Spencer complexes have no nonzero sections. The pseudo-concavity type condition is checked by exploiting geometry of minimal rational curves.
\end{abstract}


\medskip
\tableofcontents
\section{Introduction}

In a series of joint works with Ngaiming Mok,  the first author developed the   theory of varieties of minimal rational tangents (VMRT) on uniruled projective manifolds (see \cite{HM98}, \cite{Hw01}, \cite{Hw12}, \cite{Hw14}, \cite{M16} for introductory surveys). The basic idea of this theory
is that a large part of the global geometry of a uniruled projective manifold $X$ is controlled by the projective geometry of its VMRT $\sC_x \subset \BP T_x X$ at a general point $x \in X,$ that is, the variety comprising tangents directions to rational curves of minimal degree through $x$.  One formulation of this idea is the following.

\begin{question}\label{q.1}
Let $X$ be  a uniruled projective manifold with a family $\sK$ of minimal rational curves.
Let $\sC_x \subset \BP T_x X$ be the VMRT of $\sK$ at a general point $x \in X$ and $C \subset X$ be a general member of $\sK$. To what extent the projective geometry
of $\sC_x \subset \BP T_xX$ determines the biholomorphic geometry of a Euclidean neighborhood (i.e. the germ ) of the curve $C \subset X$?  \end{question}

This question is interesting because
the holomorphic geometry of  a neighborhood of a minimal rational curve determines a substantial part of the algebraic geometry of $X$. For example, we have the following result from \cite{HM01}.

\begin{theorem}\label{t.CF}
Let $X_1$ and $ X_2$ be two Fano manifolds of Picard number 1 with families of minimal rational curves $\sK_1$ and $\sK_2$, respectively. Assume that the members of $\sK_1$ (resp. $\sK_2$) through a general point of $X_1$ (resp. $X_2$) form an irreducible family.    Let $C_1 \subset X_1$ (resp. $C_2 \subset X_2$) be a general member of $\sK_1$ (resp. $\sK_2$).
If there exists a Euclidean neighborhood $C_1 \subset O_{C_1} \subset X_1$ (resp.
$C_2 \subset O_{C_2} \subset X_2$) and a biholomorphic map $\varphi: O_{C_1} \cong O_{C_2}$ with $\varphi(C_1) = C_2$, then $\varphi$ can be extended to a biregular morphism $X_1 \cong X_2$. \end{theorem}

Although Theorem \ref{t.CF} is not explicitly stated in \cite{HM01}, its proof is contained in the proof of Main Theorem in \cite{HM01}, in particular Sections 3 and 4 of \cite{HM01}.
 An algebraic proof Theorem \ref{t.CF}, which works in arbitrary characteristic,   can be found in  \cite{Gu}, Corollary 3.13 of which is a version of Theorem \ref{t.CF}.

 By results like Theorem \ref{t.CF}, Question \ref{q.1} has become  an important issue in the theory of VMRT.  The first significant result on this question is the following result of Ngaiming Mok in \cite{Mo}.

\begin{theorem}\label{t.Mok}
In Question \ref{q.1}, assume that $(\sC_x \subset \BP T_x X)$ is projectively equivalent to
the VMRT $(\sC_s \subset \BP T_s S)$ of an irreducible Hermitian symmetric space $S$ at a point $s \in S$. Then some Euclidean neighborhood of $C \subset X$ is biholomorphic to a Euclidean neighborhood of a line in $S$. \end{theorem}

Combined with Theorem \ref{t.CF}, it has the following consequence.

\begin{theorem}\label{t.Mokrecog}
Let $X$ be a Fano manifold of Picard number 1. Assume that the VMRT $(\sC_x \subset \BP T_x X)$ at a general point $x \in X$ is projectively equivalent to the VMRT $(\sC_s \subset \BP T_s S)$ of an irreducible Hermitian symmetric space $S$  at a point $s \in S$.
Then $X$ is biregular to $S$. \end{theorem}

The original motivation for results like these was to apply them to certain rigidity problems of complex structures
(see  Subsection (3.3) of \cite{M16} and Theorem \ref{t.deform} below), but they have applications in problems in projective geometry, too. For example, Theorem \ref{t.Mokrecog} has been used in the classification
problem of certain prime Fano manifolds in Proposition 4.8 of \cite{FH18}. Theorem \ref{t.Mok} and
Theorem \ref{t.Mokrecog}  have been generalized to the case when $S$ is a homogeneous contact manifold
(\cite{Mo}), and more generally, when $S$ is a homogeneous space $G/P$  of a complex simple Lie group $G$ with a maximal parabolic subgroup $P$ associated to a long root (\cite{HH}).

What if $S=G/P$ is associated to a short root? The most prominent examples  of such $G/P$ are symplectic Grassmannians. It turns out that a strict analog of Theorem \ref{t.Mok} fails for a symplectic Grassmannian $S$. Counter-examples are provided by (desingularizations of) presymplectic Grassmannians
(Definition \ref{d.spgrass}).
The VMRT at a general point of a presymplectic Grassmannian with even nullity is projectively equivalent to that of a symplectic Grassmannian (see Proposition \ref{p.SpVMRT} below or Subsection (3.3) of \cite{M16}), while the biholomorphic geometry of germs of lines on it can be different from that of a symplectic Grassmannian (see Subsections \ref{ss.positive} and \ref{ss.rank0}).

So what should be the correct analog of  Theorem \ref{t.Mok} for the short root case, e.g., for symplectic Grassmannians?  This problem  has been puzzling experts for many years: it was stated as one of the most tantalizing problems in the field in p. 390 of \cite{Hw14} and was discussed extensively in Section 3 of \cite{M16}.   Our main result is an answer to this problem as follows. In fact, our argument works equally well for odd symplectic Grassmannians (in the sense of \cite{Mi}).

\begin{theorem}\label{t.Moksp}
Let $S$ be a symplectic Grassmannian or an odd-symplectic Grassmannian.
Let $X$ be a uniruled projective manifold with a family $\sK$ of minimal rational curves.  Assume that the VMRT $\sC_x \subset \BP T_x X$ of $\sK$ at a general point $x\in X$ is projectively equivalent to
the VMRT $\sC_s \subset \BP T_s S$ of lines at a general point $s \in S.$   Then some Euclidean neighborhood of a general member of $\sK$  is biholomorphic to a Euclidean neighborhood of a general line in one of the presymplectic Grassmannians. \end{theorem}

Why is the behavior of symplectic Grassmannians regarding VMRT so different  from that of $G/P$ associated with a long root, such as Grassmannians or orthogonal Grassmannians? The key difference is the following.
\begin{itemize}
\item[(I)] If $S=G/P$ is associated to a long root, the VMRT $\sC_s \subset \BP T_s S$ is a homogeneous variety. When $(\sC_x \subset \BP T_x X)$ in Question \ref{q.1} is projectively equivalent to $(\sC_s \subset \BP T_s S)$, the VMRT's define a parabolic geometry (in the sense of \cite{CSl}) in a neighborhood of $x \in X$.
\item[(II)] If $S = G/P$ is associated to a short root, the VMRT $\sC_s \subset \BP T_s S$ is {\em not} a homogeneous variety. When $(\sC_x \subset \BP T_x X)$ in Question \ref{q.1} is projectively equivalent to $(\sC_s \subset \BP T_s S)$,  the differential geometric structure in a neighborhood of $x \in X$  defined by the VMRT's is {\em not} a parabolic geometry. \end{itemize}

    The  parabolic geometry in (I) plays a crucial role in Theorem \ref{t.Mok}. The key point in the proof of Theorem \ref{t.Mok} in \cite{Mo} is to show that such a parabolic geometry can be extended to a neighborhood $O_C$ of the curve $C$. Once this is done, the standard differential geometric machinery -- Cartan connections constructed by Tanaka in \cite{Ta}-- can be employed to settle the problem.
     As we will see in Section \ref{s.vmrt}, the argument of Mok can be modified in the case of symplectic Grassmannians in (II) to show that the corresponding differential geometric structure can be extended to a neighborhood $O_C$. But the main difficulty begins from there.
     The differential geometric structures arising in (II) have not been studied by differential geometers. These structures have several degenerate features which make it hard to apply the available generalizations  of Tanaka theory  such as \cite{Mor} and \cite{Ki}.

      We overcome the difficulty by constructing a Cartan connection in a fairly general setting
      (see Section \ref{s.Cartan}). Our construction itself is a  modification of Tanaka's prolongation procedure, but it applies to a wide class of homogeneous spaces, including those with degenerate features.  The price we pay for this wide generality is the  requirement of a condition unusual in local differential geometry: the vanishing of sections of certain vector bundles arising from Spencer complexes (see Theorem \ref{t.Cartan}). Such a vanishing condition seldom holds locally, which explains why it has never been considered by differential geometers.   Our condition is meaningful only in the global holomorphic setting: it is a kind of pseudo-concavity of the underlying complex manifold. In our setting, this condition can be checked  (see Section \ref{s.SAF}) using the geometry of rational curves. We expect that the result of Section \ref{s.Cartan} can be applied to analogous problems for many quasi-projective homogeneous varieties.

Combined with Theorem \ref{t.CF}, Theorem \ref{t.Moksp} gives the following.

\begin{theorem}\label{t.recog}
Let $S$ be a symplectic Grassmannian or an odd-symplectic Grassmannian.
Let $X$ be a Fano manifold of Picard number 1 with a family $\sK$ of minimal rational curves. Assume that the VMRT $\sC_x \subset \BP T_x X$ of $\sK$ at a general point $x \in X$ is projectively equivalent to the VMRT $\sC_s \subset \BP T_s S$ of lines at a general point $s \in S$.
Then $X$ is biregular to $S$. \end{theorem}

 Theorem \ref{t.recog} has the following application to deformation rigidity.

\begin{theorem}\label{t.deform}
Let $\pi: \sX \to \Delta$ be a smooth projective morphism from a complex manifold $\sX$ to
the unit disc $\Delta :=\{ t \in \C, \ |t| <1 \}$. If  the fiber $\pi^{-1}(t)$ is biholomorphic to a symplectic Grassmannian or an odd-symplectic Grassmannian for each $t \in \Delta \setminus \{ 0 \}$, then so is the central fiber $\pi^{-1}(0)$. \end{theorem}

For symplectic Grassmannians, Theorem \ref{t.deform} was proved in \cite{HM05} by a complicated limiting argument. Our proof here is more stream-lined and works equally well
for odd-symplectic Grassmannians. Combined with the infinitesimal deformation rigidity of odd-symplectic Grassmannians  (Theorem 0.5 in \cite{PP}),  Theorem \ref{t.deform} shows that complex structures of odd symplectic Grassmannians do not admit nontrivial deformations as K\"ahler manifolds.

We mention that there is a related work \cite{OSW} where a special case of  Theorem \ref{t.recog} is considered for symplectic Grassmannians under the stronger assumption  that $\sC_x \subset \BP T_x X$ is isomorphic to $\sC_s \subset \BP T_s S$ at {\em every point } $x \in X$. The method of \cite{OSW} is purely algebro-geometric and completely different from ours. Because of the stronger assumption, their approach does not lead to applications like Theorem \ref{t.deform}.

It is natural to ask whether an analog of Theorem \ref{t.recog} holds for any $S=G/P$ with maximal parabolic $P$. Only two cases remain unsettled after Theorem \ref{t.recog}: when $G$ is of type $F_4$ and $P$ is defined by either the  short root $\alpha_1$ or the short root $\alpha_2$ (in the notation of Section 7 in \cite{HM05}). We believe that our approach can be applied in these cases: the resulting differential geometric structures are expected to be covered by the construction in Section \ref{s.Cartan}. As we see below, however, checking the structure of Lie algebra prolongations  and the vanishing conditions for Spencer complexes
requires substantial work involving many special features of the individual homogeneous space. For this reason, we leave the question on these two $F_4$-homogeneous spaces for future investigations.

\medskip
The paper consists of two parts: Section \ref{s.Cartan} and the rest.
The general construction of Cartan connections in Section \ref{s.Cartan} is of independent interest. It is the differential geometric core of the paper and little knowledge of algebraic geometry is used in this section.
 Except those needed for the statement of Theorems \ref{t.Cartan} and \ref{t.Cartan'}, the notation and symbols in Section \ref{s.Cartan} will not be used in the rest of the paper.

The rest of the paper is concerned with the situation of Theorem \ref{t.Moksp}.
We introduce presymplectic Grassmannians and study their geometry in Section \ref{s.pre}.
Section \ref{s.prolong} will be devoted to the graded Lie algebra structure arising from
presymplectic Grassmannians. We check the vanishing conditions for the bundles arising from Spencer complexes in Section \ref{s.SAF}. Section \ref{s.Z}  presents some geometric properties of the VMRT of presymplectic Grassmannians. The theory of VMRT itself will not appear explicitly until the very last section, Section \ref{s.vmrt}, where we will give the proofs of Theorems \ref{t.Moksp}, \ref{t.recog}  and \ref{t.deform}.

\bigskip
\noindent {\bf Conventions}

\noindent 1.
We work in the holomorphic category. All geometric objects are defined over complex numbers and holomorphic.

\noindent 2.  When we say an open set, we mean it  in the
classical topology. An open set in Zariski topology will be called
a Zariski open set.

\noindent 3.  The projectivization $\BP V$ of a vector space  $V$
is the set of 1-dimensional subspaces in $V$. The dual space of $V$ is denoted by $V^*$.

\noindent 4. For two vector bundles $\sU, \sV$ on a manifold $M$, the space of sections of the vector bundle $Hom(\sU, \sV)$ on $M$ is denoted by $\Hom(\sU, \sV)$. If $\sV$ is the trivial vector bundle $M \times V$ associated to a vector space $V$, we write $\Hom( \sU, V)$ in place of $\Hom( \sU, M \times V)$.

\bigskip

\section{A construction of Cartan connections}\label{s.Cartan}

\subsection{Generalities on Cartan connections}
To fix notation, we recall the definition of a Cartan connection on a principal bundle.

\begin{definition}\label{d.Cartan}
Let $G$ be a connected algebraic group and $G^0 \subseteq G$ be a connected algebraic subgroup. Let $\fg^0 \subset \fg$ be their Lie algebras. A {\em Cartan connection of type} $(G, G^0)$ on a complex manifold $M$ with $\dim M =\dim G/G^0$ is a principal $G^0$-bundle $E \rightarrow M$ with a $\fg$-valued $1$-form $\Upsilon$ on $E$ with the following properties.
\begin{itemize}
\item[(i)] For $A \in \fg^0$, denote by $\zeta_A$ the fundamental vector field on $E$ induced by the right $G^0$-action on $E$. Then $\Upsilon(\zeta_{A})=A$ for each $A \in \fg^0$.
\item[(ii)] For $a \in G^0$, denote by $R_a: E \to E$ the right action of $a$. Then $R_{a}^{*}\Upsilon={\rm Ad}(a^{-1}) \circ \Upsilon$ for each $a\in G^0$.
\item[(iii)] The linear map $\Upsilon_{y}: T_{y} E \rightarrow \fg$ is an isomorphism for each $y \in E$. \end{itemize}
A standard example of a Cartan connection of type $(G,G^0)$ is given by $E= G$ and $M = G/G^0$ with the Maurer-Cartan form on $G$ as $\Upsilon.$ We will call it the {\em flat Cartan connection} of type $(G, G^0)$ and denote it by $E^{\rm flat}$ and $\Upsilon^{\rm flat}$.
\end{definition}

\begin{definition}\label{d.equivalence}
Let $(E, \Upsilon)$ be a Cartan connection of type $(G, G^0)$ on a complex manifold $M.$
\begin{itemize}
\item[(1)] Given an open subset $O \subset M$, the restriction $(E, \Upsilon)|_O$ is a Cartan connection on $O$ given by the restriction of $\Upsilon$ to the open subset
    $E|_O \subset E$.
\item[(2)] For a Cartan connection $(E', \Upsilon')$  of type $(G, G^0)$ on a complex manifold $M',$  an open immersion $\varphi: M \to M'$   {\em lifts to an equivalence} of $(E, \Upsilon)$ and $(E', \Upsilon')$, if there exists an open immersion $\phi: E \rightarrow E'$ descending to $\varphi$  such that $\phi^{*}\Upsilon'=\Upsilon$.
\item[(3)] The Cartan connection $(E, \Upsilon)$  is  {\em locally flat}, if each point $x \in M$ has a neighborhood $O$ with an open immersion
 $\varphi: O \to G/G^0$ which lifts to an equivalence of $(E, \Upsilon)|_O$ and $(E^{\rm flat}, \Upsilon^{\rm flat})|_{\varphi(O)}$.  \end{itemize}
\end{definition}

The following lemma is an immediate consequence of the fact that a Cartan connection is locally flat if and only if a naturally defined $\fg$-valued holomorphic 2-form (arising from the Maurer-Cartan equation) on the principal bundle $E$ vanishes identically (see Definition 1.10 and Theorem 5.1 in Chapter 5 of \cite{Sh}).

\begin{lemma}
 A Cartan connection of type $(G, G^0)$ on $M$ is locally flat if and only if its restriction on some open subset of $M$ is locally flat.\end{lemma}

The following result is from Theorem 3.1  of \cite{Bi}, where
 it was stated for a rationally connected projective manifold $M$.  The argument there works
for any complex manifold $M$ as long as there is   a holomorphic map $f: \BP^1 \to M$ such that
$f^* T(M)$ is ample.

\begin{proposition}\label{p.Biswas}
Let $M$ be a complex manifold admitting a holomorphic map $f: \BP^1 \to M$ such that
$f^* T(M)$ is ample. Then any Cartan connection of type $(G, G^0)$ on $M$ is locally flat.
\end{proposition}

The existence of a locally flat Cartan connection has the following well-known consequence,
the existence of developing maps.

\begin{proposition}\label{p.develop}
Suppose there exists a locally flat Cartan connection of type $(G, G^0)$ on a simply-connected complex manifold $M$. Then there exists an open immersion $h: M \to G/G^0$
which  lifts to an equivalence of Cartan connections. \end{proposition}

\begin{proof}
 Pick  a base point
$x_0 \in M$. By the local flatness, we have a neighborhood $O_{x_0}$ of $x_0$ and an open immersion $h_0: O_{x_0} \to G/G^0$ which lifts to an equivalence of Cartan connections.   For any point $x \in M$, we can analytically continue $h_0$ along a path in $M$ joining $x_0$ to $x$ as follows. Choose
 finitely many points $x_1, \ldots, x_N=x$ on the path with neighborhoods $O_{x_1}, \ldots, O_{x_N}$ in $M$ such that \begin{itemize} \item[(1)] the union
 $O_{x_0} \cup O_{x_1} \cup \cdots \cup O_{x_N}$ is a neighborhood of the path; \item[(2)] each intersection
 $O_{x_i} \cap O_{x_{i+1}}$ is connected and simply-connected;
  \item[(3)] there is an open immersion $h_{i}: O_{x_i} \to G/G^0$ which lifts to an equivalence of  Cartan connections.
     \end{itemize}
      By Theorem 5.2 in Chapter 5 of \cite{Sh}, there exists $g_1 \in G$ such that
      $g_1 \cdot h_{1}$ agrees with $h_0$ in the intersection $O_{x_0} \cap O_{x_1}$.
      Replace $h_1$ by $g_1 \cdot h_1$ to obtain an extension of $h_0$ to $O_{x_0} \cup O_{x_1}$.  Applying the same argument, we have  $g_2 \in G$
      such that $g_2 \cdot h_2$ agrees with $h_1$ in the intersection $O_{x_1} \cap O_{x_2}$. Replace $h_2$ by $g_2 \cdot h_2$ to obtain an extension of $h_0$ to
      $O_{x_0} \cup O_{x_1} \cup O_{x_2}$. Repeating this process, we end up with
      a holomorphic map from a neighborhood of the path into $G/G^0$ extending $h_0$.
       As $M$ is simply connected, this extension does not depends on the choice of paths and we obtain an extension of $h_0$ to the whole $M$.
       \end{proof}

These propositions show that  on a manifold admitting many rational curves, the existence of a Cartan connection has a strong consequence. The main goal of this section is the construction of Cartan connections associated to certain geometric structures for a   class of $(G,G^0)$.

\subsection{Prolongations of Lie algebras}

In this subsection, we introduce a  class of pairs of connected algebraic groups $(G, G^0)$, as specified in Proposition \ref{p.algebraic} below.

\begin{notation}
For two graded vector spaces $A = \oplus_{i \in \Z} A_i$ and $B= \oplus_{j \in \Z} B_j$, there is a natural gradation
$$ \Hom (A, B) = \bigoplus_{\ell \in \Z} \Hom(A,B)_{\ell} $$ defined by
$$\Hom(A,B)_{\ell} := \oplus_{k \in \Z} \Hom(A_k, B_{k + \ell}).$$
\end{notation}

\begin{definition}\label{d.fg}
Fix a positive integer $\nu$. Let $\fg_{-} = \fg_{-1} \oplus \cdots \oplus \fg_{-\nu}$ be a graded nilpotent Lie algebra. Denote by ${\rm gr}\Aut(\fg_{-})$ the group of  Lie algebra automorphisms of $\fg_{-}$ preserving the gradation and by ${\rm gr}\aut(\fg_{-})$ its Lie algebra. Fix a connected algebraic subgroup $G_0 \subset {\rm gr}\Aut(\fg_{-})$ and its Lie algebra $\fg_0 \subset {\rm gr}\aut(\fg_{-})$. For each positive integer $i,$ the $i$-th {\em prolongation} of $\fg_0$ is inductively defined as
\begin{align*}
\begin{aligned}
\fg_i &:= \Big\{ \phi \in \Hom(\mathfrak{g}_{-}, \bigoplus\limits_{-\nu \leq j < i} \mathfrak{g}_j)_i := \bigoplus\limits_{k=1}^{\nu} \Hom(\mathfrak{g}_{-k}, \mathfrak{g}_{-k + i}), \\
 & \qquad \phi([v_1, v_2]_{\mathfrak{g}_{-}}) = \phi(v_1)(v_2) - \phi(v_2)(v_1), \quad \text{for any} \quad v_1, v_2 \in \mathfrak{g}_{-}   \Big\}.
\end{aligned}
\end{align*}
Here $[\quad, \quad]_{\mathfrak{g}_{-}}$ denotes the Lie bracket on $\mathfrak{g}_{-}$ and,  if $\phi(v_{1})\in \mathfrak{g}_{-}$, then $$\phi(v_{1})(v_{2}):=[\phi(v_{1}), v_{2}]_{\mathfrak{g}_{-}}.$$ For convenience, we put $\fg_{-\nu - j} =0$ for every positive integer $j$ and write $$\fg_{-} = \bigoplus_{k \in \N} \fg_{-k}.$$ The graded vector space $$\fg:= \bigoplus_{k \in \Z} \fg_k$$ is called the
{\em universal prolongation } of $(\fg_0, \fg_{-})$.
\end{definition}

The following is well-known (see p. 430 of \cite{Ya}).

\begin{lemma}\label{l.fg}
In Definition \ref{d.fg}, write $$\fg^k: = \bigoplus_{j \geq k} \fg_j \mbox{ for each } k \in \Z.$$ Then $$\fg^{-\nu} = \fg=  \bigoplus_{j \in \Z} \fg_j$$ has a unique graded Lie algebra structure whose Lie bracket $[\quad, \quad]_{\fg}$ satisfies
\begin{itemize} \item[(i)] $[v_{1}, v_{2}]_{\fg}=[v_{1}, v_{2}]_{\mathfrak{g}_{-}}$  for any $ v_{1}, v_{2}\in \mathfrak{g}_{-}$;
\item[(ii)] $[\varphi, v]_{\fg}=\varphi(v)$ for any  $\varphi\in \fg^0$ and $ v\in \mathfrak{g}_{-}$;
\item[(iii)] $[\varphi_{1}, \varphi_{2}]_{\fg}(v)=[\varphi_{1}(v), \varphi_{2}]_{\fg}+[\varphi_{1}, \varphi_{2}(v)]_{\fg}$ for any $\varphi_{1}, \varphi_{2}\in \fg^0$ and $ v\in \mathfrak{g}_{-}.$
    \item[(iv)] $[\mathfrak{g}_i, \mathfrak{g}_j] \subseteq \mathfrak{g}_{i+j}$ for all $i,j \in \mathbb{Z}$.
\end{itemize}
In particular,  the subspace $\fg^0$ is a subalgebra of $\fg$ and
$\fg^{\ell}$ is an ideal of $\fg^0$ for each positive integer $\ell$.  \end{lemma}

\begin{definition}\label{d.cohomology}
In the setting of Lemma \ref{l.fg}, define, for each  $\ell \in \N$,
\begin{eqnarray*}
C^{\ell, 1}(\mathfrak{g}_{0}) & := &\Hom (\mathfrak{g}_{-}, \mathfrak{g})_{\ell} \\ & =& \bigoplus_{j\in \N}\Hom(\fg_{-j}, \mathfrak{g}_{-j+\ell}) \\
C^{\ell, 2}(\fg_0) & := & \Hom(\wedge^2 \fg_{-}, \fg)_{\ell} \\ &=&
\bigoplus_{i, j \in \N} \Hom(\fg_{-i} \wedge \fg_{-j}, \fg_{- i -j + \ell}).
\end{eqnarray*}
For an element $f \in C^{\ell, 1}(\fg_0)$, define $\partial f \in C^{\ell, 2}(\mathfrak{g}_0)$ by
$$\partial f(u, v)  = [f(u), v] + [u, f(v)] -f([u, v]) \mbox{ for all } u, v \in \fg_{-}.$$ This determines a $G_0$-module homomorphism, called the Spencer coboundary operator,  $$\partial: C^{\ell, 1}(\fg_0) \to C^{\ell, 2}(\fg_0).$$  Its cokernel $C^{\ell, 2}(\fg_0)/\partial(C^{\ell,1}(\fg_0))$ has a natural $G_0$-module structure. \end{definition}

The following is immediate from the definition.

\begin{lemma}\label{l.cohom}
Let $\fg$ be as in Definition \ref{d.fg}. Then $$\fg_{\ell} = {\rm Ker}(\partial:
 C^{\ell, 1}(\fg_0) \to C^{\ell, 2}(\fg_0))$$ for each positive integer $\ell$.
\end{lemma}

\begin{proposition}\label{p.algebraic}
In the setting of Definition \ref{d.fg},  assume that
\begin{itemize}
 \item[(i)] the action of $G_0$ on $\fg_{-}$ has no nonzero fixed vector, i.e.,
$$\{v \in \fg_{-}, \ \fg_0 \cdot v =0\} =0; \mbox{ and }$$
 \item[(ii)] there exists a  nonnegative integer $\mu$
such that $\fg_{\mu} \neq 0$ and $ \fg_{k} = 0$ for all $k > \mu$.
\end{itemize}
Note that the assumption (ii) implies that  $\fg$ has finite dimension.
Then the following holds.  \begin{itemize} \item[(1)] The adjoint representation of $\fg$  is faithful, defining an injection $\fg \cong {\rm ad}(\fg) \subset {\rm End}(\fg)$. \item[(2)] The connected Lie subgroup $G \subset {\rm GL}(\fg)$ corresponding to the Lie subalgebra  ${\rm ad}(\fg) \subset {\rm End}(\fg)$ is an algebraic subgroup.  \item[(3)] For each nonnegative integer $\ell$, the connected Lie subgroup $G^{\ell} \subset G$ corresponding to the subalgebra $\fg^{\ell} \subset \fg$ is an algebraic subgroup. \item[(4)] For each positive integer $\ell$,  the algebraic subgroup $G^{\ell}$ in (3) is a normal subgroup of $G^0$.
\item[(5)] The subgroup $\exp(\fg_0)$ of $G^0$ corresponding to the Lie algebra $\fg_0$ is naturally isomorphic to $G_0$ and  is projected isomorphically to the quotient group $G^0/G^1$. To simplify the notation, we will write $\exp(\fg_0)$ as $G_0 \subset G^0$ and
    also its image in $G^0/G^{\ell}, \ell \geq 1,$ as $G_0 \subset G^0/G^{\ell}$.  \end{itemize}  \end{proposition}

\begin{proof}
(1) is immediate from the assumption (i) and the definition of the prolongations $\fg_i$ in Definition \ref{d.fg}.

The connected subgroups of ${\rm GL}(\fg)$ corresponding to the subalgebras $\fg_{-}$ and $\fg^1$ are algebraic because $\fg_{-}$ and $\fg^1$ are nilpotent subalgebras.
Since we have chosen $G_0 \subset {\rm GL}(\fg_{-})$ as an algebraic subgroup, it defines an algebraic subgroup of ${\rm GL}(\fg)$ under the natural $G_0$-actions on $\fg_{\ell}, \ell \in \N$. As $\fg$ is generated by $\fg_{-}, \fg^1$ and $\fg_0$, we see that $G$ is an algebraic subgroup of ${\rm GL}(\fg)$ (e.g. by Corollary 24.5.10 in \cite{TY}). This proves (2) and (3).

    (4) follows from the fact that $\fg^{\ell}$ is an ideal of $\fg^0$ as mentioned in Lemma \ref{l.fg}.

    From the proof of (2), we see a natural inclusion $G_0 \subset G$, which should coincide with $\exp(\fg_0)$. As the $G_0$-action on $\fg$ preserves the grading, while the action of $G^1$ does not, we see that $\exp(\fg_0) \cap G^1 = {\rm Id}_G$. Thus $$G_0 \cong \exp(\fg_0) \cong G^0/G^1,$$ proving (5).
\end{proof}

\subsection{Statement of the main result on Cartan connections}

Our main result is a construction of a Cartan connection associated to a certain geometric structure on a complex manifold equipped with a filtration.

\begin{definition}\label{d.filtration}
Let $E$ be a complex manifold. A {\em filtration } on $E$ is a collection of subbundles
$(F^j E, j \in \Z)$ of the tangent bundle $TE$ such that
 \begin{itemize}
 \item[(i)]
 $F^{j+1} E \subset F^{j}E$ is a vector subbundle for each $j \in \Z$;
 \item[(ii)] $F^{-\nu}E = TE$ for some $\nu \in \N$;
 and
 \item[(iii)] regarding the locally free sheaf $\sO(F^iE)$ as a sheaf of vector fields on $E$, we have $$[\sO(F^{-i}E), \sO( F^{-j}E)] \subset \sO(F^{-k}E)$$ for the Lie bracket of vector fields and for any nonnegative integers
     $i, j, k,$
 satisfying $i+j \leq k$. \end{itemize}
 In this case, we say that $(E, F^jE, j \in \Z)$ is a {\em filtered manifold}.
 For a point $x \in E$, we will denote by $F^{k}_x E$ the fiber of $F^k E$ at the point $x$. By (iii), the graded vector space $$\bigoplus_{i \in \N} F_x^{-i}E/F_x^{-i+1} $$
 for each point $x \in E$ has the structure of a nilpotent graded Lie algebra, called the {\em symbol algebra} of the filtration.
\end{definition}

\begin{definition}\label{d.partialform}
Let $(E, F^j E, j \in \Z)$ be a filtered manifold.
 For a vector space $V$, an element of $\Hom(F^{-k}E, V)$ for some $k \in \N$ is called a
{\em partial form} on $E$.
 In general a partial form $\vartheta \in \Hom(F^{-k}E, V)$ is not a
1-form on $E$, so the derivative ${\rm d} \vartheta$ does not make sense as a 2-form.  But for each $i, j \geq 0$ satisfying $i+j \leq k$, we can define an element $${\rm d} \vartheta
\in \Hom ( F^{-i}E \otimes F^{-j}E, V)$$ as follows.
For a point $x \in E$, choose a neighborhood $U \ni x$ such that there exists a holomorphic splitting $$ TE|_U \cong F^{-k}E \oplus J$$ for some subbundle $J \subset TE|_U$.   Extend $\vartheta$ to a $V$-valued 1-form $\widetilde{\vartheta}$ on $U$ by setting $\widetilde{\vartheta} (J) = 0$.
For $\xi \in F^{-i}_x E$
and $\eta \in F^{-j}_x E,$ define $${\rm d} \vartheta (\xi, \eta) = {\rm d} \widetilde{\vartheta} (\xi, \eta) \ \in V.$$ This definition does not depend on the choice of $\widetilde{\vartheta}$.
In fact, if we choose  local sections $\widetilde{\xi}$ of $F^{-i}E$ and $\widetilde{\eta}$ of $F^{-j}E$ in $U$ such that $\widetilde{\xi}|_x =\xi$ and $\widetilde{\eta}|_x = \eta$, then
$${\rm d}\vartheta (\xi, \eta) = \widetilde{\xi} (\vartheta (\widetilde{\eta})) -
\widetilde{\eta} (\vartheta (\widetilde{\xi})) - \vartheta( [\widetilde{\xi}, \widetilde{\eta}]).$$ The right hand side does not depend on the choice of $\widetilde{\vartheta}$ because of Definition \ref{d.filtration} (iii), while the left hand side is independent of the choice of $\widetilde{\xi}$ and $\widetilde{\eta}.$
\end{definition}

\begin{definition}\label{d.filtered}
Let $\fg_{-}= \oplus_{k\in \N} \fg_{-k}$ be a graded nilpotent Lie algebra with $\fg_{-j} =0$ for all $j$ larger than for a fixed positive integer $\nu$.
A {\em filtration of type} $\fg_{-}$ on a complex manifold $M$ is a filtration $(F^j M, j \in \Z)$ on $M$ such that
\begin{itemize}
\item[(i)] $F^k M =0$ for all $k \geq 0$; \item[(ii)] $F^{-k} M = TM$ for all $k \geq \nu$; and
\item[(iii)] for any $x\in M$, the symbol algebra $${\rm gr}_{x}(M)
:= \bigoplus_{i \in \N} F^{-i}_{x}M/F^{-i+1}_{x}M$$  is isomorphic to $\mathfrak{g}_{-}$ as graded Lie algebras. \end{itemize}
\end{definition}

\begin{definition}\label{d.G_0-structure}
In the setting Definition \ref{d.filtered}, the {\em graded frame bundle} of the manifold $M$ with a filtration of type $\fg_{-}$  is the  ${\rm gr}\Aut(\fg_{-})$-principal bundle ${\rm grFr}(M)$ on $M$ whose fiber at $x$ is the set of graded Lie algebra isomorphisms from
$\fg_{-}$ to ${\rm gr}_x(M)$. Let $G_0 \subset {\rm gr}\Aut(\fg_{-})$ be a connected algebraic subgroup. \begin{itemize}
 \item[(1)] A $G_0$-{\em structure subordinate to the filtration} on $M$ means a $G_0$-principal subbundle  $E_0 \subset {\rm grFr}(M)$.
  \item[(2)]   Denoting by  $p: E_0 \to M$ the natural projection and by ${\rm d}p: TE_0 \to TM$ its derivative, the manifold $E_0$ is equipped with a filtration $F^{i}E_0$  whose fiber $F^{i}_{y}E_{0}$ at $y \in E_{0}$ is $$({\rm d}_y p)^{-1} \left(F^{i}_{p(y)} M \right) \mbox{ if } i \leq 0,$$ and 0 if $i \in \N$.
\item[(3)] For each $k \in \N$, there exists a natural partial form $s_{-k} \in \Hom(F^{-k} E_0, \fg_{-k})$ defined as follows.
Each $y \in E_0$ corresponds to  a
graded Lie algebra isomorphism $\fg_{-} \to {\rm gr}_{x}(M), x = p(y).$ Denote by
$y^{-1}: {\rm gr}_{x}(M) \to \fg_{-}$ the inverse isomorphism.
For a vector $\xi \in F^{-k}E_0$, let $[{\rm d} p (\xi)] \in {\rm gr}_{x}(M)$ be the class of ${\rm d} p (\xi)
\in F^{-k}_x M.$ Define $$ s_{-k} (\xi) := y^{-1} ([{\rm d} p (\xi)] ) \in \fg_{-k}.$$
The collection of partial forms $s = (s_{-k}, k \in \N)$ on $E_0$ is called the {\em soldering form } of $E_0$. \end{itemize} \end{definition}

The main result of this section is the following construction of a Cartan connection associated to a $G_0$-structure subordinate to a filtration.

\begin{theorem}\label{t.Cartan}
  Assume that $G_0 \subset {\rm gr}\Aut(\fg_{-})$ satisfies the conditions (i) and (ii) of  Proposition \ref{p.algebraic}. In particular, we have $$\fg_{k} \neq 0 \mbox{ only if } - \nu \leq k \leq \mu.$$
Let $(M, F^{j}M, j \in \Z)$ be a complex manifold  with a filtration of type $\fg_{-}$, and $E_0 \subset {\rm grFr}(M)$ be a $G_0$-structure on $M$ subordinate to the filtration. Let $$C^{\ell,2}(E_0)/\partial(C^{\ell,1}(E_0))$$ be the vector bundle associated to the principal $G_0$-bundle $E_0$ by the natural representation of $G_0$ on
$$C^{\ell, 2}(\fg_0)/\partial(C^{\ell,1}(\fg_0))$$ from Definition \ref{d.cohomology}.  Assume that $H^0(M, C^{\ell,2}(E_0)/\partial(C^{\ell,1}(E_0)) ) =0$ for all $1 \leq \ell \leq \mu + \nu.$

Then there exists a Cartan connection $(E \to M, \Upsilon)$ of type $(G, G^0)$ on $M$ such that the $G_0$-structure $E_0 \to M$ is isomorphic to the quotient of $E \to M$ by the normal subgroup
$G^1 \subset G^0$ and the soldering form $s$ on $E_0$ is given by the quotient of  the $\fg_{-}$-components of $\Upsilon$, i.e., for any $k \in \N$ and $\xi \in F^{-k} E_0$,
$$s_{-k} (\xi) = \fg_{-k}\mbox{-component of } \Upsilon(\widetilde{\xi})$$
where
$\widetilde{\xi} \in TE$  is a vector  projected to $\xi$ by the quotient map $E \to E_0$.  \end{theorem}

The next four subsections will be devoted to the proof of Theorem \ref{t.Cartan}.
In Subsection \ref{ss.special}, we present a modified version of Theorem \ref{t.Cartan}.

\subsection{$G^0$-frame bundle of length $\ell$}\label{ss.G^0}

The proof of Theorem \ref{t.Cartan}  follows a standard prolongation procedure developed by N. Tanaka (e.g. \cite{Ta}). There are several presentations (e.g. \cite{CS}, \cite{CSl} \cite{Mor}) of this procedure  each of which uses a construction somewhat different from others, although the basic ideas are all equivalent to Tanaka's. Unfortunately, they do not fit our purpose completely and we have to modify them.
Although it should be possible to adapt any of these presentations,  we find the one by Cap and Schichl (\cite{CS}) most convenient for our purpose and perhaps most accessible to algebraic geometers. In this subsection, we give a number of definitions adapted from \cite{CS}, modifying some notation and terminology to make it more streamlined and accessible to algebraic geometers.  It is not necessary to be familiar with \cite{CS} to follow our argument, but it will help the reader to consult  \cite{CS} for the detailed explanation at several points of the construction.

\begin{notation}\label{n.quotient}
As $\fg$ has a fixed gradation, there is a natural identification for any integers $i \leq j, $
$$\fg_i + \fg_{i+1} + \cdots + \fg_j  \cong \fg^i/\fg^{j+1}.$$ In \cite{CS}, the left hand side is used consistently, but we will often use the right hand side notation whenever we find it  more convenient. We often regard an  element in the quotient
$\fg^i/\fg^{j+1}$ as an element of $\fg$ by the above identification. \end{notation}

\begin{definition}\label{d.inducedfiltration}
Let $(M, F^j M, j \in \Z)$ be a complex manifold with a filtration of type $\fg_{-}$ and let $(G, G^0)$ be as in Proposition \ref{p.algebraic}.
Fix a positive integer $\ell$. Let $\pi: E \to M$ be a $G^0/G^{\ell}$-principal bundle on $M$.
For each $A \in \fg^0/\fg^{\ell}$, denote by $\zeta_A$ the fundamental vector field on $E$ induced by the right $G^0/G^{\ell}$-action on $E$.
   We define a filtration $(F^k E \subset T E, k \in \Z) $ on $E$ as follows. \begin{itemize}
\item[(a)] For $k \geq \ell$, the subbundle $F^{k} E$ is zero.
\item[(b)] For $0 \leq k < \ell$, the subbundle $F^k E \subset TE$ is generated by the
vector fields $\{\zeta_A, \ A \in  \fg^k/\fg^{\ell}\}$ induced by the right $G^0/G^{\ell}$-action on $E$.
\item[(c)] For $k \in \mathbb{N}$, the subbundle $F^{-k} E \subset TE$ has the fiber
 $F^{-k}_y  E$ at $y \in E$ defined by
 $$F^{-k}_y E := ({\rm d}_y \pi)^{-1} (F^{-k}_{\pi(y)} M)$$ where
  ${\rm d}_y \pi: T_y E \to T_{\pi(y)}M$ is the derivative of $\pi$ at $y$. \end{itemize}
  \end{definition}

\begin{definition}\label{d.length}
In the setting of Definition \ref{d.inducedfiltration}, a $G^0$-{\em frame bundle of length} $\ell$ is
a  $G^0/G^{\ell}$-principal bundle $\pi: E \to M$ equipped with a collection of partial forms
$$\theta = \{ \theta_{-j} \in \Hom( F^{-j}E, \fg^{-j}/\fg^{-j+ \ell}), j \in \N \}$$
  satisfying the following properties.
 \begin{itemize}
\item[(i)] For each $y \in E$, the kernel of  $\theta_{-j}|_y \in \Hom(F^{-j}_y E, \fg^{-j}/\fg^{-j+\ell}) $ is
$F^{-j + \ell}_y E,$ i.e., the homomorphism $\theta_{-j}|{y}$ induces an isomorphism
 $$F^{-j}_y E/ F_y^{-j+\ell} E \cong \fg^{-j}/\fg^{-j+ \ell}. $$\item[(ii)] When $j \geq 2$, the restriction of $$\theta_{-j} \in \Hom(F^{-j}E, \fg^{-j}/\fg^{-j+ \ell})$$ to the subbundle $F^{-j+1}E \subset F^{-j}E$
 is equal to the composition of $$\theta_{-j+1} \in \Hom(F^{-j+1}E, \fg^{-j+1}/\fg^{-j +1 + \ell})$$ with the natural homomorphism $$\fg^{-j+1}/\fg^{-j+1 + \ell} \to \fg^{-j+1}/\fg^{-j + \ell} \subset \fg^{-j} / \fg^{-j + \ell}.$$ In other words, the following diagram commutes for each $y \in E$ and $ j \geq 2 $.
$$\begin{array}{ccc}
F^{-j+1}_y E & \stackrel{\theta_{-j+1}|_y}{\rightarrow} & \mathfrak{g}^{-j+1}/\fg^{-j+1 +\ell} \\ \cap & & \downarrow  \\
 F^{-j}_y E &\stackrel{\theta_{-j}|_y}{\rightarrow} & \mathfrak{g}^{-j}/\fg^{-j+ \ell}.
\end{array}$$
\item[(iii)]  Each $\theta_{-j}$ is $G^0/G^{\ell}$-equivariant, i.e. $$R_{b}^{*}\theta_{-j}={\rm Ad}(b^{-1})\circ\theta_{-j}$$ where $R_{b}$ denotes the principal right action of $b \in G^0/G^{\ell}$ and ${\rm Ad}$ denotes the natural action of $G^0/G^{\ell}$ on  $\mathfrak{g}^{-j}/\fg^{-j+ \ell}$ induced by the adjoint representation of $G$.
    \item[(iv)] For $A\in \fg^0/\fg^{\ell}$,
   the element $\theta_{-j}(\zeta_{A}) \in \fg^{-j}/\fg^{-j + \ell} $ is  the projection of $A$ in $$\fg^0/\fg^{-j+ \ell} \subset \fg^{-j}/\fg^{-j+ \ell}$$ if $- j + \ell > 0$ and is zero if $-j+ \ell \leq 0.$
\end{itemize}
Note that $\theta_{-j} =0$ if $j \geq \nu + \ell$ and the value of  $\theta_{-\nu-i}$ is
$\theta_{-\nu}$ modulo $\fg^{-\nu -i + \ell}$ for all positive integer $i$.
Thus $\theta$ is determined by $\theta_{-1}, \cdots, \theta_{-\nu}$.
\end{definition}

\begin{remark}\label{r.Cartan}
A $G_0$-structure $E_0 \to M$ subordinate to the filtration of type $\fg_{-}$ with the
soldering from $$(s_{-k}, k \in \N),\  s_{-k} \in \Hom(F^{-k}E_0, \fg_{-k})$$ is a $G^0$-frame bundle of length 1 by setting $\theta_{-k} = s_{-k}$. On the other hand, if $\ell \geq \mu + \nu +1$, then   a $G^0$-frame bundle $(E, \theta)$ of length $\ell$ is a $G^0$-principal bundle and $\Upsilon:= \theta_{-\nu} \in \Hom
(T E, \fg)$ makes it a Cartan connection of type $(G, G^0)$ on $M$. Thus $G^0$-frame bundles of length $\ell$, for
$2 \leq \ell \leq \mu+ \nu $, are intermediate geometric objects lying between a  $G_0$-structure and a Cartan connection. We will use them as auxiliary objects in an inductive construction for the proof of Theorem \ref{t.Cartan}.
\end{remark}

\begin{definition}\label{d.Theta}
Let $(E, \theta)$ be a $G^0$-frame bundle of length $\ell \geq 1$.
 Denote by $\Theta_{-k} \in \Hom(F^{-k}E, \fg_{-k})$ the partial form obtained from $$\theta_{-k} \in \Hom(F^{-k}E, \fg^{-k}/\fg^{-k + \ell})$$ by the quotient $$\fg^{-k}/\fg^{-k+ \ell} \to \fg^{-k}/\fg^{-k + 1} \ = \fg_{-k}.$$ We have  $${\rm d} \Theta_{-k}  \in \Hom(F^{-i}E \otimes F^{-j} E,   \fg_{-k})$$ for any nonnegative $i, j$ satisfying $i+j \leq k \geq 1$, as explained in Definition
 \ref{d.partialform}.  We say that $(E, \theta)$ {\em satisfies the structure equation} if for any $k \in \mathbb{N}$ and any two nonnegative integers $i$ and  $j$ satisfying $i+j \leq k$,
 $${\rm d} \Theta_{-k} + [\Theta_{-i}, \Theta_{-j}] = 0$$
as elements of $\Hom(F^{-i} E \otimes F^{-j} E, \fg_{-k})$.
\end{definition}

  \begin{definition}\label{d.torsion}
Let $(E, \theta)$ be a $G^0$-frame bundle of length $\ell \geq 1$. For each point $y \in E$, we define an element  $ \tau_y \in \Hom(\wedge^2 \fg_{-}, \fg)$ which sends $\fg_{-i} \wedge \fg_{-j}$ for each $i, j \in \N$  to $$\fg_{-i-j} + \fg_{-i-j+1} + \cdots + \fg_{-i-j+ \ell -1} \ \cong \ \fg^{-i-j}/\fg^{-i-j+ \ell} $$ as follows.
Recall that for any nonnegative integers $i,j,k$ satisfying $i+j \leq k \geq 1$, we have
$${\rm d} \theta_{-k} \in \Hom(F^{-i}E \otimes F^{-j}E, \fg^{-k}/\fg^{-k+\ell})$$
 from Definition \ref{d.partialform}.
Given $v \in \fg_{-i}$ and $w \in \fg_{-j}$, choose $\xi \in F^{-i}_y E$ and $\eta \in F^{-j}_y E$ such that $\theta_{-i}(\xi) =v$ and $\theta_{-j}(\eta) = w$. Define $$\tau_{y} (v, w):= {\rm d} \theta_{-i-j}(\xi, \eta) \in \fg^{-i-j}/\fg^{-i-j + \ell}.$$ It is easy to check that this does not depend on the choice of $\xi, \eta$ (see Section 3.9 of \cite{CS}). The assignment $y \mapsto \tau_y$ determines
 a  holomorphic function $\tau$ on $E$ with values in $\Hom(\wedge^2 \fg_{-}, \fg)$,  called  the {\em torsion} of $(E, \theta).$ We can write $$\tau_y = \tau_y^0 + \tau_y^1 + \cdots + \tau_y^{\ell-1}, \ \tau_y^k \in \Hom(\wedge^2 \fg_{-}, \fg)_k$$ to define a $\Hom(\wedge^2 \fg_{-}, \fg)_k$-valued holomorphic function $\tau^k, 0 \leq k \leq \ell-1$, on $E$.   Note that
$\tau_y^0 (v,w) = -[v, w]$ for all  $v, w \in \fg_{-}$ and $y \in E$  if $(E, \theta)$ satisfies the structure equation in the sense of Definition \ref{d.Theta}.
We will write $$\tau^{\geq 1} := \sum_{k=1}^{\ell -1} \tau^k.$$ \end{definition}

\subsection{From $(E, \theta)$ to $(\hat{E}, \hat{\theta})$}\label{ss.hat}

In this subsection, we will recall some definitions and results from \cite{CS}, with suitable modification. As the adaptation to our setting  is more-or-less straight-forward with only minor differences in notation, we will skip the proofs by giving precise references to the corresponding part of \cite{CS}.

\begin{definition}\label{d.hat}
Let $(E, \theta)$ be a $G^0$-frame bundle of length $\ell \geq 1$, satisfying the structure equation in the sense of Definition \ref{d.Theta}. For a point $y \in E$, let $$\hat{E}_y \subset \bigoplus_{j \in \N} \Hom(F^{-j}_y E, \fg^{-j}/\fg^{-j+ \ell +1})$$ be the subset consisting of $$\bigoplus_{j \in \N} \varphi_{-j},  \ \varphi_{-j} \in \Hom(F^{-j}_y E, \fg^{-j}/\fg^{-j + \ell +1}),$$  satisfying the following conditions.
\begin{itemize}
\item[(i)] $\varphi_{-j} \equiv \theta_{-j}|_y \mod \fg^{-j + \ell}.$
\item[(ii)] The restriction of $\varphi_{-j}$ to $F^{-j+1}_y E$ coincides with $\theta_{-j+1}|_y.$ In particular, it takes values in $\fg^{-j+1}/\fg^{-j+1 + \ell}$.
\item[(iii)] $\varphi_{-1}(\zeta_A) = A $ for all $A \in \fg^0/\fg^{\ell}.$
\end{itemize} Define $$\hat{E} := \bigcup_{y \in E} \hat{E}_y.$$ Denote by $\hat{\pi}: \hat{E} \to E$ the natural projection. \end{definition}

In the next two lemmata, we describe two different actions of Lie groups on $\hat{E}$.
Lemma \ref{l.affine} is proved in Proposition 3.13 of \cite{CS} and Lemma \ref{l.actionhat} is proved in Section 3.14 of \cite{CS}.

\begin{lemma}\label{l.affine} In Definition \ref{d.hat},  for $\varphi:= (\varphi_{-j}, j \in \N) \in \hat{E}_y $ and $$ \Hom(\fg_{-}, \fg)_{\ell} \ni \psi:= (\psi_{-j}, j \in \N), \psi_{-j} \in \Hom(\fg_{-j}, \fg_{-j+ \ell}),$$ define $$(\varphi \cdot \psi)_{-j} := \varphi_{-j} + \psi_{-j} \circ \Theta_{-j}|_y$$ where $\Theta_{-j}|_y \in \Hom( F^{-j}_y E, \fg_{-j})$ is as in Definition \ref{d.Theta}.
Then the collection $$\varphi \cdot \psi := \left( (\varphi \cdot \psi)_{-j}, j \in \N \right)$$ is an element of $\hat{E}$, hence the assignment $\varphi \mapsto \varphi \cdot \psi$ defines an action of the vector group $C^{\ell,1} (\fg_0) = \Hom(\fg_{-}, \fg)_{\ell}$ on
$\hat{E}.$ By this action,
the morphism $\hat{\pi}: \hat{E} \to E$  is a $C^{\ell,1}(\fg_0)$-principal bundle over $E$. \end{lemma}

\begin{lemma}\label{l.actionhat}
Let us use the notation of  Definition \ref{d.hat}.
For an element $b \in G^0/G^{\ell +1}$, let $b_0 \in G^0/G^{\ell}$ be its class modulo $G^{\ell}$. For an element  $$\varphi_{-j} \in \Hom( F^{-j}_y E, \fg^{-j}/\fg^{-j+ \ell +1}),$$
define $$\varphi_{-j} \cdot b \in \Hom(F^{-j}_{y \cdot b_o} E, \fg^{-j}/\fg^{-j+ \ell +1})$$
by $$\xi \in F^{-j}_{y\cdot b_0} E \ \mapsto \ \varphi_{-j} \cdot b (\xi) := {\rm Ad}(b^{-1}) \circ \varphi_{-j} \left( (R_{b_0^{-1}})_* \xi \right).$$ Then $(\varphi_{-j} \cdot b, j \in \N)$ is an element of $\hat{E}$, defining a free right action of $G^0/G^{\ell +1}$ on $\hat{E}$. Moreover, the morphism $\hat{\pi}: \hat{E} \to E$ is equivariant with respect to the projection $G^0/G^{\ell +1} \to G^0/G^{\ell}$.   \end{lemma}

\begin{definition}\label{d.partialhat}
In the setting of Definition \ref{d.hat}, define a filtration  $(F^k \hat{E}, k \in \Z)$ on $\hat{E}$ as follows:
\begin{itemize}
\item[(i)] for $k \geq \ell + 1$ the subbundle $F^k \hat{E}$ is zero; and
\item[(ii)] for each $k \leq \ell$, each $y \in E$ and $\varphi \in \hat{E}_y$, $$ F^k_{\varphi} \hat{E} := ({\rm d}_{\varphi} \hat{\pi})^{-1} \left( F^k_{y} E \right) $$
\end{itemize}
 For $\varphi = (\varphi_{-k}, k \in \N) \in \hat{E}_y$ and $\xi \in T^{-j}_{\varphi} \hat{E},$ and $j \in \N$, define
$$\hat{\theta}_{-j} (\xi) := \varphi_{-j} ( {\rm d}_{\varphi} \hat{\pi} (\xi)) \ \in \fg^{-j}/\fg^{-j + \ell +1}.$$ This gives
  a partial form on $\hat{E}$ $$ \hat{\theta}_{-j} \in \Hom( F^{-j} \hat{E}, \fg^{-j}/\fg^{-j + \ell +1}).$$
\end{definition}

The following is from Section 3.15 of \cite{CS}.

\begin{lemma}\label{l.hatRb}
In the setting of Definition \ref{d.partialhat}, the collection of partial forms $\hat{\theta} = (\hat{\theta}_{-j}, j \in \N)$ satisfies the following analogue of Definition \ref{d.length}.
 \begin{itemize}
\item[(i)] For each $\varphi \in \hat{E}$ and $j \in \N$, the kernel of  $$\hat{\theta}_{-j}|_{\varphi} \in \Hom(F^{-j}_{\varphi} \hat{E}, \fg^{-j}/\fg^{-j+\ell+1}) $$ is
$F^{-j + \ell + 1}_{\varphi} \hat{E}.$
\item[(ii)] When $j \geq 2$, the restriction of $$\hat{\theta}_{-j} \in \Hom(F^{-j} \hat{E}, \fg^{-j}/\fg^{-j+ \ell+1})$$ to the subbundle $F^{-j+1}\hat{E} \subset F^{-j} \hat{E}$
is equal to the composition of $$\hat{\theta}_{-j+1} \in \Hom(F^{-j+1}\hat{E}, \fg^{-j+1}/\fg^{-j +2 + \ell})$$ with the natural homomorphism $$\fg^{-j+1}/\fg^{-j+2 + \ell} \to \fg^{-j+1}/\fg^{-j +1 + \ell} \subset \fg^{-j} / \fg^{-j + 1+ \ell}.$$
\item[(iii)]   $$R_{b}^{*}\hat{\theta}_{-j}={\rm Ad}(b^{-1})\circ \hat{\theta}_{-j}$$ where $R_{b}$ denotes the principal right action of $b \in G^0/G^{\ell+1}$ in Lemma \ref{l.actionhat} and ${\rm Ad}$ denotes the natural action of $G^0/G^{\ell+1}$ on  $\mathfrak{g}^{-j}/\fg^{-j+ \ell+1}$ induced by the adjoint representation of $G$.
    \item[(iv)] For $A\in \fg^0/\fg^{\ell+1}$,
   the element $\hat{\theta}_{-j}(\zeta_{A}) \in \fg^{-j}/\fg^{-j + \ell +1} $ is  the projection of $A$ in $$\fg^0/\fg^{-j+ \ell+1} \subset \fg^{-j}/\fg^{-j+ \ell +1}$$ if $- j + \ell +1 > 0$ and is zero if $-j+ \ell +1 \leq 0.$
\end{itemize}
\end{lemma}

We have the following analog of Definition \ref{d.torsion} on $\hat{E}$.

\begin{definition}\label{d.torsionhat} In the setting of Definition \ref{d.partialhat},
for each point $\varphi \in \hat{E}_y$, we define an element  $ \tau_{\varphi} \in \Hom(\wedge^2 \fg_{-}, \fg)$ which sends $\fg_{-i} \wedge \fg_{-j}$ for each $i, j \in \N$  to $$\fg_{-i-j} + \fg_{-i-j+1} + \cdots + \fg_{-i-j+ \ell} \ \cong \ \fg^{-i-j}/\fg^{-i-j+ \ell+1} $$ as follows.
 Choose a neighborhood $U \subset E$ of $y \in E$ and a local holomorphic section $\sigma: U \to \hat{E}$ satisfying $\sigma (y) = \varphi$.  Then $\sigma^* \hat{\theta}_{-k}, k \in \N,$ is a partial form on $U$ and for any $i, j \in \mathbb{N}$ satisfying $i +j \leq k$, we have $${\rm d} \sigma^* \hat{\theta}_{-k} \ \in \ \Hom(F^{-i} E \otimes F^{-j} E, \fg^{-k}/\fg^{-k + \ell +1}).$$
 For   $v \in \fg_{-i}$ and $w \in \fg_{-j},$ choose $\xi \in F_y^{-i} E$ and $\eta \in F_y^{-j} E$ such that $\varphi_{-i}(\xi) =v$ and $\varphi_{-j}(\eta) = w$.
Define $$\hat{\tau}_{\varphi}(v, w) := {\rm d} \sigma^*\hat{\theta}_{-i-j}(\xi, \eta)  \ \in \ \fg^{-i-j}/\fg^{-i-j+ \ell + 1}.$$ This does not depend on the choice of $\xi, \eta$ and $\sigma$ (Section 3.17 of \cite{CS}). The assignment $\varphi \mapsto \hat{\tau}_{\varphi}$ determines
 a  $\Hom(\wedge^2 \fg_{-}, \fg)$-valued holomorphic function $\hat{\tau}$ on $\hat{E}$,  called  the {\em torsion} of $(\hat{E}, \hat{\theta}).$ We can write $$\hat{\tau}  = \hat{\tau}^0 + \hat{\tau}^1 + \cdots + \hat{\tau}^{\ell}$$ for a $\Hom(\wedge^2 \fg_{-}, \fg)_k$-valued holomorphic function $\hat{\tau}^k$ on $\hat{E}$ for each $0 \leq k \leq \ell$.
 Also define $$\hat{\tau}^{\geq 1} := \hat{\tau}^1 + \hat{\tau}^2 + \cdots + \hat{\tau}^{\ell}.$$   Definition \ref{d.torsion} and Definition \ref{d.hat} (i) imply  that $\hat{\tau}_{\varphi}^k = \tau^k_{y}$ for any $\varphi \in \hat{E}_y$ and  $0 \leq k \leq \ell -1.$
\end{definition}

   Next lemma, proved in Section 3.18 of \cite{CS}, describes the behavior of the torsion function under the group action of Lemma \ref{l.affine}.

\begin{lemma}\label{l.3.18}
In Definition \ref{d.torsionhat}, pick $\varphi, \tilde{\varphi} \in \hat{E}_y$ and $\psi \in C^{\ell,1}(\fg_0)$ such that $$\tilde{\varphi}_{-j}(\xi)  = \varphi_{-j} (\xi) + \psi (
\Theta_{-j}(\xi)) $$ for all $j \in \N$ and $\xi \in T^{-j}_y E$. Then in terms of $$\partial: C^{\ell, 1}(\fg_0) = \Hom(\fg_-, \fg)_{\ell} \ \to \  C^{\ell, 2} (\fg_0) = \Hom(\wedge^2 \fg_{-}, \fg)_{\ell}$$ in
Definition \ref{d.cohomology}, we have
$\hat{\tau}_{\tilde{\varphi}} = \hat{\tau}_{\varphi} + \partial \psi,$ i.e., for $v, w \in \fg_{-}$,
$$\tau_{\tilde{\varphi}} (v,  w) = \tau_{\varphi} (v, w) + [\psi(v), w] + [v, \psi(w)] - \psi ([v,w]).$$\end{lemma}

The behavior of torsion function under the group action of Lemma \ref{l.actionhat} is described  in the next lemma. Its proof is given in Section 3.19 of \cite{CS}: the formula below is a translation into our notation of the last displayed equation of that section,  our $\hat{\tau}^{\geq 1}_{\varphi \cdot b}(v,w)$ corresponding to ${\rm d} \bar{\sigma}^*\hat{\theta}_{i+j}(\xi, \eta) + [v,w]$ in \cite{CS}.

\begin{lemma}\label{l.3.19}
In Definition \ref{d.torsionhat}, pick $\varphi \in \hat{E}_y$ and
 $b\in G^0/G^{\ell+1}$. Then for any $v \in \fg_{-i}$ and $ w \in \fg_{-j}$, we have
$$
\hat{\tau}^{\geq 1}_{\varphi\cdot b} (v, w) = {\rm Ad}(b^{-1}) \cdot \hat{\tau}^{\geq 1}_{\varphi} ({\rm Ad}_{-}(b) \cdot v, {\rm Ad}_{-}(b) \cdot w) \ \in \ \fg^{-i-j}/\fg^{-i-j + \ell +1},
$$
where ${\rm Ad}(b^{-1})$ denotes the  action of $b^{-1} \in G^0/G^{\ell +1 }$ on $\fg^{-i-j}/\fg^{-i-j + \ell +1}$ induced by the adjoint representation and ${\rm Ad}_{-}(b)$ denotes the action of $b$ on $\mathfrak{g}_{-}$ defined by the following composition:
$$ \mathfrak{g}_{-k} \stackrel{{\rm Ad}(b)}{\rightarrow}  \mathfrak{g}_{-k} + \cdots + \mathfrak{g}_{-k + \ell} \stackrel{{\rm projection}}{\rightarrow} \mathfrak{g}_{-k} + \cdots + \mathfrak{g}_{\min\{-k + \ell, -1\}}.$$ \end{lemma}

\subsection{From $(\hat{E}, \hat{\theta})$ to $(\widetilde{E}, \widetilde{\theta})$}\label{ss.widetilde}

From now on, we proceed differently from \cite{CS}.

\begin{definition}\label{d.gamma}
In the setting of Definition \ref{d.torsionhat},
 Lemma \ref{l.3.18} implies that the torsion function $\hat{\tau}$ on $\hat{E}$ induces a $ C^{\ell, 2}(\fg_0)/\partial C^{\ell, 1}(\fg_0)$-valued holomorphic function on $E$ whose value at $y \in E$ is $$ \gamma (y) := \hat{\tau}^{\ell}_{\varphi} \mod \partial C^{\ell, 1}(\fg_0) \mbox{ for any } \varphi \in \hat{E}_y.$$
\end{definition}

\begin{lemma}\label{l.equivariance}
In the setting of  Definition \ref{d.torsionhat}, assume that $\tau^{\geq 1} = \sum_{k=1}^{\ell -1} \tau^k =0.$ Write $E_0 \to M$ for the $G_0$-principal bundle obtained by the quotient of $E \to M$ via the $G^1/G^{\ell}$-action. Then the function $\gamma$ on $E$ is constant along the fibers of $E \to E_0$ and descends to a $G_0$-equivariant holomorphic map $ E_0 \to C^{\ell, 2}(\fg_0)/\partial C^{\ell, 1}(\fg_0)$,  inducing an element $$\bar{\gamma} \in H^0(M, C^{\ell,2}(E_0)/\partial C^{\ell, 1}(E_0)),$$
a holomorphic section of the vector bundle on $M$ associated to  $E_0$ via the natural $G_0$-representation on $C^{\ell, 2}(\fg_0)/\partial C^{\ell, 1}(\fg_0).$ \end{lemma}

\begin{proof}
Let us examine the behavior of $\gamma$ under the action of $G^0/G^{\ell}$.
Choose an element $b \in G^0/G^{\ell +1}$ and pick any point $\varphi \in \hat{E}_y, y \in E$. By the assumption $\tau^{\geq 1}=0$, we have $\hat{\tau}^k =0$ for $1 \leq k \leq \ell-1$ and  $\hat{\tau}^{\geq 1} = \hat{\tau}^{\ell}$. Thus for any $v \in \fg_{-i}$ and $w \in \fg_{-j},$  Lemma \ref{l.3.19} reads
\begin{eqnarray}\label{e.1} \hat{\tau}^{\ell}_{\varphi \cdot b} (v, w) &=& {\rm Ad}(b^{-1}) \cdot \hat{\tau}^{\ell}_{\varphi} ({\rm Ad}_{-}(b) \cdot v, {\rm Ad}_-(b) \cdot w) \ \in \fg_{-i-j + \ell}. \end{eqnarray}

Suppose that $b \in G^1/G^{\ell +1}$. Then
$$v - {\rm Ad}_-(b) \cdot v \in \fg^{-i+1}/\fg^0 \mbox{ and } w - {\rm Ad}_-(b) \cdot w \in \fg^{-j+1}/\fg^0.$$ Since  $$\hat{\tau}^{\ell}_{\varphi} (\fg^{-i+1}, \fg^{-j}) = 0 = \hat{\tau}^{\ell}_{\varphi}(\fg^{-i}, \fg^{-j+1}) $$ as elements of $\fg_{-i-j+ \ell}$ and $\hat{\tau}^{\ell}_{\varphi} \in \Hom(\wedge^2 \fg_-, \fg)_{\ell}$, we see that $$\hat{\tau}^{\ell}_{\varphi}({\rm Ad}_-(b) \cdot v, {\rm Ad}_-(b) \cdot w) = \hat{\tau}^{\ell}_{\varphi} (v,w) \ \in \fg_{-i-j + \ell}.$$ Then $b^{-1} \in G^1/G^{\ell +1}$ implies
$${\rm Ad}(b^{-1}) \cdot \hat{\tau}^{\ell}_{\varphi}(v,w)  = \hat{\tau}^{\ell}_{\varphi}(v,w).$$
Thus (\ref{e.1}) shows $$\hat{\tau}^{\ell}_{\varphi \cdot b} (v, w) = \hat{\tau}^{\ell}_{\varphi} (v,w).$$
This proves that $\gamma$ is constant along the fibers of $E \to E_0$.

\medskip
 Suppose $b $ is an element in the image of $G_0$ (from Proposition \ref{p.algebraic} (5))  in $G^0/G^{\ell +1}$. Then ${\rm Ad}(b)$ preserves the grading of $\fg$ and  (\ref{e.1}) becomes
 \begin{eqnarray}\label{e.2} \hat{\tau}^{\ell}_{\varphi \cdot b} (v, w) &=& {\rm Ad}(b^{-1}) \cdot \hat{\tau}^{\ell}_{\varphi} ({\rm Ad}(b) \cdot v, {\rm Ad}(b) \cdot w). \end{eqnarray} Recall that the natural action of $b^{-1} \in G_0$ on
$C^{\ell, 2}(\fg_0) = \Hom (\wedge^2 \fg_0, \fg)_{\ell}$ sends
an element $\alpha \in C^{\ell, 2}(\fg_0)$ to $b^{-1} \cdot \alpha$ defined by $$
(b^{-1} \cdot \alpha) (v, w) := {\rm Ad}(b^{-1}) \cdot \alpha( {\rm Ad}(b) v, {\rm Ad}(b) w).$$
Thus the equation (\ref{e.2}) yields $$\gamma (y \cdot b_0) = b_0^{-1} \cdot \gamma(y)$$
where $b_0 \in G^0/G^{\ell}$ is the image of $b$ under the projection $G^0/G^{\ell +1} \to G^0/G^{\ell}$. This proves the $G_0$-equivariance. Consequently, we obtain the element $\bar{\gamma}$. \end{proof}

\begin{lemma}\label{l.surjective}
In the setting of Lemma \ref{l.equivariance}, assume that $\bar{\gamma} =0$. Then for any $y \in E$, the restriction of $\hat{\tau}^{\ell}$ induces a $G_0$-equivariant surjective holomorphic map $\hat{\tau}^{\ell}|_y: \hat{E}_y \to \partial C^{\ell, 1} (\fg_0).$
\end{lemma}

\begin{proof}
By the definition of $\bar{\gamma}$  in Lemma \ref{l.equivariance}, the assumption $\bar{\gamma} =0$ implies $\gamma =0$.
Thus $\hat{\tau}^{\ell}$ induces a $G_0$-equivariant holomorphic map
$\tau^{\ell}|_y: \hat{E}_y \to  \partial C^{\ell, 1} (\fg_0).$ By Lemma \ref{l.affine} and Lemma \ref{l.3.18}, this map is surjective. \end{proof}

\begin{definition}\label{d.widetilde}
Let $\pi: E \to M$ and $\theta= (\theta_{-j}, j \in \N)$ be a $G^0$-frame bundle of length $\ell \geq 1$. Assume that \begin{itemize} \item[(a)]  $\theta$ satisfies the structure equation;
\item[(b)] the torsion function $\tau$ of $(E, \theta)$ satisfies  $\tau^{\geq 1}=0$; and  \item[(c)] the section $\bar{\gamma} \in H^0(M, C^{\ell,2}(E_0)/\partial C^{\ell,1}(E_0))$ defined in Lemma \ref{l.equivariance} is zero.\end{itemize}  By Lemma \ref{l.surjective},   the  holomorphic map $\hat{\tau}^{\ell}: \hat{E} \to \partial C^{\ell, 1}(\fg_0)$ is surjective. We define the following.
\begin{itemize}
\item[(1)] $\widetilde{E} := \{ \varphi \in \hat{E}, \hat{\tau}^{\ell}_{\varphi} = 0 \}$.
\item[(2)] $\widetilde{\pi} := \hat{\pi}|_{\widetilde{E}} : \widetilde{E} \to E.$
\item[(3)] A subbundle $F^{j} \widetilde{E} \subset T\widetilde{E}$ whose fiber $F^{j}_{\varphi} \widetilde{E}$ at $\varphi \in \widetilde{E}$ is $$({\rm d}_{\varphi} \widetilde{\pi})^{-1} ( F^{j} E) \mbox{ if } j \leq \ell,$$ and 0 if $j \geq \ell + 1.$
\item[(4)] A partial form $\widetilde{\theta}_{-j} \in \Hom(F^{-j} \widetilde{E}, \fg^{-j}/\fg^{-j+\ell +1})$ for each $j \in \N$ such that
$$   \widetilde{\theta}_{-j}(\xi) :=   \varphi_{-j}( {\rm d} \widetilde{\pi} (\xi))$$
     for each $\varphi \in \widetilde{E}$ and $\xi \in F^{-j}_{\varphi} \widetilde{E}$.  \end{itemize} \end{definition}

    \begin{proposition}\label{p.widetilde}
    In the setting of Definition \ref{d.widetilde}, we have the following.
    \begin{itemize}
    \item[(1)] The subset $\widetilde{E}$ of $\hat{E}$ is stable under the action of $G^0/G^{\ell+1}$.
\item[(2)] The action of $G^0/G^{\ell+1}$ on $\widetilde{E}$ is free and the morphism $\widetilde{\pi}: \widetilde{E}\rightarrow E$ is equivariant with respect to the natural projection $G^0/G^{\ell+1}\rightarrow G^0/G^{\ell}$. \item[(3)] The composition $\pi \circ \widetilde{\pi}: \widetilde{E} \to M$ is a principal $G^0/G^{\ell+1}$-bundle over $M.$
\item[(4)] The principal bundle $\widetilde{E}\to M$ with the partial forms $\widetilde{\theta}$ is a $G^0$-frame bundle of length $\ell+1$ satisfying the structure equation.
\item[(5)] The torsion $\widetilde{\tau}$ of $(\widetilde{E}, \widetilde{\theta})$ satisfies $\widetilde{\tau}^{\geq 1}=0$. \end{itemize} \end{proposition}

\begin{proof}
The formula (\ref{e.1}) in the proof of Lemma \ref{l.equivariance} shows that $\widetilde{E}$ is stable under the action of $G^0 / G^{\ell+1}$, proving  $(1)$.

Lemma \ref{l.actionhat} says that the action of $G^0 / G^{\ell+1}$ on $\hat{E}$ is free, the morphism $\hat{\pi} : \hat{E} \to E$ is equivariant and the subgroup $G^{\ell} / G^{\ell+1}$ of $G^0 / G^{\ell+1}$ preserves the fibres of $\hat{\pi}$.
Thus the action of $G^0 / G^{\ell+1}$ on $\widetilde{E}$ is free and the morphism $\widetilde{\pi}$ is equivariant  with respect to $G^0/G^{\ell+1}\rightarrow G^0/G^{\ell}$. This proves (2).

By Lemma \ref{l.affine} and Lemma \ref{l.surjective},
               each fiber of $\widetilde{\pi}: \widetilde{E} \rightarrow E$ is biregular to the affine space $$ {\rm Ker} \left(\partial: C^{\ell, 1}(\fg_0) \to C^{\ell, 2}(\fg_0) \right).$$ By Lemma \ref{l.cohom}, this is isomorphic to $\fg_{\ell}$. We know $G^{\ell}/G^{\ell +1}$ acts freely on each fiber of $\widetilde{\pi}$ and the action is equivalent to its adjoint action on $\fg_{\ell}$. Since the latter is a transitive action, we see that $G^{\ell}/ G^{\ell+1}$ acts transitively on each fiber of $\widetilde{\pi}$ and $\widetilde{\pi} : \widetilde{E} \rightarrow E$ is a principal $G^{\ell} / G^{\ell+1}$-bundle. Thus $\pi: E \rightarrow M$ is a principle $G^0/ G^{\ell}$-bundle and $\widetilde{\pi} : \widetilde{E} \to E$ is equivariant. We conclude that $\widetilde{E} \rightarrow M$ is a principal $G^0/ G^{\ell+1}$-bundle, proving (3).

By Lemma \ref{l.hatRb} and the definition of $\widetilde{\theta}$, we know that $\widetilde{\theta}$ is a frame form of length $\ell+1$ on $\tilde{E}$. Let $$\Theta_{-k} \in \Hom(F^{-k} E, \fg_{-k}) \left( \mbox{ resp. } \widetilde{\Theta}_{-k}
\in \Hom(F^{-k} \widetilde{E}, \fg_{-k}) \right)$$ be the partial form obtained from $\theta_{-k}$ (resp.  $\widetilde{\theta}_{-k}$) as in Definition \ref{d.Theta}. It is clear from the definition of $\widetilde{\theta}_{-k}$ that
$$\widetilde{\Theta}_{-k} = \Theta_{-k} \circ {\rm d} \widetilde{\pi}|_{F^{-k}\widetilde{E}}.$$
Since $\theta$  satisfies the structure equation  by the condition (a) of Definition \ref{d.widetilde}, so does $\widetilde{\theta}$.  This proves (4).

 Using Definition \ref{d.torsionhat}, we can see that $\widetilde{\tau}^{k}_{\varphi} = \hat{\tau}^{k}_{\varphi}$ for any $\varphi \in \widetilde{E}$ and $0 \leq k \leq \ell$.
By the definition of $\widetilde{E}$, we have $\hat{\tau}^{\ell}_{\varphi} =0$ if $\varphi \in \widetilde{E}.$  By the condition (b) of Definition \ref{d.widetilde},
we have $\hat{\tau}^{k}_{\varphi} = \tau^k_{\hat{\pi}(\varphi)} =0$ for $1 \leq k \leq \ell -1$. It follows that
$\widetilde{\tau}^k =0$ for $1 \leq k \leq \ell$, proving (5).
\end{proof}

\subsection{Inductive proof of Theorem \ref{t.Cartan}}\label{ProofCartan}

    To prove Theorem \ref{t.Cartan}, let us denote by $\pi_0: E_0 \to M$ the $G_0$-structure and
    set $\theta^0_{-k} =s_{-k}, k \in \N$, the soldering form on $E_0$.
    Then $(E_0, \theta^0)$ is a $G^0$-frame bundle of length 1, as mentioned in Remark \ref{r.Cartan}.

We claim that $(E_0, \theta^0)$ satisfies the conditions (a), (b), (c) of Definition \ref{d.widetilde}. (b) is automatic from $\ell =1$ and (c) is immediate by the assumption of Theorem \ref{t.Cartan}.
We claim that (a) follows from the fact that the filtration on $M$ is of type $\fg_{-}.$
In fact, given $\xi \in F^{-i}_y E_0$ and $\eta \in F^{-j}_y E_0$ for a point $y \in E_0$ and  $i, j \in \N$, choose a neighborhood $U$ of $y$ and  sections $\vec{\xi}$ of $F^{-i} E_0|_{U}$ extending $\xi$ and $\vec{\eta}$ of $F^{-j} E_0|_U$ extending $\eta$. Since $s_{-k}(F^{-i} E_0) = 0 = s_{-k}(F^{-j} E_0)$ for $k=i+j$, we have
$${\rm d} s_{-k} (\vec{\xi}, \vec{\eta}) = \vec{\xi} \left( s_{-k}(\vec{\eta}) \right) -\vec{\eta}\left( s_{-k}(\vec{\xi}) \right) - s_{-k}([\vec{\xi}, \vec{\eta}]) = - s_{-k}([\vec{\xi}, \vec{\eta}]).$$
But $s_{-k}([\vec{\xi}, \vec{\eta}])_y = [s_{-i}(\xi), s_{-j}(\eta)]$ because
$${\rm d} \pi_0 ([\vec{\xi}, \vec{\eta}])_y = [{\rm d}\pi_0 (\xi), {\rm d} \pi_0 (\eta)]$$ in the symbol ${\rm gr}_x(M) \cong \fg_{-}.$ Thus
 $${\rm d} s_{-k}(\xi, \eta) = {\rm d} s_{-k} (\vec{\xi}, \vec{\eta})_y = -[s_{-i}(\xi), s_{-j}(\eta)]$$
 which is exactly the structure equation.

    Now we apply Proposition \ref{p.widetilde} and set
    $$E_1 := \widetilde{E_0}, \ \pi_1 := \pi_0 \circ \widetilde{\pi_0}, \ \theta^1:= \widetilde{\theta^0}$$ to obtain a $G^0$-frame bundle
    $(\pi_1: E_1 \to M, \theta^1)$ of length $2$.  Then  $(E_1, \theta^1)$ satisfies the conditions (a) and (b) of Definition \ref{d.widetilde} by Proposition \ref{p.widetilde}. Moreover, the quotient of $E_1$ by $G^1/G^2$ is just $E_0$.
    Thus the condition (c) of Definition \ref{d.widetilde} follows from the assumption of
    Theorem \ref{t.Cartan}.

    Repeating the same argument, we can construct inductively a $G^0$-frame bundle $(\pi_k: E_k \to M, \theta^k)$ of length $\ell  = k +1$ for any $1 \leq k \leq \nu + \mu$ such that the conditions (a) and (b) are satisfied, and $E_0$ is the quotient of $E_k$ by $G^1/G^{k+1}$, while the condition (c) is satisfied for $0 \leq k \leq \nu + \mu - 1$ by assumption of Theorem \ref{t.Cartan}. For $k= \nu + \mu$, this gives (as mentioned in Remark
     \ref{r.Cartan} ) a Cartan connection $(E = E_{\nu + \mu}, \Upsilon = \theta^{\nu+ \mu}_{-\nu})$. From the construction, the soldering form of $E_0$ agrees with the quotient of $\Upsilon$. This finishes the proof of Theorem \ref{t.Cartan}.

\subsection{Evaluation of $\bar{\gamma}$ in a special setting}\label{ss.special}

Even when the vanishing condition in Theorem \ref{t.Cartan} is not satisfied,  we can construct a Cartan connection if we can show that the section $\bar{\gamma}$ in Lemma \ref{l.equivariance} vanishes at each stage of the inductive process in our proof of Theorem \ref{t.Cartan}. In this subsection, we examine the following special setting where this can be checked effectively.

\begin{notation}\label{n.Gst}
In Definition \ref{d.G_0-structure}, assume that $\nu =1$ and $\fg_-= \fg_{-1}$  is an abelian Lie algebra. Then ${\rm grFr}(M)$ is the ordinary frame bundle ${\rm Fr}(M)$. Let $G_0\subset {\rm GL}(\fg_{-1})$ be a connected algebraic subgroup satisfying the conditions of
 Proposition \ref{p.algebraic}. A $G_0$-principal subbundle $E_0 \subset {\rm Fr}(M)$ is
 a $G_0$-structure on $M$ in the usual sense (e.g. Definition 2.1 in Chapter VII of \cite{St}).  Assume that there exists a $G_0$-invariant subspace $D \subset \fg_{-1}$.
 Then we have the following two homomorphisms $\varpi$ and $\varrho$.
 \begin{itemize} \item[(a)] The associated vector bundle $E_0 \times^{G_0} D$ determines a distribution $\sD \subset TM.$ Denote by $\varpi \in \Hom(\wedge^2 \sD, TM/\sD)$ the Frobenius bracket tensor of the distribution $\sD$. \item[(b)]
Consider the $G_0$-homomorphism obtained by the composition
$$C^{1,2}(\fg_0) = \Hom(\wedge^2 \fg_{-1}, \fg_{-1})  \stackrel{\rho}{\to} \Hom(\wedge^2 D, \fg_{-1}) \stackrel{\chi}{\to} \Hom(\wedge^2 D, \fg_{-1}/D)$$ where $\rho$ is the restriction to the $G_0$-submodule
$D \subset \fg_{-1}$ and $\chi$ is the $G_0$-module quotient $\fg_{-1} \to \fg_{-1}/D$.
It is easy to see that $\partial(C^{1,1}(\fg_0))$ is annihilated by $\chi \circ \rho$.
 The resulting homomorphism
$$ C^{1,2}(\fg_0)/\partial(C^{1,1}(\fg_0)) \to \Hom(\wedge^2 D, \fg_{-1}/D))$$ induces a homomorphism
$$\varrho: H^0(M, C^{1,2}(E_0)/\partial(C^{1,1}(E_0))) \to  \Hom(\wedge^2 \sD, TM/\sD).$$
\end{itemize}
\end{notation}

\begin{proposition}\label{p.Gst}
In the setting of Notation \ref{n.Gst}, let $\theta^0$ be the soldering form on $E_0$.
As $(E_0, \theta^0)$ is a $G^0$-frame bundle of length 1, we have the section
$\bar{\gamma} \in H^0(M, C^{1,2}(E_0)/\partial(C^{1,1}(E_0)))$ defined in Lemma \ref{l.equivariance}
with $E=E_0$. Then $\varrho(\bar{\gamma}) = -\varpi$ in $\Hom(\wedge^2 \sD, TM/\sD)$. \end{proposition}

\begin{proof}
Since the problem is local, we will replace $M$ by the germ of a point $x \in M$.

From Definition \ref{d.torsionhat} applied to $E=E_0,$ if we choose a section $\sigma_0: M \to E$ and a section $\sigma: \sigma_0(M) \to \hat{E},$ then for any $\xi, \eta \in \sD_x$,
 \begin{eqnarray}\label{e.varrho} \varrho(\bar{\gamma}) (\xi, \eta) &=& \sigma_0(x)_D \circ \chi \circ \left( (\sigma_0^* {\rm d} \sigma^*
 \hat{\theta}_{-2}) (\xi, \eta) \right) \end{eqnarray} where $\sigma_0(x)_D$ denotes the
  isomorphism $\fg_{-1}/D \to T_x M/\sD_x$ induced by  $$\sigma_0(x) \in E_x = {\rm Isom}(\fg_{-1}, T_x M), \ \sigma_0(x) (D) = \sD_x.$$

On the other hand,  if we choose a ($\fg_{-1}/D$)-valued 1-form
$\vartheta$ on $M$ such that ${\rm Ker}(\vartheta) = \sD$, then for any $\xi, \eta \in \sD_x$, \begin{eqnarray}\label{e.varpi}  \varpi (\xi, \eta) &=& - \vartheta_x^{-1} \circ {\rm d} \vartheta (\xi, \eta), \end{eqnarray}
where $\vartheta_x^{-1}: \fg_{-1}/D \to T_x M/\sD_x$ is the isomorphism induced by $\vartheta$. Since (\ref{e.varpi}) does not depend on the choice of $\vartheta$, we may choose $\vartheta$ as $\chi \circ \sigma_0^* s_0$ for the soldering form $s_0$ on $E_0=E$.
Then \begin{eqnarray*} {\rm d} \vartheta (\xi, \eta) &=& \chi \circ (\sigma_0^* {\rm d} s_0) (\xi, \eta) \\ &=& \chi \circ (\sigma_0^* {\rm d} \sigma^* \hat{\theta}_{-2}) (\xi, \eta). \end{eqnarray*} By our choice of $\vartheta$, the isomorphism $\vartheta_x^{-1}$ agrees with $\sigma_0(x)_D$ in (\ref{e.varrho}). Thus the image of the above equation
under $\vartheta_x^{-1}$ agrees with (\ref{e.varrho}). Combining it with (\ref{e.varpi}) proves the proposition. \end{proof}

We have the following modification of Theorem \ref{t.Cartan}.

\begin{theorem}\label{t.Cartan'}
In the setting of Notation \ref{n.Gst}, assume that \begin{itemize} \item[(i)] the Frobenius bracket tensor $\varpi$ vanishes, i.e., the distribution $\sD$ is integrable;
\item[(ii)] the homomorphism $\varrho$ is injective; and
\item[(iii)]  $H^0(M, C^{\ell,2}(E_0)/\partial(C^{\ell,1}(E_0))) =0 $ for all $2 \leq \ell \leq \mu + 1$. \end{itemize}
 Then there exists a Cartan connection $(E \to M, \Upsilon)$ of type $(G, G^0)$ on $M$ related to $E_0 \subset {\rm Fr}(M)$ in the way described in Theorem \ref{t.Cartan}.
 \end{theorem}

\begin{proof}
By the assumptions (i) and (ii),  we see that $\bar{\gamma}$ in Proposition \ref{p.Gst} vanishes.
Then we can proceed just as in the proof of Theorem \ref{t.Cartan}.
Starting from $E_0 \to M$ and the soldering form $\theta_0$, the vanishing of $\bar{\gamma}$ implies that we can construct a $G^0$-frame bundle $E_1 = \widetilde{E}_0$ of length $2$ as described in Proposition \ref{p.widetilde}. Starting from $E_1$, we can proceed just as in the proof of Theorem \ref{t.Cartan} because of the vanishing assumption (iii).  Thus we can construct inductively a $G^0$-frame bundle of length $\ell$ for all $2 \leq \ell \leq \mu + 2$. This gives a Cartan connection with the desired properties.
\end{proof}

\section{Geometry of presymplectic Grassmannians}\label{s.pre}

\begin{definition}\label{d.spgroup}
A vector space $V$ with a marked nonzero element $\omega \in \wedge^2 V^*$ is called a {\em presymplectic vector space}.
\begin{itemize}
\item[(1)] The {\em presymplectic group} is
$$\Sp_{\omega}(V) := \{ g \in \GL (V), \ \omega( g \cdot v, g \cdot w) = \omega(v, w) \mbox{ for all } v, w \in V\}.$$
\item[(2)] The {\em null space} of $\omega$ is
$${\rm Null}_{\omega} := \{ w \in V, \ \omega(w, v) =0 \mbox{ for all } v \in V\}.$$
The null space is preserved by the action of $\Sp_{\omega}(V)$.
We will write $\no$ for the dimension of ${\rm Null}_{\omega}$.
\item[(3)] The form $\omega$ is symplectic   if $\no =0$ and odd-symplectic  if $\no=1.$ In these two cases, we say that $\omega$ is {\em maximally nondegenerate}.
    \item[(4)] The quotient $V/{\rm Null}_{\omega}$ will be denoted by $V^{\flat}$.
It is equipped with the symplectic form $\omega^{\flat}$ induced by $\omega$.  For $v \in V$, its class in $V^{\flat}$ will be denoted by $v^{\flat}$. For a subspace $W \subset V$, its image in $V^{\flat}$ will be denoted by $W^{\flat}$.
\end{itemize}
\end{definition}
The following is straight-forward.

\begin{lemma}\label{l.Sp}
Choose a subspace $V^{\sharp} \subset V$ complementary to ${\rm Null}_{\omega}$. The restriction $\omega^{\sharp}= \omega|_{V^{\sharp}}$ is symplectic and $(V^{\sharp}, \omega^{\sharp})$ is naturally isomorphic to $(V^{\flat}, \omega^{\flat})$. It induces  natural isomorphisms
\begin{eqnarray*}
\Sp_{\omega}(V) &\cong & \Hom(V^{\sharp}, {\rm Null}_{\omega})
\sd (\Sp_{\omega^{\sharp}}(V^{\sharp})\times \GL({\rm Null}_{\omega})) \\
& \cong & \Hom (V^{\flat}, {\rm Null}_{\omega}) \sd (\Sp_{\omega^{\flat}}(V^{\flat})\times \GL({\rm Null}_{\omega})) \end{eqnarray*} where $\varphi \in \Hom (V^{\flat}, {\rm Null}_{\omega})$ acts on $V$ by
$$ v \in V  \ \mapsto \ v + \varphi ( v^{\flat}).$$
In particular, the algebraic group $\Sp_{\omega}(V)$ is connected with center $\{\pm 1\}$.
\end{lemma}

\begin{definition}\label{d.spgrass}
Let $(V, \omega)$  be a presymplectic vector space.
\begin{itemize}
\item[(1)] For a natural number $m$, let  $\Gr(m, V)$ be the Grassmannian of subspaces of dimension $m$ in $V$. The subvariety of $\Gr(m, V)$ defined by $$ \Go(m, V) := \{ [W] \in \Gr(m, V), \ \omega (W, W) =0 \} $$ is called a {\em presymplectic Grassmannian}.  If $\no=0$, it is the classical symplectic Grassmannian. If $\no =1$, it is the odd-symplectic Grassmannian studied in
    \cite{Mi}.  \item[(2)] For a nonnegative integer $k$, define the $k$-th {\em stratum}
$$\Go(m, V; k)  :=  \{ [W] \in \Go(m, V), \ \dim (W \cap {\rm Null}_{\omega}) = k \}.$$
    This is empty unless  $k \leq \min \{ m, \no\}$. We have a decomposition into disjoint strata $$\Go(m, V) = \bigcup_{k=0}^{\min \{m, \no\}} \Go(m, V; k).$$
     \end{itemize}
\end{definition}

\begin{assumption}\label{a.no} For the rest of this section, we assume that $$4 \leq 2 m \leq \dim V - \no.$$ \end{assumption}

\begin{proposition}\label{p.cell}
For each $0 \leq k \leq \min \{m, \no\}$, the followings hold.
\begin{itemize} \item[(1)] The morphism $$\psi_k: \Go(m, V; k) \to \Gr(k, {\rm Null}_{\omega}) \times \G_{\omega^{\flat}}(m-k, V^{\flat})$$ defined by sending $[W] \in \Go(m, V;k)$ to
$$\psi_k ([W]) = \left( [W \cap {\rm Null}_{\omega}], [W^{\flat}] \right)$$ has a structure of a vector bundle of rank $(m-k)(\no-k)$. Consequently, the stratum $\Go(m, V; k)$ is a connected and simply-connected nonsingular variety.  It consists of a single point if  $k=m=\no$ and  contains positive-dimensional projective subvarieties otherwise. \item[(2)] The  stratum $\Go(m,V; k)$ is an orbit of the natural action of $\Sp_{\omega}(V)$ on $\Go(m,V)$.
\item[(3)]  $ \dim \Go(m, V;k) = \frac{1}{2} (m-k) ( 2 \dim V - 3m + 3k +1 ) + k (\no -m).$
\end{itemize}
\end{proposition}

\begin{proof}
Fix a subspace $V^{\sharp} \subset V$ complementary to ${\rm Null}_{\omega}$. For a subspace
$U \subset V^{\flat}$, let $U^{\sharp} \subset V^{\sharp}$ be the corresponding subspace under the natural isomorphism $V^{\sharp} \cong V^{\flat}$.
Given $[E] \in \Gr(k, {\rm Null}_{\omega})$ and $[U] \in \G_{\omega^{\flat}}(m-k, V^{\flat})$, the sum $E \oplus U^{\sharp}$ determines an element $\sigma_k([E],[U]) \in \Go(m, V;k)$. This defines a section $$\sigma_k:  \Gr(k, {\rm Null}_{\omega}) \times \G_{\omega^{\flat}}(m-k, V^{\flat}) \to \Go(m, V; k)$$ of $\psi_k$.

Given $[E] \in \Gr(k, {\rm Null}_{\omega})$ and $[U] \in \G_{\omega^{\flat}}(m-k, V^{\flat})$, fix a subspace  $E^c \subset {\rm Null}_{\omega}$ complementary to $E$
and a subspace $U^c \subset V^{\flat}$ complementary to $U$.
Regard $\Hom(U, E^c)$ as a subspace of $\Hom(V^{\flat}, {\rm Null}_{\omega})$ consisting of homomorphisms whose kernels contain $U^c$ and images lie in $E^c$.

We claim that via the action of $\Hom(U, E^c)$ on $\Go(m,V;k)$ described in Lemma \ref{l.Sp}, the fiber $\psi_k^{-1}([E], [U])$ can be identified with the vector space $\Hom ( U, E^c)$.
In fact, it is easy to check that the orbit map through $\sigma_k([E],[U])$  $$\varphi \in \Hom(U, E^c) \mapsto \varphi \cdot [E \oplus U^{\sharp}]$$
is injective with the image in the fiber  $\psi_k^{-1}([E], [U]).$
Let us check that it is surjective onto the fiber.
Given $[W] \in \psi_k^{-1}([E],[U])$ and $u\in U$,
 pick $u_W \in W$ such that $u_W^{\flat} = u.$ Then $u_W - u^{\sharp} \in {\rm Null}_{\omega}$. Let $\varphi_W (u) \in E^c$ be the element corresponding to  $ u_W - u^{\sharp}$ by the quotient ${\rm Null}_{\omega} \to E^c$ modulo $E.$
As $W \cap {\rm Null}_{\omega}  = E$, this does not depend on the choice of $u_W$. This determines an element $\varphi_W \in \Hom(U, E^c)$. From the definition, we see
that $$\varphi_W \cdot [E \oplus U^{\sharp}] = [W],$$ verifying the surjectivity and proving the claim. The claim implies (1).

 The map $\psi_k$ is equivariant under  naturally induced transitive action of $\Sp_{\omega}(V)$ on $\Gr(k, {\rm Null}_{\omega})\times\G_{\omega^{\flat}}(m-k, V^{\flat})$. We have just showed that
  each fiber of $\psi_k$ belongs to an orbit
of $\Sp_{\omega}(V)$. It follows that $\Go(m,V;k)$ is an orbit of $\Sp_{\omega}(V)$, proving (2).

From (1),  \begin{eqnarray*} \dim \Go(m, V;k) &=& \dim \Gr(k, {\rm Null}_{\omega}) + \dim \G_{\omega^{\flat}}(m-k, V^{\flat}) \\ & & + \dim \Hom ( \C^{m-k}, \C^{\no -k}).\end{eqnarray*}
Using the standard fact (e.g. p. 608 of \cite{HM05}),
\begin{eqnarray*} \dim \Gr(k, {\rm Null}_{\omega}) &=& k (\no -k), \\  \dim \G_{\omega^{\flat}}(m-k, V^{\flat}) &=& \frac{1}{2}(m-k)( 2 \dim V^{\flat} - 3 (m-k) +1), \end{eqnarray*} we
 obtain (3).
 \end{proof}

\begin{proposition}\label{p.closure}
The closure of $\Go(m, V; k)$ in $\Gr(m, V)$ is
$$\bigcup_{j \geq k} \Go(m, V; j).$$
It follows that $\Go(m, V)$ is irreducible and  $\Go(m, V; 0)$ is a Zariski open subset of
$\Go(m, V).$  \end{proposition}

\begin{proof}
 It suffices to show that for each $0 \leq k \leq \min \{m-1, \no-1\}$, the closure of $\Go(m, V; k)$ contains an element of $\Go(m, V; k+1)$.
 Given $[W] \in \Go(m, V;k)$, we can \begin{itemize}
  \item[(i)] pick  $v \in {\rm Null}_{\omega}$ satisfying $v \not\in W$ from $k \leq \no -1;$  \item[(ii)] pick a nonzero complement $W^{\sharp} \subset W$ to  $W \cap {\rm Null}_{\omega}$ from $k \leq m-1;$ \item[(iii)]  fix a subspace $W_0 \subset W^{\sharp}$ of codimension 1 and
 a nonzero vector $w \in W^{\sharp} \setminus W_0$; and \item[(iv)]  pick
 an  element $\varphi \in \Hom (V^{\flat}, {\rm Null}_{\omega})$ satisfying $\varphi(w^{\flat}) = v$ and $\varphi (W^{\flat}_0) =0$, in terms of the nonzero vector
 $w^{\flat} \in V^{\flat}$ and (possibly zero) $W^{\flat}_0 \subset V^{\flat}$. \end{itemize}
   For  $0 \neq  t \in \C, $  the element $$t \varphi \in \Hom(V^{\flat}, {\rm Null}_{\omega}) \subset \Sp_{\omega}(V)$$ acts on $\frac{1}{t} w$ by  $$\frac{1}{t}w \in W \mapsto \frac{1}{t} w
 + t \varphi(\frac{1}{t}w^{\flat}) = \frac{1}{t}w +  v$$ and moves $W$ to $(W \cap {\rm Null}_{\omega}) + W_0 + \C (\frac{1}{t}w +  v).$ Thus the closure of the orbit of $[W] \in \Go(m, V)$ under the additive subgroup $\{ t \varphi, t \in \C\}$ of $\Sp_{\omega}(V)$ contains an element of $\Go(m, V; k+1).$
  This proves the proposition. \end{proof}

 \begin{corollary}\label{c.codim}
  In Proposition \ref{p.closure}, the complement of $\Go(m,V;0)$ is a subvariety of codimension at least 2 in $\Go(m,V)$. \end{corollary}

  \begin{proof}
  From Proposition \ref{p.cell} (3), we have
  $$\dim \Go(m,V;0) - \dim \Go(m,V;1) = \dim V - 2m - \no +2.$$
  By Assumption \ref{a.no}, this is at least 2. \end{proof}

\begin{proposition}\label{p.smooth}
The set $\Go(m,V)^{\rm reg}$ of nonsingular points of  $\Go(m, V)$  contains $\Go(m, V;1)$ and the singular locus of $\Go(m, V)$ has codimension at least $ 7 $.
\end{proposition}

\begin{proof}
The second statement about the codimension of the singular locus is
a direct consequence of the first statement and
Proposition \ref{p.cell} (3), namely,
$$\dim \Go(m,V;0) - \dim \Go(m, V; 2) = 2 \dim V - 4m -2 \no +7$$
is at least 7 from our assumption $2 m \leq \dim V - \no.$

 Let us prove the first statement.
When $\no =0$, the result is automatic because $\Go(m, V) = \Go(m, V;0)$ is nonsingular.

Assuming $\no \geq 1$, fix a subspace $F \subset {\rm Null}_{\omega}$ with $\dim F = \no -1$. Let $\wedge^2_F V^*$ be the subspace of $\wedge^2 V^*$ comprising elements whose null spaces contain $F$.   Let $\sT$ be the tautological bundle of rank $m$ on $\Gr(m, V)$. An element $\theta \in \wedge^2_F V^*$ can be regarded as a global section $\widetilde{\theta}$ of $\wedge^2 \sT^*$, whose value at $[W] \in \Gr(m, V)$ is the restriction $\theta|_W \in \wedge^2 \sT^*_{[W]}.$
Thus $\Go(m,V)$ is the zero locus of the section $\widetilde{\omega}$.

If $[W] \in \Gr(m, V)$ satisfies $W \cap F =0$, any element of $\wedge^2 W^*$ can be extended to an element
of $\wedge^2_F V^*$. This means that the sections $$\{ \widetilde{\theta}, \ \theta \in \wedge^2_F V^* \}$$ generate the fiber of $\wedge^2 \sT^*$ at the point $[W]$ satisfying $W \cap F =0$. Thus by  Bertini's theorem, the zero locus of $\widetilde{\theta}$ for a general
$\theta \in \wedge^2_F V^*$ is nonsingular at $[W]$ if $W \cap F =0$.
By our choice of $F$, we know that $\dim V - \dim F = \dim V - \no +1$ is odd, which implies that a general element of $\wedge^2_F V^*$ has nullity $\dim F +1 = \no.$
Since any two elements of $\wedge^2_F V^*$ of the same nullity $\no$ are isomorphic by some $a \in \GL(V)$ satisfying $a(F)=F$, we see that
the zero locus of $\widetilde{\omega}$ is nonsingular at $[W]$ if $W \cap F =0$.

Since $F$ is a hyperplane in ${\rm Null}_{\omega}$, there exists $[W] \in \Go(m, V;1)$ satisfying $W \cap F =0$. We conclude that $\Go(m, V;1)$ has nonempty intersection with
$\Go(m,V)^{\rm reg}$. But $\Go(m, V;1)$ is a $\Sp_{\omega}(V)$-orbit by Proposition \ref{p.cell}. It follows that $\Go(m, V;1) \subset \Go(m, V)^{\rm reg}$. \end{proof}

\begin{remark} Proposition \ref{p.smooth} says that $\Go(m,V)$ is nonsingular if $\no =1$, a result proved in \cite{Mi}. In fact, the above argument of ours is a modification of the proof of Proposition 4.1 in \cite{Mi}. \end{remark}

\begin{proposition}\label{p.SpVMRT}
Let us regard $\Go(m,V) \subset \Gr(m,V)$ as subvarieties of $\BP \wedge^m V$ by  the Pl\"ucker embedding. A line on $\Gr(m,V)$ or $\Go(m;V)$ refers to a line in $\BP \wedge^m V$ lying on them. \begin{itemize} \item[(a)] All lines on $\Go(m,V)$ through a point $x \in \Go(m,V;0),$  are contained in $\Go(m, V;0) \cup \Go(m, V;1) \subset \Go(m, V)^{\rm reg}$;  \item[(b)]  For a point $x \in \Go(m,V;0) \subset \Gr(m, V),$ there exists a maximally nondegenerate element $\omega' \in \wedge^2 V^*$ such that $\G_{\omega'}(m,V) \subset \Gr(m, V)$ contains $x$ and the set of lines through $x$ on $\Go(m,V)$ and the set of lines through $x$ on $\G_{\omega'}(m,V)$ coincide. \item[(c)] When $m \leq \no$, all lines on $\Gr(m,V)$ passing  through a point of
$\Go(m, V; m)$ are contained in $\Go(m, V)$. \item[(d)] When $m \leq \no$, the Zariski tangent space of
$\Go(m, V)$ at a point  $x \in \Go(m, V;m)$ is equal to the tangent space $T_x \Gr(m, V).$ In particular, the stratum $\Go(m, V;m)$ is contained in the singular locus of $\Go(m, V)$.
\end{itemize} \end{proposition}

\begin{proof}
Recall that for a point $[W] \in \Gr(m, V),$ the choice of a flag $W_{-} \subset W \subset W_+ \subset V$ with $\dim W_-= \dim W -1$ and $\dim W_+ = \dim W +1$ determines a line $\F(W_-, W_+)
\subset \Gr(m, V)$ through $[W]$ defined by $$\F(W_-, W_+):= \{ [W'] \in \Gr(m, V), \ W_- \subset W' \subset W_+\}$$ and all lines on $\Gr(m,V)$ through $[W]$ arise this way.

If $[W] \in \Go(m,V;0)$ and $\F(W_-, W_+) \subset \Go(m, V)$, then
$W \cap {\rm Null}_{\omega} =0$ gives $$\dim \left( W_+ \cap {\rm Null}_{\omega} \right) \  \leq \ 1.$$
Thus any $[W'] \in \F(W_-, W_+)$ satisfies $\dim ( W' \cap {\rm Null}_{\omega})   \leq 1$, which means $$\F(W_-, W_+) \subset \Go(m, V;0) \cup \Go(m, V;1),$$ proving (a).

For a subspace $U \subset V$, define
$$U^{\perp \omega}:= \{ w \in V, \ \omega(w, U) =0\}.$$
We claim that for any point $[W] \in \Go(m, V)$,  the line $\F(W_-, W_+) \subset \Gr(m,V)$ determined by a flag
$W_{-} \subset W \subset W_+ \subset V$ with $\dim W_-= \dim W -1$ and $\dim W_+ = \dim W +1$
 is contained in $\Go(m, V)$  if and only if
$W_+ \subset W_-^{\perp \omega}.$
 In fact, the inclusion $\F(W_-, W_+) \subset \Go(m, V)$ implies that any element $w \in W_+$ is contained in some hyperplane $W' \subset W_+$
satisfying $W_- \subset W'$ and  $\omega(W', W') =0$. This is equivalent to the inclusion $W_+ \subset W_-^{\perp \omega}$. The opposite direction of the claim is straight-forward.

Fix a point $[W] \in \Go(m, V; 0)$.
Pick an element $\sigma \in \wedge^2 {\rm Null}_{\omega}$ which is maximally nondegenerate.
From $W \cap {\rm Null}_{\omega} =0$, we can pick a subspace $ H \subset V$  such that
$H$ is complementary to ${\rm Null}_{\omega}$ and contains $W$. We have a maximally nondegenerate form $\omega' \in \wedge^2 V^*$ defined by \begin{eqnarray*} \omega'(u,v) &=& \omega(u,v) \mbox{ if } u \in H, v \in V \\
\omega'(u, v) &=& \sigma(u, v) \mbox{ if } u, v \in {\rm Null}_{\omega}.\end{eqnarray*}
Then for any hyperplane $W_- \subset W$, we have $W_-^{\perp \omega} = W_-^{\perp \omega'}$.
By the above claim, this shows that the sets of lines through $[W]$ on $\Go(m, V)$
and $\G_{\omega'}(m,V)$ are identical, proving (b).

Now assume that $m \leq \no$ and $[W] \in \Go(m, V;m).$ This means that $W \subset {\rm Null}_{\omega}$ and $W^{\perp \omega} = V$. Thus for any hyperplane $W_- \subset W$, we have $W_-^{\perp \omega} = V$. Then the above claim shows that for any choice of $W_+ \subset V$, the line  $\F(W_-, W_+)$ is contained in
$\Go(m,V)$, proving (c).

The first statement in (d) is a consequence of (c) and the well-known fact that the tangent to lines through a point $x \in \Gr(m, V)$ spans $T_x \Gr(m, V)$. This implies the second statement in (d) because  $$\dim \Go(m,V)  = \frac{m}{2} (2 \dim V - 3m +1) \ < \ \dim \Gr(m, V)$$ from Proposition \ref{p.cell} (3).
\end{proof}

\begin{corollary}\label{c.normal}
The variety $\Go(2, V)$ is normal and its singular locus is nonempty if $\no \geq 2$.
\end{corollary}

\begin{proof}
It is easy to see (as in the proof of Proposition \ref{p.smooth}) that $\Go(2, V)$ is a hyperplane section of $\Gr(2, V)$ under the Pl\"ucker embedding.
So its singularities are of  hypersurface type and it must be normal because the singular loci has codimension at least $7$ from Proposition \ref{p.smooth}. By Proposition \ref{p.SpVMRT}, it is singular at points of $\Go(2, V;2)$. \end{proof}

We do not know whether $\Go(m, V)$ is normal for $m \geq 3$.
For us, the following is sufficient.

\begin{proposition}\label{p.normalization}
When $\no \geq 2$ and $m \geq 3$, the singular locus of the normalization of $\Go(m, V)$ is nonempty. \end{proposition}

\begin{proof}
Let $V^{\sharp} \subset V$ be a subspace complementary to ${\rm Null}_{\omega}$. Then
$(V^{\sharp}, \omega|_{V^{\sharp}})$ is a symplectic vector space of dimension $2 \ell := \dim V - \no$.
Let $$\{ \alpha_1, \ldots, \alpha_{\ell}, \beta_1, \ldots, \beta_{\ell} \}$$
be a symplectic basis of $(V^{\sharp}, \omega_{V^{\sharp}})$ satisfying $$\omega(\alpha_i, \beta_j) = \delta_{ij}, \ \omega(\alpha_i, \alpha_j) = \omega(\beta_i, \beta_j) = 0$$ for all $1 \leq i, j \leq \ell.$ Since $3 \leq m \leq \ell$ by Assumption \ref{a.no}, we can define \begin{eqnarray*}
V_0 & := & {\rm Null}_{\omega} + \langle \alpha_{m-1}, \alpha_m, \ldots, \alpha_{\ell}, \beta_{m-1}, \beta_m, \ldots, \beta_{\ell} \rangle \\
V_+ & :=& \langle \alpha_1, \ldots, \alpha_{m-2} \rangle \\
V_- &:= & \langle \beta_1, \ldots, \beta_{m-2} \rangle. \end{eqnarray*}
Let $\omega_0$ be the restriction of $\omega$ to $V_0$.
Define a $\C^*$-action on $V$ such that $t \in \C^*$ acts by
$$v_0 \in V_0 \mapsto v_0, \ v_+ \in V_+ \mapsto t v_+, \ v_- \in V_- \mapsto t^{-1} v_-.$$ This gives a multiplicative subgroup ${\bf T} \subset \Sp_{\omega}(V)$.

The induced ${\bf T}$-action on $\Gr(m, V)$ fixes a point $[W] \in \Gr(m, V)$ if and only if $W$ has a ${\bf T}$-homogeneous basis, or equivalently, we can write
$W = W_0 + W_+ + W_-$ for some subspaces $$W_0 \subset V_0, \ W_+ \subset V_+ \mbox{ and } W_- \subset V_-.$$ From this, it is easy to check that
$$S:= \{ [W] \in \Go(m, V), \ W = W_0 + V_+ \mbox{ for some } [W_0] \in \G_{\omega_0}(2, V_0) \}$$ is an irreducible component of the fixed loci of ${\bf T}$-action on $\Go(m,V).$  Since ${\bf n}_{\omega_0} = \no \geq 2$ and $$\dim V_0 - {\bf n}_{\omega_0} = 2 (\ell -m +2) \geq 4,$$ we can apply Corollary \ref{c.normal} to see that
$S \cong \G_{\omega_0}(2, V_0)$ is a normal variety with nonempty singular loci.
Since the point  $[W] \in S$ given by
$$W = \langle \alpha_{m-1}, \alpha_m\rangle + V_+$$
is contained in $\Go(m,V;0)$, the subvariety $S \subset \Go(m,V)$ has nonempty intersection with the nonsingular locus $\Go(m,V)^{\rm reg}$.

Suppose that the normalization of $\Go(m,V)$ is nonsingular.
 The ${\bf T}$-action lifts to the normalization and its fixed loci must be nonsingular.
 The fixed loci must contain a component $S'$ which is birational over $S$ because
 $S$ has nonempty intersection with $\Go(m, V)^{\rm reg}$.
 Then $S' \to S$ is a finite birational morphism. As $S$ is normal, it has to be biregular. But $S'$ is nonsingular while $S$ has nonempty singular loci, a contradiction.
\end{proof}

  \begin{theorem}\label{t.SpPic}
  Let $X$ be a Fano manifold of Picard number 1. Assume that
  there exists a Zariski open subset in $X$ biholomorphic to $\Go(m,V;0)$ for some presymplectic space
  $(V, \omega)$ satisfying Assumption \ref{a.no}. Then $\omega$ is maximally nondegenerate and $X$ is  biregular to $\Go(m, V)$. \end{theorem}

  \begin{proof}
  Let $X' \subset X$ be a Zariski open subset admitting a biholomorphic map $h: X' \to \Go(m,V;0).$ The complement $X \setminus X'$ has codimension at least 2 in $X$ because
  $X'$ contains positive-dimensional projective subvarieties by Proposition \ref{p.cell} (1) and $X$ has Picard number 1. Let $\sL$ be the line bundle on
  $\Go(m, V)$ inducing  the Pl\"ucker embedding $\Go(m, V) \subset \BP \wedge^m V$. There exists a line bundle $L$  on $X$ such that $L|_{X'} = h^* \sL$ because $X \setminus X'$ has codimension at least $2$. Then $L$ must be ample because $\sL$ is ample and $X$ has Picard number 1.
  Let $\nu: Y \to \Go(m, V)$ be the normalization of $\Go(m,V)$. Since the codimension of the complement of $\Go(m,V;0)$ in $Y$ is at least 2 by Corollary \ref{c.codim}, the Hartogs extension gives a natural identification, for every $\ell \in \N$,
     $$H^0(X, L^{\otimes \ell} ) \cong H^0(X', h^*\sL^{\otimes \ell}) \cong H^0(\Go(m,V;0), \sL^{\otimes \ell}) \cong H^0(Y, \nu^* \sL^{\otimes \ell}).$$ Since $L$ is ample on $X$ and $\nu^* \sL$ is ample on $Y$, this implies that $X$ is biregular to $Y$. Thus $\no \leq 1$ by Corollary \ref{c.normal} and Proposition \ref{p.normalization}. Thus  $\omega$ must be maximally nondegenerate  and
     $X$ is biregular to $\Go(m,V)$. \end{proof}

\section{Presymplectic Lie algebras as prolongations}\label{s.prolong}
\subsection{Prolongation of a nilpotent Lie algebra associated to a presymplectic vector space}
The goal of this subsection is to compute the prolongation of the following graded nilpotent Lie algebra.

\begin{notation}\label{n.positive}
Let $U$ be a vector space with $\dim U \geq 2$. Let $(Q, \omega)$  be a presymplectic vector space, i.e., a vector space $Q$ together with a nonzero anti-symmetric form $\omega \in \wedge^2 Q^*$.
\begin{itemize} \item[(1)] Set $$\fg_{-1} := U \otimes Q, \ \fg_{-2} := \Sym^2 U$$ and provide $\fg_-= \fg_{-1} \oplus \fg_{-2}$ with a graded nilpotent Lie algebra structure with the bracket defined by
$$[ u_1 \otimes q_1, \ u_2 \otimes q_2] := \omega (q_1, q_2) \ u_1 \odot u_2$$ for $q_1, q_2 \in Q$,  $u_1, u_2 \in U$ and their symmetric product $u_1 \odot u_2 \in \Sym^2 U$.
\item[(2)] Recall that the presymplectic Lie algebra of $(Q, \omega)$ is
$$\spo(Q):= \{ \varphi \in \fgl(Q), \ \omega(\varphi \cdot q_1, q_2) = \omega(\varphi \cdot q_2, q_1) \mbox{ for all } q_1, q_2 \in Q\}.$$ The Lie algebra $\fgl(U) \oplus \spo(Q)$ has a faithful representation on $\fg_-$  where $(a, c)  \in \fgl(U) \oplus \spo(Q)$ acts on $\fg_-$ by \begin{eqnarray*}
(a,c) \cdot (u\otimes q) &=& (a \cdot u) \otimes q + u \otimes (c \cdot q) \\
(a, c) \cdot (u_1 \odot u_2) &=& (a \cdot u_1) \odot u_2 + u_1 \odot (a \cdot u_2) \end{eqnarray*}
for $u, u_1, u_2 \in U$ and $q \in Q.$ It preserves the grading and the Lie algebra structure of $\fg_-$.
Let $\fg_0 \subset {\rm End}(\fg_-)$ be the image of this representation of $\fgl(U) \oplus \spo(Q)$.
\end{itemize}
\end{notation}

\begin{remark}
The Lie algebra $\fg_0$ is strictly smaller than ${\rm gr}\aut(\fg_{-})$ if ${\rm Null}_{\omega} \neq 0$. In fact, taking a complement $Q^{\sharp} \subset Q$ of the null space ${\rm Null}_{\omega}$, the action of
 $\fgl(U \otimes {\rm Null}_{\omega})$  on $$
 \fg_- = (U \otimes {\rm Null}_{\omega})  \ \oplus \ (U \otimes Q^{\sharp}) \ \oplus \ \Sym^2 U$$ which annihilates $(U \otimes Q^{\sharp})  \oplus \Sym^2 U$
 preserves the nilpotent Lie algebra structure. But $\fgl( U \otimes {\rm Null}_{\omega}) \not\subset \fg_0$ if ${\rm Null}_{\omega} \neq 0$.
\end{remark}

\begin{definition}\label{d.h}
In Notation \ref{n.positive}, given $h \in \Hom(U,Q)$ and  $B \in \Sym^2 U^*$, we define the following elements for each  $w \otimes p \in U \otimes Q = \fg_{-1},$
$v \odot w \in \Sym^2 U = \fg_{-2}$:
\begin{itemize}
\item[(1)] an element $h_{v \odot w} \in U \otimes Q = \fg_{-1}$ defined by
    $$h_{v \odot w} := \frac{1}{2} w \otimes h(v)  + \frac{1}{2} v \otimes h(w);$$
\item[(2)] an element $h^U_{w\otimes p} \in \fgl(U) \subset \fg_0$  defined by
    $$ u \in U \ \mapsto \ h^U_{w\otimes p} (u) := \frac{1}{2} \omega (h(u), p) w;$$
    \item[(3)] an element $h^Q_{w \otimes p} \in \spo(Q) \subset \fg_0$  defined by $$q \in Q \ \mapsto \ h^Q_{w\otimes p}(q) := \frac{1}{2} \omega(h(w),q) p + \frac{1}{2} \omega(p,q)\ h(w);$$
        \item[(4)] an element $B_{v \odot w} \in \fgl(U) \subset \fg_0$  defined by
$$ u \in U \ \mapsto \ B_{v \odot w} (u) := \frac{1}{2} B (u, v) \ w + \frac{1}{2} B(u, w) \ v ;\mbox{ and }$$
\item[(5)] an element $B_{w \otimes p} \in \Hom(U,Q)$ defined by
$$ u \in U \  \mapsto \ B_{w\otimes p}(u) := B (u, w) p.$$
\end{itemize} \end{definition}

\begin{theorem}\label{t.positive}
For $(\fg_0, \fg_-)$ in Notation \ref{n.positive}  and its universal prolongation $\fg = \oplus_{k \in \Z} \fg_k$  in the sense of Definition \ref{d.fg}, we have   the following identifications of $\fg_0$-modules \begin{eqnarray*}
\fg_1 &=& \Hom (U, Q) \\ \fg_2 &=& \Sym^2 U^* \\ \fg_k &=& 0 \mbox{ for all } k \geq 3 \end{eqnarray*} such that $h \in \Hom(U, Q) = \fg_1$ and $B \in \Sym^2 U^* = \fg_2$ satisfy
for each $v \odot w \in \fg_{-2}$ and $ w \otimes p \in \fg_{-1}$, \begin{eqnarray*} h \cdot ( v \odot w) &=& h_{v \odot w} \in  \fg_{-1} \\
h \cdot( w \otimes p) &=& h^U_{w \otimes p}+ h^Q_{w \otimes p} \in \fg_0 \\
B \cdot( v \odot w) &=& B_{v \odot w} \in \fg_0 \\
B \cdot (w \otimes p) &=& B_{w \otimes p} \in \fg_1
\end{eqnarray*}
\end{theorem}

    Theorem \ref{t.positive} will be proved in the next five lemmata.

    \begin{lemma}\label{l.positive1} Recall from Definition \ref{d.fg} and Notation \ref{n.positive}, that an  element  $$\varphi \in \fg_1 \subset  \Hom(\fg_{-2}, \fg_{-1}) \oplus
    \Hom(\fg_{-1},  \fg_0)$$ satisfies for all $u, w_1, w_2,  \in U$ and $p, q_1, q_2 \in Q$,
    \begin{itemize} \item[(a)] $\varphi(u \otimes p) \cdot (w_1 \odot w_2) = [\varphi(w_1 \odot w_2), u\otimes p], \mbox{ and }$
    \item[(b)] $\omega(q_1,q_2) \ \varphi(w_1 \odot w_2) = \varphi( w_1 \otimes q_1) \cdot (w_2 \otimes q_2) - \varphi(w_2\otimes q_2) \cdot (w_1 \otimes q_1).$ \end{itemize}
For each $\varphi \in \fg_1$, there exists an element $h \in \Hom(U,Q)$ such that \begin{itemize}
\item[(i)] $\varphi(w_1 \odot w_2) = h_{w_1 \odot w_2}$, and
\item[(ii)] $\varphi(w \otimes q) = h^U_{w \otimes q} + h^Q_{w \otimes q}.$ \end{itemize}
\end{lemma}

\begin{proof}
To start with, we claim \begin{eqnarray} \label{e.u2} \varphi(u^2) \in   u \otimes Q \mbox{ for all } u \in U. \end{eqnarray}
First, assume that $\dim Q \geq 3$. Putting $ 0 \neq u = w_1=w_2$ in (b), we have
   $$\omega(q_1, q_2) \ \varphi(u^2)   \ \in \  u \otimes Q + U \otimes q_1 + U \otimes q_2. $$ This implies \begin{eqnarray*} \varphi(u^2)\in \bigcap\limits_{q_1, q_2 \text{ general }}(u \otimes Q + U \otimes q_1 + U \otimes q_2). \end{eqnarray*} Fixing a complement
  $U^{\sharp} \subset U$ of $\C u$, we can write
  $$u \otimes Q + U \otimes q_1 + U \otimes q_2 = (u \otimes Q) \oplus (U^{\sharp} \otimes (\C q_1 + \C q_2)).$$ Then the above intersection is equal to
  $$(u \otimes Q) \oplus (U^{\sharp} \otimes \bigcap\limits_{q_1, q_2 \text{ general }}
  (\C q_1 + \C q_2)).$$ But $\dim Q \geq 3$ implies
  $$ \bigcap\limits_{q_1, q_2 \text{ general }}
  (\C q_1 + \C q_2) =0$$ and we obtain $\varphi(u^2) \in u \otimes Q,$ proving the claim when
  $\dim Q \geq 3$.
If $\dim Q <3$,  then $\dim Q = 2$ and ${\rm Null}_{\omega}=0.$
 Pick two linearly independent vectors $u, v \in U$ and put $v= w_1= w_2$ in (a) to obtain $$ v \odot U \ \ni \   2 v  \odot ( \varphi(u \otimes p) \cdot v )  =  [\varphi(v^2), u \otimes p]\in u \odot U.$$
   As $v$ and $u$ are linearly independent and ${\rm Null}_{\omega} =0$,
   we have $$[\varphi(v^2), u \otimes p] \in \C u \odot v,$$ which  implies
$\varphi(v^2) \in v \otimes Q,$ proving the claim.

   From (\ref{e.u2}),  for each $0 \neq u \in U$, there exists a unique element $q_u \in Q$ such that
$\varphi(u^2) = u \otimes q_u$. Since $q_{tu} = t q_u$ for $t \in \C$, the association $u \mapsto q_u \in Q$ defines an element of $$H^0( \BP U, \sO(1) \otimes Q) = \Hom(U, Q).$$
Thus we have an element $h \in \Hom(U,Q)$ such that
$\varphi(u^2) = u \otimes h(u).$ By polarizing it, we see that $h$ satisfies (i).

Let us write $\varphi(u \otimes q) = (a_{u\otimes q}, c_{u \otimes q}) \in \fgl(U) \oplus \spo(Q)$. Notation \ref{n.positive} (2) says
$$\varphi(u \otimes q) \cdot (w^2) = 2 a_{u \otimes q}(w) \odot w \mbox{ for all } w \in U. $$
By (a) and (i), this is equal to \begin{eqnarray*} [\varphi(w^2), u \otimes q]   &=& [w \otimes h(w), u \otimes q] \\ &=&  \omega(h(w), q) u \odot w. \end{eqnarray*}
Thus $$ 2 a_{u \otimes q}(w)  =  \omega(h(w), q ) u = 2 h^U_{u \otimes q} (w),$$ checking the first half of (ii).

From (i), the left hand side of (b) is
\begin{eqnarray}\label{e.bl} \omega(q_1, q_2) \left( \frac{1}{2} w_1 \otimes h(w_2) + \frac{1}{2} w_2 \otimes h(w_1) \right).\end{eqnarray}
The right hand side of (b) is
$$ a_{w_1 \otimes q_1}(w_2) \otimes q_2 +
w_2 \otimes c_{w_1 \otimes q_1}(q_2) - a_{w_2 \otimes q_2}(w_1) \otimes q_1 - w_1 \otimes c_{w_2 \otimes q_2} (q_1).$$ From the first half of (ii) checked above, this is equal to
\begin{eqnarray*}\label{e.br}\frac{1}{2} \omega(h(w_2), q_1) w_1 \otimes q_2 + w_2 \otimes c_{w_1 \otimes q_1}(q_2) \\ - \frac{1}{2} \omega(h(w_1), q_2) w_2 \otimes q_1 - w_1 \otimes c_{w_2 \otimes q_2}(q_1). \end{eqnarray*}
For this to be equal to (\ref{e.bl}) when $w_1$ and $w_2$ are linearly independent (using $\dim U \geq 2$), we should have
 \begin{eqnarray*} \frac{1}{2} \omega(q_1, q_2) h(w_2) &=& \frac{1}{2} \omega( h(w_2), q_1) q_2 - c_{w_2 \otimes q_2} (q_1) \\ \frac{1}{2} \omega(q_1, q_2) h(w_1) &=& -\frac{1}{2} \omega(h(w_1), q_2) q_1 + c_{w_1 \otimes q_1}(q_2). \end{eqnarray*}
Thus we obtain for general (hence all ) $u \in U$ and $q_1, q_2 \in Q$,
$$c_{u \otimes q_1} (q_2) = \frac{1}{2} \omega(q_1, q_2) h(u) + \frac{1}{2} \omega(h(u), q_2) q_1,$$ proving the second half of (ii). \end{proof}

\begin{lemma}\label{l.positive2}
The homomorphism $$\pi_1: \Hom(U,Q) \mapsto \Hom(\fg_{-2}, \fg_{-1}) \oplus \Hom(\fg_{-1}, \fg_0)$$ defined by $$h \in \Hom (U, Q) \mapsto \pi_1(h) := (v \odot w \mapsto h_{v \odot w}, w \otimes p \mapsto h^U_{w \otimes p} + h^Q_{w \otimes p})$$ gives a $\fg_{0}$-module isomorphism $\Hom(U, Q) \cong \fg_1.$
\end{lemma}

\begin{proof}
It is straight-forward to check that the image of $\pi_1$ is contained in $\fg_1$, and $\pi_1$ is an injective $\fg_{0}$-module homomorphism. Lemma \ref{l.positive1} says that the image of $\pi_1$ contains $\fg_1$. This proves the lemma. \end{proof}

    \begin{lemma}\label{l.positive3}
    Recall from Definition \ref{d.fg} and Notation \ref{n.positive}, that an  element
    of $\fg_2$ is given by a pair $$(\phi \in \Hom(\fg_{-2}, \fg_{0}), \Phi \in
    \Hom(\fg_{-1},  \fg_1))$$  satisfying for all $u_1, u_2, w_1, w_2 \in U$ and $q_1, q_2 \in Q$,
    \begin{itemize} \item[(a)] $\phi(u_1 \odot u_2) \cdot (w_1 \odot w_2) = \phi(w_1 \odot w_2) \cdot  (u_1 \odot u_2)$ \item[(b)] $\Phi(u_1 \otimes q_1) \cdot (w_1 \odot w_2) = \phi(w_1 \odot w_2) \cdot (u_1 \otimes q_1)$
    \item[(c)] $\omega(q_1,q_2) \ \phi(u_1 \odot u_2) = \Phi( u_1 \otimes q_1) \cdot (u_2 \otimes q_2) - \Phi(u_2\otimes q_2) \cdot (u_1 \otimes q_1).$ \end{itemize}
        For each $(\phi, \Phi) \in \fg_2$,
there exists an element $B \in \Sym^2 U^*$ such that \begin{itemize}
\item[(i)] $\phi(v \odot w) = B_{v \odot w}$ and
\item[(ii)] $\Phi(w \otimes q) = B_{w\otimes q}.$ \end{itemize}
\end{lemma}

\begin{proof}
From Lemma \ref{l.positive2}, we can view $\Phi$ as a homomorphism from $U \otimes Q$ to $\Hom (U, Q) = U^* \otimes Q$.
 Given $u, w \in U$ and $q \in Q$, define  $$h:= \Phi (u \otimes q) \in \Hom (U, Q).$$ Then
\begin{eqnarray}\label{e.h(b)}  \Phi(u \otimes q) \cdot w^2 = h_{w^2} = w  \otimes h(w) \in w \otimes Q. \end{eqnarray}
Setting $\phi(w^2) = (a_{w^2}, c_{w^2}) \in \fgl(U) \oplus \spo(Q)$, we have
\begin{eqnarray}\label{e.h(a)} \phi(w^2) \cdot (u \otimes q) = a_{w^2}(u) \otimes q + u \otimes c_{w^2}(q). \end{eqnarray}
Since $\dim U \geq 2$, we can choose $u$ independent of $w$ and obtain from the equality of (\ref{e.h(b)}) and (\ref{e.h(a)}) by (b),
\begin{eqnarray}\label{e.cw1} c_{w^2} &= &\lambda_{w^2}id_{Q} \mbox{ for some } \lambda_{w^2}\in\mathbb{C},  \mbox{ and } \end{eqnarray}
\begin{eqnarray}\label{e.Phiu1} \Phi(u\otimes q) &\in & U^* \otimes q. \end{eqnarray}
(\ref{e.Phiu1}) implies that the homomorphism $\Phi: U \otimes Q \to U^* \otimes Q$ must be given by a bilinear form
$B \in U^* \otimes U^*$ such that \begin{eqnarray}\label{e.hw1} h(w) &=& \Phi(u \otimes q) (w) = B(u, w) q. \end{eqnarray} Putting it into (\ref{e.h(b)}),
we obtain
$$h_{w^2}= B(u, w) w \otimes q.$$
Equating this with (\ref{e.h(a)}) gives
\begin{eqnarray}\label{e.aw} a_{w^2} (u) =  B(u, w) w-\lambda_{w^2}u.\end{eqnarray}
(\ref{e.cw1}), (\ref{e.hw1}) and (\ref{e.aw})  prove (i) and (ii), once we  verify that $B \in \Sym^2 U^*$ and $\lambda_{w^2}=0$ for all $w\in U$. From (a), $$  2w \odot a_{u^2} (w) = \phi(u^2) \cdot (w^2) = \phi(w^2) \cdot (u^2) = 2u \odot a_{w^2}(u). $$ By (\ref{e.aw}), this implies that
 $$  B(w,u) w \odot u - \lambda_{u^2}w^2 = B(u,w) u \odot w -\lambda_{w^2}u^2.$$ Since $\dim U\geq 2$, we can choose $u$ independent of $w$ and obtain $B \in \Sym^2 U^*$, and $\lambda_{w^2}=0$ for all $w\in U$.
  \end{proof}

  \begin{lemma}\label{l.positive4}
The homomorphism $\pi_2: \Sym^2 U^* \mapsto \Hom(\fg_{-2}, \fg_{0}) \oplus \Hom(\fg_{-1}, \fg_1)$ defined by $$B \in \Sym^2 U^* \mapsto \pi_1(B) := (v \odot w \mapsto B_{v \odot w}, w \otimes p \mapsto B_{w \otimes p})$$ gives a $\fg_{0}$-module isomorphism $\Sym^2U^* \cong \fg_2.$
\end{lemma}

\begin{proof}
It is straight-forward to check that the image of $\pi_2$ is contained in $\fg_2$, and $\pi_2$ is an injective $\fg_{0}$-module homomorphism. Lemma \ref{l.positive3} says that the image of $\pi_2$ contains $\fg_2$. This proves the lemma. \end{proof}

\begin{lemma}\label{l.positive5}
For every $k \geq 3$, the prolongation $\fg_k$ vanishes.
\end{lemma}

\begin{proof}
  Recall from Definition \ref{d.fg} and Notation \ref{n.positive}, that an  element  $\varphi \in \fg_3 \subset (\Hom(\fg_{-2}, \fg_1) \oplus \Hom(\fg_{-1}, \fg_2))$ satisfies  \begin{itemize} \item[(a)] $\varphi(u^2) \cdot w^2 = \varphi(w^2) \cdot u^2$ for any $u, w \in U$; and  \item[(b)] $\varphi(u^2) \cdot (w \otimes q) =\varphi(w \otimes q) \cdot u^2$  for any $w \otimes q \in \fg_{-1}$ and $u^2 \in \fg_{-2}$. \end{itemize}
   For $h^u := \varphi(u^2)$ and $ h^w:= \varphi(w^2) \in \Hom(U,Q)=\fg_1,$ Lemma \ref{l.positive2} gives
$$\varphi(u^2) \cdot w^2 = h^u_{w^2} = w \otimes h^u(w) \in w \otimes Q$$
 $$\varphi(w^2) \cdot u^2 =h^w_{u^2} = u \otimes h^w(u)  \in u \otimes Q.$$
As these two are equal for all choices of $u, w \in U$ by (a) and $\dim U \geq 2$, we see that $\varphi(u^2) = \varphi(w^2) =0$. This means \begin{eqnarray}\label{e.varphi} \varphi(\fg_{-2}) =0. \end{eqnarray}
Pick $w \otimes q \in \fg_{-1}$ and write $B:= \varphi(w \otimes q) \in \Sym^2 U^* = \fg_2$ from Lemma \ref{l.positive4}.
Then $B_{u^2} \in \fg_0$ satisfies \begin{eqnarray}\label{e.Bu2} B_{u^2} \cdot v = B(u, v) u \mbox{ for any } v \in U.\end{eqnarray}
From (b),
$$B_{u^2} = \varphi(w \otimes q) \cdot u^2 = \varphi(u^2) \cdot (w \otimes q)$$
which vanishes by (\ref{e.varphi}).
 Thus (\ref{e.Bu2}) must vanish for all $u, v \in U$, implying $B=0$. It follows that
 $\varphi(\fg_{-1}) =0$. This proves that $\fg_3=0$.

   An element $\varphi \in \fg_4 \subset (\Hom(\fg_{-2}, \fg_2) \oplus \Hom(\fg_{-1}, \fg_3))$ satisfies
   $$\varphi(u^2) \cdot (w \otimes q) = \varphi(w \otimes q) \cdot u^2$$  for any
   $u^2 \in \fg_{-2}$ and $w \otimes q \in \fg_{-1}.$ But the right hand side vanishes because $\fg_3=0$.
 It follows that $\varphi(\fg_{-2}) =0.$ Thus $\fg_4=0$.

 The vanishing of $\fg_k$ for $k \geq 5$ is automatic from $\fg_3= \fg_4=0$. \end{proof}

\subsection{A gradation of the presymplectic Lie algebra}\label{ss.positive}
We will relate the result in the previous subsection to a gradation of a presymplectic Lie algebra.

 \begin{definition}\label{d.VUQ}
 Let $U$ be a vector space of dimension at least 2 and let $(Q, \omega \neq 0)$ be a presymplectic vector space.
 Set $V:= Q \oplus U \oplus U^*$ and extend $\omega \in \wedge^2 Q^*$ to the presymplectic form $\widetilde{\omega} \in \wedge^2 V^*$  defined by \begin{itemize}
 \item[(1)] $\widetilde{\omega}|_{\wedge^2 Q} = \omega$; \item[(2)]
$\widetilde{\omega}(Q, U \oplus U^*) =0$; \item[(3)] $ \widetilde{\omega}(f,v) = f(v) \mbox{ for all } f \in U^*, v \in U$; \item[(4)] $\widetilde{\omega}(U,U) = \widetilde{\omega}(U^*, U^*) =0.$ \end{itemize} Then ${\rm Null}_{\widetilde{\omega}} = {\rm Null}_{\omega} \subset Q$. For simplicity, we will  write $\widetilde{\omega}$ as $\omega$. Define a grading on $V$ by $V_{-1} = U, V_0 = Q$ and $V_1 = U^*$. It induces a natural grading  $$\spo(V) = \spo(V)_{-2} \oplus \spo(V)_{-1} \oplus \spo(V)_0 \oplus \spo(V)_1 \oplus \spo(V)_2$$ defined by
  $$\spo(V)_k = \spo(V) \cap  \Hom(V, V)_k.$$
\end{definition}

The following lemma is straight-forward from the definition
$$\spo(V) = \{ \varphi \in \fgl(V),
 \omega(\varphi(v_1), v_2) + \omega( v_1, \varphi(v_2)) = 0 \mbox{ for all } v_1, v_2 \in V\}.$$

\begin{lemma}\label{l.spgrade}
\begin{itemize} \item[(1)]
$\spo(V)_2$ consists of elements $\alpha \in \Hom(U, U^*)$ satisfying  $\omega( \alpha(u), v) = \omega (\alpha(v), u) \mbox{ for all } u, v \in U.$
\item[(2)] $\spo(V)_1$ consists of pairs $$(\beta, \beta') \in \Hom(U,Q) \oplus \Hom(Q,U^*)$$
such that $\beta$ is arbitrary and $\beta'$  is uniquely determined by  $$\omega(\beta'(q), u) = \omega(\beta(u), q) \mbox{ for all } u \in U \mbox{ and } q \in Q.$$
\item[(3)] $\spo(V)_{-1}$ consists of pairs $$(\gamma, \gamma')  \in \Hom(U^*, Q) \oplus\Hom(Q, U)$$ such that $\gamma$ is arbitrary and  $\gamma'$ is uniquely determined by $$\omega(\gamma'(q), \lambda) = \omega(\gamma(\lambda), q) \mbox{ for all }\lambda \in U^* \mbox{ and } q \in Q.$$
    \item[(4)] $\spo(V)_{-2}$ consists of elements $\eta \in \Hom(U^*, U)$ satisfying $$\omega(\eta(\lambda), \mu) = \omega(\eta(\mu), \lambda) \mbox{ for all } \lambda, \mu \in U^*.$$
        \end{itemize} \end{lemma}

\begin{lemma}\label{l.spTanaka} Let us use the notation of Theorem \ref{t.positive} and Lemma \ref{l.spgrade}.
\begin{itemize}
\item[(a)]
If $\alpha \in \spo(V)_2$ satisfies $[\alpha, \spo(V)_{-1}] =0,$ then $\alpha=0.$
\item[(b)] If $(\beta, \beta') \in \spo(V)_1$ satisfies $[(\beta, \beta'), \spo(V)_{-2}]=0$, then $\beta= \beta'=0$.
        \item[(c)] Define $\psi_{-1}: U \otimes Q \to \spo(V)_{-1}$ by sending
 $$\gamma \in U \otimes Q = \Hom(U^*, Q) \mbox{ to }(\gamma, \gamma') \in \spo(V)_{-1}.$$
Define $\psi_{-2}: \Sym^2 U \to \spo(V)_{-2}$ by sending $$u \odot v \in \Sym^2 U \mbox{ to } \eta_{u \odot v} \in \Hom(U^*,U)$$ defined by $$\eta_{u \odot v} (\lambda) := \lambda(u)v + \lambda(v) u \mbox{ for } \lambda \in U^*.$$
Then $\psi_{-1}+\psi_{-2}: \fg_- \to \spo(V)_{-1} + \spo(V)_{-2}$ is an isomorphism of graded Lie algebras, i.e., an isomorphism satisfying
$$[\psi_{-1}(w_1\otimes q_1), \psi_{-1}(w_2 \otimes q_2)]
= \omega(q_1, q_2) \psi_{-2} (w_1 \odot w_2).$$
\item[(d)]
The natural action of $\fg_0= \fgl(U) \oplus \spo(Q)$ on $V= U^* \oplus Q \oplus U$ gives an injective homomorphism $\psi_0: \fg_0 \to \spo(V)_0$. This is an isomorphism satisfying the commuting diagram $$
\begin{array}{ccc} \fg_0 & \stackrel{\psi_0}{\longrightarrow} & \spo(V)_0 \\
\downarrow  & & \downarrow \\
\End(\fg_-) & \stackrel{\psi_{-1} + \psi_{-2}}{\longrightarrow} & \End(\spo(V)_{-1}+\spo(V)_{-2}) \end{array} $$ where the first vertical arrow is the inclusion from Notation \ref{n.positive} (2),  the second vertical arrow is
induced by the adjoint representation of $\spo(V)$ and the lower horizontal arrow is induced by the isomorphism
$\psi_{-1}+\psi_{-2}: \fg_- \cong \spo(V)_{-1} + \spo(V)_{-2}$ in (c).
    \end{itemize} \end{lemma}

\begin{proof}
Let $\alpha$ be as in (a). For any $(\gamma, \gamma') \in \spo(V)_{-1}$ in the notation of Lemma \ref{l.spgrade}, the assumption for $\alpha$ reads
$$\alpha \circ (\gamma, \gamma') = (\gamma, \gamma') \circ \alpha$$ as elements of ${\rm End}(V)$. Then for any $q \in Q\backslash {\rm Null}_{\omega}$, $$\alpha(\gamma'(q)) = \alpha \circ (\gamma, \gamma') (q) =
(\gamma, \gamma') \circ \alpha(q) =0.$$ But for any $u \in U$, we can choose $\gamma'$ so that $\gamma'(q)=u. $ Thus $\alpha =0.$

Let $(\beta, \beta')$ be as in (b). For any $\eta \in \spo(V)_{-2}$, we have
$$(\beta, \beta') \circ \eta = \eta \circ (\beta, \beta') .$$
Since the right hand side annihilates  any $\lambda \in U^*$,
we have $$0= (\beta, \beta')\circ \eta(\lambda) = \beta(\eta(\lambda)).$$
This implies $\beta=0$ because for any $u \in U$, we can choose $\eta$ and $\lambda$ such that $u= \eta(\lambda)$ by Lemma \ref{l.spgrade} (4). Then $\beta'=0$ from Lemma \ref{l.spgrade} (2).

From Lemma \ref{l.spgrade} (3) and (4), it is immediate that $\psi_{-1}$ and $\psi_{-2}$ in (c) are isomorphisms.
The Lie bracket in (c) is the commutator in ${\rm End}(V)$
\begin{eqnarray}\label{e.comm} \quad\quad\quad \psi_{-1}(w_1\otimes q_1) \circ \psi_{-1}(w_2 \otimes q_2) - \psi_{-1}(w_2 \otimes q_2) \circ \psi_{-1}(w_1 \otimes q_1).\end{eqnarray}
Note that \begin{eqnarray}\label{e.psi-1} \psi_{-1}(w \otimes q) (\lambda) &=& \lambda(w) q
\mbox{ for all } \lambda \in U^*.\end{eqnarray}
 From $\psi_{-1}(w \otimes q) \in \spo(V)$ and (\ref{e.psi-1}), we have
$$\omega( \psi_{-1}(w \otimes q)(p), \lambda) = \omega(\psi_{-1}(w \otimes q)(\lambda), p) =
\omega(\lambda(w) q, p) = \omega(q,p) \lambda(w)$$ for all $\lambda \in U^*$, which gives
\begin{eqnarray}\label{e.psi2} \psi_{-1}(w \otimes q) (p) &=& \omega(q, p) w. \end{eqnarray}
Evaluating the first term of (\ref{e.comm}) at $\lambda \in U^*$  applying (\ref{e.psi-1})
and (\ref{e.psi2}),
$$\psi_{-1}(w_1 \otimes q_1) (\lambda(w_2) q_2) = \lambda(w_2) \omega(q_1, q_2) w_1.$$
By the same argument, the second term in (\ref{e.comm}) evaluated at $\lambda$ is
$\lambda(w_1) \omega(q_2, q_1) w_2$.
Consequently,  (\ref{e.comm}) evaluated at $\lambda$ is
$$\omega(q_1, q_2) \left( \lambda(w_1) w_2 + \lambda(w_2) w_1 \right) = \omega(q_1, q_2) \psi_{-2}(w_1 \odot w_2) (\lambda).$$ This proves (c).

The injective homomorphism $\psi_0$ in (d) must be surjective because $\dim \fg_0 = \dim \spo(V)_0$ by  Lemma \ref{l.spgrade} and our knowledge of
$\dim \spo(V)$.   To prove (d), it remains to check the commuting diagram.

 For $(a,c) \in \fg_0= \fgl(U) \times \spo(Q),$ the element $\psi_0(a, c) \in \spo(V)_0 \subset \fgl(V)$ acts on $V = U^* \oplus Q \oplus U$ by $$\psi_0 (a,c)(\lambda + q + u) = a^* \lambda + c(q) + a(u)$$ where $a^*\lambda (u) = - \lambda (a \cdot u).$ For $v^2 \in \fg_{-2} = \Sym^2 U$, The definition of $\psi_{-2}$ in $(c)$ implies that \begin{eqnarray*}
 &&[\psi_0 (a, c), \psi_{-2}(v^2)](\lambda + q + u) \\
&=& \psi_0(a, c)\circ\psi_{-2}(v^2)(\lambda + q +u) - \psi_{-2}(v^2)\circ \psi_0 (a, c)(\lambda + q + u) \\
&=& \psi_0 (a, c)(2\lambda(v)v) - \psi_{-2}(v^2)(a^* \lambda + c(q) + a(u)) \\
&=& 2\lambda(v)a(v) + 2\lambda(a(v))v \\
&=& \psi_{-2}(2v\odot a(v))(\lambda + q + u) \\
&=& \psi_{-2}((a, c)\cdot (v^2))(\lambda + q +u). \end{eqnarray*}
It follows that \begin{eqnarray} \label{e.ad-2}\psi_{-2}((a, c)\cdot (v^2))=[\psi_0(a, c), \psi_{-2}(v^2)]. \end{eqnarray} For $w \otimes p \in \fg_{-1}=U \otimes Q$,
the formulae \eqref{e.psi-1} and \eqref{e.psi2} imply that \begin{eqnarray*}
 &&[\psi_0(a, c), \psi_{-1}(w \otimes p)](\lambda + q + u) \\
&=& \psi_0(a, c)\circ\psi_{-1}(w \otimes p)(\lambda + q +u) - \psi_{-1} (w \otimes p)\circ \psi_0(a, c)(\lambda + q + u) \\
&=& j(a, c)(\lambda(w)p + \omega(p, q) w) - \psi_{-1}(w \otimes p)(a^* \lambda + c(q) + a(u)) \\
&=& \lambda(w)c(p) + \omega(p, q)a(w) + \lambda(a(w))p - \omega(p, c(q)) w \\
&=& \psi_{-1}(a(w) \otimes p + w \otimes c(p))(\lambda + q + u) \\
&=& \psi_{-1}((a, c)\cdot (w \otimes p))(\lambda + q +u). \end{eqnarray*}
It follows that \begin{eqnarray} \label{e.ad-1} \psi_{-1}((a, c)\cdot (w \otimes p))=[\psi_0(a, c), \psi_{-1}(w \otimes p)]. \end{eqnarray}
\eqref{e.ad-2} and \eqref{e.ad-1} imply the commuting diagram in (d).
\end{proof}

\begin{theorem}\label{t.gradesp}
The isomorphism $$\psi_{-1}+ \psi_{-2}: \fg_- \to \spo(V)_{-1} + \spo(V)_{-2}$$ in Lemma \ref{l.spTanaka} extends to a graded Lie algebra isomorphism $\psi: \fg \to \spo(V)$. \end{theorem}

\begin{proof}
By Lemma \ref{l.spTanaka} (d), we obtain
 $$\psi_0 + \psi_{-1} + \psi_{-2}: \fg_0 + \fg_- \to \spo(V)_0 + \spo(V)_{-1} + \spo(V)_{-2},$$ a graded Lie algebra isomorphism.
 The adjoint representation of $\spo(V)$ and the above isomorphism induce a homomorphism $$ \chi: \spo(V)_1 \to \Hom(\fg_-, \fg)_1.$$ The fact $\chi$ is given by the adjoint representation of $\spo(V)$ implies that its image  is contained  in the first prolongation $\fg_1 \subset \Hom(\fg_-, \fg)_1$. Lemma \ref{l.spTanaka}  (b) implies that $\chi$ is injective.  But $\dim \spo(V)_1 = \dim \fg_1$ from
 Lemma \ref{l.spgrade} (2). Thus the image of $\chi$ is $\fg_1$ inducing an isomorphism
 $\psi_1: \fg_1 \to \spo(V)_1$. The same argument using Lemma \ref{l.spTanaka} (a) and Lemma \ref{l.spgrade} (1), gives an isomorphism
 $\psi_2: \fg_2 \to \spo(V)_2$. As these isomorphisms are compatible with the Lie algebra structures,
 we obtain a graded Lie algebra isomorphism $$\psi = \psi_2 + \psi_1 + \psi_0 + \psi_{-1} + \psi_{-2}:\fg \to \spo(V).$$
\end{proof}

\begin{proposition}\label{p.spgrade}
In Theorem \ref{t.gradesp}, let $m= \dim U \geq 2$ and consider the presymplectic Grassmannian
$\Go(m, V)$. Note that Assumption \ref{a.no} holds. Let $[U^*] \in \Go(m, V;0)$ be the point corresponding to $U^* \subset V$.
Then the connected subgroup of $\Sp_{\omega}(V)$ corresponding to the Lie subalgebra
$\spo(V)_2 + \spo(V)_1 + \spo(V)_0$ of $\spo(V)$ is the isotropy subgroup $P_{[U^*]}$ of the point
$[U^*]$ under the natural action of $\Sp_{\omega}(V)$ on $\Go(m,V)$. In particular, the homogeneous space $G/G^0$ associated to $(\fg_0, \fg_-)$ in the sense of Proposition \ref{p.algebraic} is biregular to $\Go(m, V; 0)$.
\end{proposition}

\begin{proof}
Since the action of $\spo(V)_2 + \spo(V)_1 + \spo(V)_0$ on $V$ preserves $U^*$, it is contained in the Lie algebra of the isotropy subgroup $P_{[U^*]}$ of $[U^*]$.
From Lemma \ref{l.spTanaka}, we can check $$\dim \spo(V)_{-1} + \dim \spo(V)_{-2} =
\dim \Go(m,V;0).$$ Thus $\spo(V)_2 + \spo(V)_1 + \spo(V)_0$ coincides with the isotropy subalgebra of $[U^*]$. The variety $\Go(m, V; 0)=\Sp_{\omega}(V)/P_{[U^*]}$ is simply connected by Proposition \ref{p.cell}. Thus the isotropy subgroup $P_{[U^*]}$ must be connected.

  Theorem \ref{t.gradesp} implies that $G$ is the adjoint group of ${\rm Sp}_{\omega}(V)$. But the center $\{\pm1\}$ of ${\rm Sp}_{\omega}(V)$ from Lemma \ref{l.Sp} acts trivially on $\Go(m,V)$. Thus $G/G^0 \cong \Go(m,V;0)$.
\end{proof}

%
%

\subsection{The case when $\omega \in \wedge^2 Q^*$ is zero}\label{ss.rank0}
In the previous subsections, we have studied certain Lie algebras associated to a presymplectic form $0 \neq \omega \in \wedge^2 Q^*.$ In this subsection, we look at the case
 when $\omega \in \wedge^2 Q^*$ is zero.
In this case, we need to modify Notation \ref{n.positive} as follows.

\begin{notation}\label{n.rank0}
Let $U$ and $Q$  be two vector spaces with $\dim U \geq 2$ and $\dim Q \geq 1$.
\begin{itemize} \item[(1)] Let $\fg_- = \fg_{-1} $ be the sum $(U \otimes Q) \oplus \Sym^2 U$.
\item[(2)] The Lie algebra
 $\Hom (U, Q) \sd (\fgl(U) \oplus \fgl(Q))$ has a faithful representation on $\fg_-$ where $\fgl(U) \oplus \fgl(Q)$ acts as in Notation \ref{n.positive} (2) with $\omega =0$ and  an element $\varphi \in \Hom(U,Q)$ annihilates  $U \otimes Q$ and sends $u \odot v \in \Sym^2 U$ to
$$  \varphi \cdot (u \odot v) = \frac{1}{2} u \otimes \varphi(v) + \frac{1}{2} v \otimes \varphi(u) \ \in U \otimes Q.$$ Let  $\fg_0 \subset {\rm End}(\fg_{-1})$ be the image of this faithful representation.
\end{itemize}
\end{notation}

\begin{definition}\label{d.B0}
In Notation \ref{n.rank0}, given $B \in \Sym^2 U^*$, $v \odot w \in \Sym^2 U \subset \fg_{-1}$, and $w \otimes p \in U \otimes Q \subset \fg_{-1},$ we define $B_{v \odot w}\in\fgl(U)\subset\fg_0$ and $B_{w \otimes p}\in\Hom(U, Q) \subset \fg_0$ in the same way as Definition \ref{d.h} (4) and (5) respectively.
\end{definition}

The following  analog of Theorem \ref{t.positive} in the setting of Notation \ref{n.rank0} is a refinement of Proposition 3.8 of \cite{FH12} and Theorem 1.1.2 of \cite{HM05}.

\begin{theorem}\label{t.FH}
The universal prolongation of $(\fg_0, \fg_-)$ in Notation \ref{n.rank0} satisfies $\fg_k =0$ for $k \geq 2$ and we can identify $\fg_1$ with $ \Sym^2 U^*$ as $\fg_0$-modules such that $B \in \Sym^2 U^* = \fg_1$ acts on $v \odot w \in \Sym^2 U \subset \fg_{-1}$ and $ w \otimes p \in U \otimes Q \subset \fg_{-1}$ by  \begin{eqnarray*}
B \cdot( v \odot w) &=& B_{v \odot w} \in \fgl(U) \subset  \fg_0 \\
B \cdot (w \otimes p) &=& B_{w \otimes p} \in \Hom(U, Q) \subset \fg_0
\end{eqnarray*} where $B_{v \odot w}$ and $B_{w \otimes p}$ are as in Definition \ref{d.h} (4) and (5).
\end{theorem}

\begin{proof}
By Proposition 3.8 of \cite{FH12} and Theorem 1.1.2 of \cite{HM05} (equivalently, Theorem 2.3 of \cite{FH12}), there is an isomorphism $\fg_1\cong\Sym^2 U^*$ as vector spaces and $\fg_k=0$ for all $k \geq 2$. Consider the linear map $$  \Sym^2 U^* \rightarrow \Hom(\fg_{-1}, \fg_0) $$ which sends $ B \in \Sym^2 U^*$ to $$ (v \odot w \mapsto B_{v \odot w}, \  w \otimes p \mapsto B_{w \otimes p}). $$  It is straight-forward to check that this is an injective $\fg_0$-module homomorphism and its image is contained in $\fg_1$.   Then the fact $\dim\fg_1=\dim\Sym^2 U^*$ implies that this is a $\fg_0$-module isomorphism between $\Sym^2 U^*$ and $\fg_1$.
\end{proof}

The following definition is an analog of Definition \ref{d.VUQ} in the setting of Notation \ref{n.rank0}.

\begin{definition}\label{d.Vrank0}
Define $V := U^* \oplus Q \oplus U$ and a presymplectic form $\omega \in \wedge^2 V^*$ by
\begin{itemize} \item[(1)] $Q = {\rm Null}_{\omega};$ \item[(2)] $\omega(\lambda, u) = \lambda(u)$ for $\lambda \in U^*$ and $u \in U$; \item[(3)] $\omega (U, U) =0$; and
\item[(4)] $\omega(U^*, U^*) =0$. \end{itemize}
Define a grading on $V$,  $$V= V_{\frac{1}{2}} \oplus V_{-\frac{1}{2}}  \mbox{ by }
V_{-\frac{1}{2}} = U \oplus Q\mbox{  and  }V_{\frac{1}{2}} = U^*.$$
This induces a grading $\spo(V) = \spo(V)_1 \oplus \spo(V)_0 \oplus \spo(V)_{-1}$. \end{definition}

The following is an analog of Lemma \ref{l.spgrade} in the setting of Notation \ref{n.rank0}.

\begin{lemma}\label{l.rank0grade}
In the setting of Notation \ref{n.rank0} and Definition \ref{d.Vrank0}, we have the following. \begin{itemize}
\item[(1)]
$\spo(V)_1$ consists of elements $\gamma \in \Hom(U, U^*) \subset \Hom(V_{-\frac{1}{2}}, V_{\frac{1}{2}})$ satisfying
$$(\gamma (u) ) (v) = (\gamma(v))(u) \mbox{ for all } u, v \in U.$$
\item[(2)] $\spo(V)_{-1}$ consists of pairs $$(\alpha, \beta) \in \Hom(U^*, Q) \oplus  \Hom(U^*, U)$$ where $\alpha$ is arbitrary and $\beta$ satisfies $$\omega(\beta(\lambda), \mu) = \omega(\beta(\mu), \lambda) \mbox{ for all } \lambda, \mu \in U^*.$$ \end{itemize} \end{lemma}

\begin{proof}
 Let $$(\gamma, \eta) \in \Hom(U, U^*) \oplus  \Hom(Q, U^*)$$ be an element of $\spo(V)_1$. Then for any $u, v \in U$ and $q \in Q$,
$$\omega(\gamma(u) + \eta(q), v) = \omega(\gamma(v), u+ q) = \omega(\gamma(v), u).$$
This shows that $\eta=0$ and $\omega(\gamma (u), v) = \omega(\gamma(v), u)$, proving (1).
(2) is straight-forward.  \end{proof}

The following is an analog of Lemma \ref{l.spTanaka}  in the setting of Notation \ref{n.rank0}.

\begin{lemma}\label{l.sprank0}
In the setting of Definition \ref{d.Vrank0} and  Lemma \ref{l.rank0grade}, we have the following. \begin{itemize} \item[(a)] If $\gamma \in \spo(V)_1$ satisfies $[\gamma, \spo(V)_{-1}]=0$, then $\gamma =0$. \item[(b)] Define $\psi_{-1}: \fg_{-1}=  (U \otimes Q) \oplus \Sym^2 U \to \spo(V)_{-1}$ by sending $w \otimes q \in U \otimes Q$ and $u \odot v \in \Sym^2 U$ to  $$\psi_{-1}(w\otimes q + u \odot v)
= (\alpha_{w \otimes q}, \beta_{u \odot v}) \in \Hom(U^*, Q) \oplus \Hom(U^*, U)$$
 where $$\alpha_{w \otimes q}(\lambda) := \lambda(w) q \mbox{ and } \beta_{u \odot v}(\lambda) := \frac{1}{2} \lambda(u) v + \frac{1}{2} \lambda(v) u \mbox{ all } \lambda \in U^*.$$ Then $\psi_{-1}$ is an isomorphism.
 \item[(c)]
    The natural action of $\fg_0 = \Hom(U,Q) \sd (\fgl(U) \oplus \fgl(Q))$ on $V = U^* \oplus Q \oplus U$ gives an injective homomorphism $\psi_0: \fg_0 \to \spo(V)_0$. This is an isomorphism and satisfies the commuting diagram
    $$\begin{array}{ccc} \fg_0 & \stackrel{\psi_0}{\longrightarrow} & \spo(V)_0 \\ \downarrow & & \downarrow \\ \End(\fg_-) & \stackrel{\psi_{-1}}{\longrightarrow} & \End(\spo(V)_{-1}) \end{array} $$ where the first vertical arrow is the inclusion in
    Notation \ref{n.rank0}, the second vertical arrow is the adjoint action and the lower horizontal arrow is the isomorphism induced by $\psi_{-1}$ in (b).
\end{itemize}    \end{lemma}

\begin{proof}
The condition for $\gamma$ in (a) says that  for any $(\alpha, \beta) \in \spo(V)_{-1}$,
$$\gamma \circ (\alpha, \beta) = (\alpha, \beta) \circ \gamma.$$
Evaluating it at $\lambda \in U^*$, we obtain
$\gamma (\beta(\lambda)) =0.$ Given any $u \in U$, we can choose some $\beta$ as in Lemma \ref{l.rank0grade} (2)  and
$\lambda \in U^*$ such that $u = \beta(\lambda)$. Thus $\gamma =0$, proving (a).

(b) is immediate from Lemma \ref{l.rank0grade} (2).

The injective homomorphism $\psi_0$ in  (c) is surjective because
  we can check $\dim \fg_0 = \dim \spo(V)_0$ from Lemma \ref{l.rank0grade} and our knowledge of  $\dim \spo(V).$ To prove (c), it remains to check the commuting diagram.

That the action of  $\fgl(U) \times \fgl(Q) \subset \fg_0$
on $\fg_{-1}$ corresponds via $\psi_{-1}$ to the one given by the adjoint representation of $\spo(V)$
can be proved in  the same way as Lemma \ref{l.spTanaka} (d). So let us just consider the action of an element $\varphi \in \Hom(U,Q) \subset \fg_0.$ The element $\psi_0(\varphi) \in \spo(V)_0 \subset \fgl(V)$ acts on $(\lambda + p + t) \in U^* \oplus Q \oplus U = V$  by
$$\psi_0(\varphi) (\lambda+p +t ) =  \varphi(t).$$
For $ w \otimes q + u \odot v \in  \fg_{-1}$, the definition of $\psi_{-1}$ in (b) implies that
\begin{eqnarray*}
 & & \psi_0(\varphi) \circ \psi_{-1}(w\otimes q + u \odot v) (\lambda + p + t) \\
 & = &
 \psi_0(\varphi) \left( \alpha_{w \otimes q} (\lambda) + \beta_{u \odot v} (\lambda) \right) \\
 & =& \psi_0(\varphi) ( \lambda(w) q + \frac{1}{2} \lambda(u) v + \frac{1}{2} \lambda(v) u ) \\
 &=& \frac{1}{2} \lambda(u) \varphi(v) + \frac{1}{2} \lambda(v) \varphi(u), \end{eqnarray*}
 and \begin{eqnarray*} & &
 \psi_{-1}(w \otimes q + u \odot v) \circ \psi_0(\varphi) (\lambda + p + t) \\
 &=&  \psi_{-1}(w \otimes q + u \odot v)  (\varphi (t)) \\
 &=& (\alpha_{w \otimes q} + \beta_{u \odot v}) (\varphi(t))\ = \ 0, \end{eqnarray*} which gives
 \begin{equation}\label{e.com}
 [\psi_0(\varphi), \psi_{-1}(w\otimes q + u \odot v)] (\lambda + p + t) =
 \frac{1}{2} \lambda(u) \varphi(v) + \frac{1}{2} \lambda(v) \varphi(u).\end{equation}
 On the other hand,
 \begin{eqnarray*} & &
 \psi_{-1}\left( \varphi \cdot ( w \otimes q + u \odot v) \right) (\lambda + p + t) \\ & =&
 \psi_{-1} \left( \frac{1}{2} u \otimes \varphi(v) + \frac{1}{2} v \otimes \varphi(u) \right) (\lambda + p + t) \\ &=&
 \left(\frac{1}{2} \alpha_{u \otimes \varphi(v)} + \frac{1}{2} \alpha_{v \otimes \varphi(u)} \right) (\lambda + p + t) \\ & =& \frac{1}{2} \alpha_{u \otimes \varphi(v)}(\lambda) + \frac{1}{2} \alpha_{v \otimes \varphi(u)}(\lambda) \\ &=& \frac{1}{2} \lambda(u) \varphi(v) + \frac{1}{2} \lambda(v) \varphi(u).\end{eqnarray*}
Since this is equal to (\ref{e.com}), we have
 $$ [\psi_0(\varphi), \psi_{-1}(w\otimes q + u \odot v)] = \psi_{-1}\left( \varphi \cdot ( w \otimes q + u \odot v) \right), $$
 which   implies the commuting diagram in (c).
  \end{proof}

The analog of Theorem \ref{t.gradesp} is the following.

\begin{theorem}\label{t.sprank0}
By Lemma \ref{l.sprank0}, we have an isomorphism of graded Lie algebras
$\spo(V)_0 + \spo(V)_{-1} \cong \fg_0 + \fg_{-1}$. It extends to an isomorphism of graded Lie algebras $\spo(V) \cong \fg$, realizing $\spo(V)$ as the universal prolongation of $(\fg_0, \fg_{-1})$. \end{theorem}

\begin{proof}
The proof is a direct modification of the proof of Theorem \ref{t.gradesp}.
The adjoint representation induces a homomorphism
$\chi: \spo(V)_1 \to \Hom(\fg_{-1}, \fg_0)$ whose image lies in $\fg_1$.
Lemma \ref{l.sprank0} (a) implies that $\chi$ is injective and then $\dim \spo(V)_1 = \dim \fg_1$ by Theorem \ref{t.FH} shows that $\chi$ is an isomorphism. This gives the extended isomorphism of Lie algebras $\spo(V) \cong \fg$. \end{proof}

We omit the proof of the following proposition, which is essentially the same as the proof of Proposition \ref{p.spgrade}, just applying  Theorem \ref{t.sprank0} instead of Theorem \ref{t.gradesp}.

\begin{proposition}\label{p.spgraderank0}
In Theorem \ref{t.sprank0}, let $m= \dim U$ and consider the presymplectic Grassmannian
$\Go(m, V)$. Note that Assumption \ref{a.no} holds. Let $[U^*] \in \Go(m, V;0)$ be the point corresponding to $U^* \subset V$.
Then the connected subgroup of $\Sp_{\omega}(V)$ corresponding to the Lie subalgebra
$\spo(V)_1 + \spo(V)_0$ of $\spo(V)$ is the isotropy subgroup $P_{[U^{*}]}$ of the point
$[U^*]$ under the natural action of $\Sp_{\omega}(V)$ on $\Go(m,V)$. In particular, the homogeneous space $G/G^0$ associated to $(\fg_-, \fg_0)$ in the sense of Proposition \ref{p.algebraic} is biregular to $\Go(m, V; 0)$.
\end{proposition}
%

\section{Existence of Cartan connections}\label{s.SAF}

\subsection{Preliminaries on multi-filtrations}

\begin{definition}\label{d.multifil}
Let $A$ be a complex vector space. A collection of subspaces $$\{ A_{i_1, \ldots, i_k}, \ i_1, \ldots, i_k \in \Z\}$$ is a {\em $k$-tuple filtration} of $A$ if \begin{itemize}
\item[(1)] $A_{-n_1, \ldots, -n_k} = 0, \ A_{n_1, \ldots, n_k} =A$ for some $n_1, \ldots, n_k \in \N$; and
\item[(2)] $A_{i_1, \ldots, i_k} \subset A_{j_1, \ldots, j_k}$ when $i_1 \leq j_1, \ldots, i_k \leq j_k$. \end{itemize}
    If we do not need to specify the number $k$, we will call it a {\em multi-filtration} of $A$. If $k=1$, it is a filtration of $A$.
    The quotient space
    $${\rm gr}(A)_{i_1, \ldots, i_k} := A_{i_1, \ldots, i_k} / (\sum_{\ell=1}^k A_{i_1, \ldots, i_{\ell-1}, i_{\ell} -1, i_{\ell+1}, \ldots, i_k} )$$
is called a {\em graded object} associated  to the multi-filtration.
 \end{definition}

 \begin{definition}\label{d.reductive}
 Let $G_0$ be a connected algebraic group with Lie algebra $\fg_0$.  Let $R \subset G_0$ be the nilradical with the Lie algebra  ${\mathfrak r} \subset \fg_0$ such that
  $G_0/R$ and $\fg_0/{\mathfrak r}$ are reductive. For a $G_0$-module $A$,
  a multi-filtration $\{A_{i_1, \ldots, i_k} \}$ is said to be {\em $\fg_0$-reductive} if \begin{itemize} \item[(i)] each $A_{i_1, \ldots, i_k}$ is a $\fg_0$-submodule of $A$; and \item[(ii)] the ${\mathfrak r}$-action on each graded object ${\rm gr}(A)_{i_1, \ldots, i_k}$ is trivial, i.e., $${\mathfrak r} \cdot A_{i_1, \ldots, i_k}
  \subset \sum_{\ell =1}^k A_{i_1, \ldots, i_{\ell-1}, i_{\ell} -1, i_{\ell +1}, \ldots, i_k}.$$ \end{itemize}
  This implies that each graded object is a representation of the reductive group $G_0/R$.
  \end{definition}

  The following  lemma is straight-forward.

  \begin{lemma}\label{l.reductive}
Let $\{A_{i_1, \ldots, i_k} \}$  be a $\fg_0$-reductive $k$-tuple filtration of a $G_0$-module $A$.  \begin{itemize}
\item[(1)] For any $G_0$-submodule $B \subset A$, the intersection
$$B_{i_1, \ldots, i_k} := B \cap A_{i_1, \ldots, i_k}$$
gives a $\fg_0$-reductive $k$-tuple filtration of $B$ with a natural inclusion of  graded objects $${\rm gr}(B)_{i_1, \ldots, i_k} \subset {\rm gr}(A)_{i_1, \ldots, i_k}$$ as $G_0/R$-modules.
\item[(2)] For any surjective $G_0$-module homomorphism $f: A \to B$, the image
$$B_{i_1, \ldots, i_k} := f(A_{i_1, \ldots, i_k})$$
gives a $\fg_0$-reductive $k$-tuple filtration of $B$ with a natural surjection between graded objects $${\rm gr}(A)_{i_1, \ldots, i_k} \to {\rm gr}(B)_{i_1, \ldots, i_k}$$ as $G_0/R$-modules.
\item[(3)] If $\{ B_{j_1, \ldots, j_{\ell}}\}$ is a $\fg_0$-reductive $\ell$-tuple filtration of a $G_0$-module $B$, then the tensor product
    $$(A \otimes B)_{i_1, \ldots, i_k, i_{k+1}, \ldots, i_{k+\ell}}:= A_{i_1, \ldots, i_k} \otimes B_{i_{k+1}, \ldots, i_{k+\ell}}$$ gives a $\fg_0$-reductive $(k+ \ell)$-tuple filtration of $A \otimes B$ with a natural isomorphism of graded objects $${\rm gr}(A \otimes B)_{i_1, \ldots, i_{k + \ell}} \cong {\rm gr}(A)_{i_1, \ldots, i_k} \otimes {\rm gr}(B)_{i_{k+1}, \ldots, i_{k + \ell}}$$ as $G_0/R$-modules.
\item[(4)] The collection of subspaces
$$(A^*)_{i_1, \ldots, i_k}:=\{f \in A^* \mid f(A_{-i_1, \ldots, -i_k})=0\}$$
gives a $k$-tuple filtration of $A^*$.
\item[(5)] When $k=1$,  the filtration of $A^*$  in (4) is $\fg_0$-reductive and there is a natural isomorphism of graded objects $${\rm gr}(A^*)_{i} \cong ({\rm gr}(A)_{-i+1})^{*}$$ as $G_0/R$-modules.
    \end{itemize} \end{lemma}

\begin{remark}
Both assertions in Lemma \ref{l.reductive} (5) may fail when $k \geq 2$.
\end{remark}

 \begin{lemma}\label{l.trivial}
Let $M$ be a complex manifold and let $E_0 \to M$ be a $G_0$-principal bundle. For a $G_0$-module $A$, denote by $E_0 \times^{G_0}\!A$ the vector bundle associated to $A$.
 Let $\{ A_{i_1, \ldots, i_k} \}$ be a $\fg_0$-reductive multi-filtration.
\begin{itemize}
\item[(i)] If $H^0(M, E_0 \times^{G_0} {\rm gr}(A)_{i_1, \ldots, i_k}) =0$ for any
$i_1, \ldots, i_k \in \Z$, then $H^0(M, E_0 \times^{G_0} B) =0$ for any surjective $G_0$-module homomorphism $A \to B$.
\item[(ii)] If there exists $(j_1, \ldots, j_k) \in \Z^k$ such that $$H^0(M, E_0 \times^{G_0} {\rm gr}(A)_{i_1, \ldots, i_k}) =0 \mbox{ for all }
(i_1, \ldots, i_k) \in \Z^k \setminus \{(j_1, \ldots, j_k)\},$$
then for any surjective $G_0$-module homomorphism $A \to B$, there exists a natural injective homomorphism $$H^0(M, E_0 \times^{G_0}\! B) \rightarrow H^0(M, E_0 \times^{G_0} {\rm gr}(B)_{j_1, \ldots, j_k})$$
in terms of the $\fg_0$-reductive multi-filtration on $B$ from Lemma \ref{l.reductive} (2), which arises from the composition
$$H^0(M, E_0 \times^{G_0} \! B) \leftarrow H^0(M,E_0 \times^{G_0} B_{j_1, \ldots, j_k}) \to H^0(M, E_0 \times^{G_0} {\rm gr}(B)_{j_1, \ldots, j_k})$$
where the first arrow is an isomorphism induced by the inclusion $B \supset B_{j_1, \ldots, j_k}$ and the second arrow is an injective homomorphism induced by the quotient $B_{j_1, \ldots, j_k}\to {\rm gr}(B)_{j_1,\ldots, j_k}$.
 \end{itemize}
 \end{lemma}

\begin{proof}
For (i), if $A= B$, the result follows from the standard exact sequence argument.
So to prove (i), it suffices to show $$H^0(M, E_0 \times^{G_0} {\rm gr}(B)_{i_1, \ldots, i_k}) =0 \mbox{ for any }
i_1, \ldots, i_k \in \Z$$ in terms of the $\fg_0$-reductive multi-filtration on $B$ from Lemma \ref{l.reductive} (2). But the $G_0$-module surjection $${\rm gr}(A)_{i_1, \ldots, i_k}
\to {\rm gr}(B)_{i_1, \ldots, i_k}$$ is a $G_0/R$-module surjection.
As $G_0/R$ is reductive, we have an injection as $G_0/R$-modules $${\rm gr}(B)_{i_1, \ldots, i_k}
\to {\rm gr}(A)_{i_1, \ldots, i_k}$$ which induces an injection
$$H^0(M, E_0 \times^{G_0} {\rm gr}(B)_{i_1, \ldots, i_k} )
\subset H^0(M, E_0 \times^{G_0} {\rm gr}(A)_{i_1, \ldots, i_k}).$$ Since the right hand side vanishes by the assumption, so does the left hand side. This proves (i).

(ii) is also straight-forward if $A=B$. But  the  argument
we just used in the proof of (i) shows that the assumption in (ii) implies
$$H^0(M, E_0 \times^{G_0} {\rm gr}(B)_{i_1, \ldots, i_k}) =0 \mbox{ for all }
(i_1, \ldots, i_k) \in \Z^k \setminus \{(j_1, \ldots, j_k)\}.$$ Thus we are done. \end{proof}

\subsection{SAF-condition and vanishing results}

\begin{definition}\label{d.MU}
Fix a base point $o \in \BP^1$.
Let $M$ be a complex manifold and let $\sU$ be a vector bundle on $M$ of rank $m \geq 2$.
 We say that  $(M, \sU)$ {\em satisfies SAF-condition}   (SAF standing for Single Ample Factor) if  there exist \begin{itemize} \item[(1)] a nonempty Zariski-open subset $M' \subset M$; \item[(2)]  a nonempty Euclid-open subset $\sU^o_x \subset \sU_x \setminus \{0\}$  for each point $x \in M';$ \item[(3)]  a holomorphic map $f_u: \BP^1 \to M$ for each $u \in \sU^o_x$ such that
\begin{itemize}
\item[(i)] $f_u(o) = x$;
\item[(ii)] the vector bundle $f_u^* \sU$ on $\BP^1$ has a unique ample line subbundle
 $\sL_u$ with a  splitting $ f_u^*\sU \cong \sL_u \oplus \sO_{\BP^1}^{m-1}$; and
 \item[(iii)] the fiber of $\sL_u$ at $o$ is
$\C u \subset \sU_x$.
 \end{itemize} \end{itemize}
Given such $(M, \sU)$, a vector bundle $\sV$ on $M$ is said to be {\em of trivial type} (resp. {\em of nonpositive type}) if we can choose $M'$ in (1)  and $\sU^o_x$ for each $x \in M'$ in (2) such that for any $f_u: \BP^1 \to M$ in (3), the pull-back bundle $f_u^* \sV$  on $\BP^1$ has a splitting $\sO(a_1) \oplus \cdots \sO(a_d), d = {\rm rank}(\sV)$ with $a_1 = \cdots = a_d =0$ (resp. $0 \geq a_1 \geq \cdots \geq a_d$).
For example, the dual bundle $\sU^*$ is of nonpositive type.
\end{definition}

\begin{proposition}\label{p.MU}
Let $(M, \sU)$ be a pair satisfying SAF-condition. Let $\sV$ be a vector bundle of nonpositive type on $M$. Then
\begin{itemize}
\item[(1)] $\Hom(\bigotimes^k \sU, \sV \otimes \Sym^{\ell} \sU) =0$ for $0 \leq \ell < k$; and   \item[(2)] $\Hom(\wedge^2 \sU, \sV \otimes  \Sym^2 \sU) =0. $\end{itemize}
             \end{proposition}
\begin{proof}
 Given $\varphi \in \Hom(\bigotimes^k \sU, \sV \otimes \Sym^{\ell} \sU)$,
let $u_1, \ldots, u_k \in \sU^o_x$ be $k$ general vectors  at a point $x \in M'.$ We may assume that $u_1, \ldots, u_k$ are
pairwise linearly independent.  Let $f_{u_i}: \BP^1 \to M$ and $\sL_{u_i} \subset f_{u_i}^* \sU$ be as in Definition \ref{d.MU}. From the splitting type $f_{u_i}^* \sU \cong \sL_{u_i} \otimes \sO^{m-1}$, we can find sections $$s^1_{i}, \cdots, s^k_{i} \in H^0(\BP^1, f_{u_i}^* \sU)$$ such that $s^j_{i}(o) = u_j$ for $1 \leq j \leq k$ and
$s^i_{i}$ is a section of $\sL_{u_i}$. Since $\sV$ is of nonpositive type and  $f_{u_i}^* \sU \cong \sL_{u_i} \otimes \sO^{m-1},$ we see  that
$\varphi(s^1_{i}, \ldots, s^k_{i})$ is a section of the subbundle $$ f_{u_i}^*\sV \otimes (\sL_{u_i} \odot \Sym^{\ell-1} f_{u_i}^* \sU) \subset f_{u_i}^*(\sV \otimes \Sym^{\ell} \sU).$$ Thus for each $1 \leq i \leq k$, $$\varphi_x(u_1, \cdots, u_k) \in  \sV_x \otimes (u_i \odot \Sym^{\ell-1} \sU_x).$$ It follows that $$\varphi_x(u_1, \cdots, u_k) \in  \bigcap_{i=1}^k \left( \sV_x \otimes (u_i \odot \Sym^{\ell-1} \sU_x) \right).$$
Since $u_1, \ldots, u_k$ are pairwise linearly independent with $k > \ell$, the intersection on the right must be zero. This shows $\varphi =0$, proving (1).

Given $\varphi \in \Hom(\wedge^2 \sU, \sV \otimes \Sym^2 \sU), $ choose general two (hence linearly independent) vectors $u, v \in \sU^o_x$ at $x \in M'$.  As above, we can choose two sections $s_u \in H^0(\BP^1, \sL_u)$ and $s_v \in H^0(\BP^1, f_u^* \sU)$ such that $s_u(o) = u$ and $s_v(o) = v$.
Then $f_u^* \sU \cong \sL_u \oplus \sO_{\BP^1}^{m-1}$ implies that $\varphi(s_u \wedge s_v)$ must be a section of the subbundle $$ f_u^* \sV \otimes (\sL_u \odot f_u^* \sU) \subset f_u^* \sV \otimes \Sym^2 f_u^* \sU.$$
Thus $$\varphi_x(u \wedge v) \in \sV_x \otimes (u \odot \sU_x).$$
Since this must hold with $u$ and $v$ switched, we have
$$\varphi_x(u \wedge v) \in (\sV_x \otimes (u \odot \sU_x)) \cap (\sV_x \otimes (v \odot \sU_x)) = \sV_x \otimes (u \odot v).$$  By continuity, this holds for all $u, v \in \sU_x$. But
a linear homomorphism $\varphi_x \in \Hom( \wedge^2 \sU_x,
\sV_x \otimes \Sym^2 \sU_x)$ satisfying $$\varphi_x(u \wedge v) \in \sV_x \otimes (u \odot v)$$ for all $u, v \in \sU_x$ must be zero.
In fact, by the linearity in $u$ and $v$, such a homomorphism should be given by a fixed element $c \in \sV_x$ as $\varphi_x(u \wedge v) = c \otimes (u \odot v)$. Then the anti-symmetry on the
left hand side and the symmetry on the right hand side imply $\varphi_x = c =0$. This proves (2).
\end{proof}

 \begin{proposition}\label{p.symbol}
 Let $(M,\sU)$ be a pair satisfying SAF-condition and let $\sQ$ be a vector bundle of trivial type on $M$.
 Let $\varphi\in \Hom(\wedge^2( \sU \otimes \sQ), \Sym^2 \sU)$ be a homomorphism of vector bundles. Then there exist  a nonnegative integer $k$ and a section $\omega \in H^0(M, \wedge^2 \sQ^*)$ such
 that \begin{itemize}
  \item[(i)] for each $x \in M$,
 $u, v \in \sU_x$, and $p,q \in \sQ_x$, we have $$\varphi_x (u \otimes p, v \otimes q)
 = \omega_x(p,q) u \odot v;$$
 \item[(ii)] the subset $M_o := \{ x \in  M, \ {\bf n}_{\omega_x} = k \},$ where ${\bf n}_{\omega_x}$ is the nullity of the presymplectic vector space $(\sQ_x, \omega_x)$,  is nonempty and Zariski-open in $M$;
 \item[(iii)]  the pair $(M_o, \sU|_{M_o})$ satisfies SAF-condition; and
          \item[(iv)] both  the subbundle $\sN \subset \sU|_{M_o}$ defined by the null-spaces of $\omega_x, x \in M_o,$ and the quotient bundle $\sQ^{\flat}:= (\sQ|_{M_o})/\sN$ are of trivial type.
               \end{itemize}
 \end{proposition}

 \begin{proof}
 Since $\Sym^2 \sQ^*$ is of trivial type, Proposition \ref{p.MU} (2) gives $$\Hom(\wedge^2 \sU \otimes \Sym^2 \sQ, \Sym^2 \sU) = \Hom(\wedge^2 \sU, \Sym^2 \sQ^* \otimes \Sym^2 \sU) =0.$$ This implies $\varphi \in \Hom(\Sym^2 \sU \otimes \wedge^2 \sQ, \Sym^2 \sU).$

For two general $u, v \in \sU^o_x$ and any $q, p \in Q_x$ at a  point $x \in M'$, we have $f_u : \BP^1 \to M$ and
$\sL_u \subset f_u^* \sU$ described in Definition \ref{d.MU} such that $f_u^* \sQ$ is a trivial bundle. Choose sections $$s_u \in H^0(\BP^1, \sL_u), \  s_v \in H^0(\BP^1, f_u^* \sU) \mbox{ and } s_p, s_q \in H^0(\BP^1, f_u^* \sQ)$$ such that $s_u(o) = u, s_v(o) =v, s_p(o) = p, \mbox{ and } s_q(o) = q$. Then
 $$\varphi((s_u \otimes s_p) \wedge (s_v \otimes s_q)) \in H^0( \BP^1, \sL_u \odot f_u^* \sU),$$ which implies $\varphi_x( (u \otimes p) \wedge (v \otimes q)) \in u \otimes \sU_x.$
Applying the same argument  with $u$ and $v$ switched, we have
 $$\varphi_x((u\otimes p) \wedge (v \otimes q)) \in \C u \odot v.$$
By continuity, this must hold for any $u, v \in \sU_x$ and any $x \in M$.
 It follows that for each $x \in M$, there exists a unique $\omega_x \in \wedge^2 \sQ_x^*$ such that
 $$\varphi_x((u\otimes p) \wedge (v \otimes q)) = \omega_x(p,q) \ u \odot v.$$
 This defines $\omega \in H^0(M, \wedge^2 \sQ^*)$ satisfying (i).

The existence of $k \geq 0$ satisfying (ii) is automatic.

As $\sQ$ is of trivial type, if
${\rm \bf n}_{\omega_x} = k$ at a point $x \in M'$, then it holds at every point of $f_u(\BP^1) \subset M$
for  $u \in \sU^o_x$, implying $f_u(\BP^1) \subset M_o$. (iii) follows.

As $\sQ$ is of trivial type, so is $\wedge^2 \sQ^*$. Thus
$f_u^* \sN$ is a trivial subbundle of $f_u^* \sQ$, which proves (iv). \end{proof}

\begin{proposition}\label{p.grvanish} In the setting of Proposition \ref{p.symbol}, let us assume that
$M = M_o$. Denote by $sp_{\omega^{\flat}}(\sQ^{\flat})$  the vector bundle whose fiber at $x \in M$ is the symplectic Lie algebra $\fsp_{\omega^{\flat}_x}(\sQ^{\flat}_x)$.  Define
$$ \begin{array}{lll}
&  & \sB_0 := \Sym^2 \sU \\ & & \sB_1:= \sU \otimes \sQ^{\flat} \\
\sA_1:= \wedge^2 \Sym^2 \sU &  & \sB_2:= \sU \otimes \sN \\ \sA_2:= \Sym^2 \sU \otimes \sU \otimes \sN  &  & \sB_3 := \sU^* \otimes \sU
 \\ \sA_3:= \Sym^2 \sU \otimes \sU \otimes \sQ^{\flat} & \mbox{ and } & \sB_4 := sp_{\omega^{\flat}}(\sQ^{\flat})\\
\sA_4:= \wedge^2 (\sU \otimes \sN) &  & \sB_5:= \sN^* \otimes \sN \\
 \sA_5:= \sU \otimes \sN \otimes \sU \otimes \sQ^{\flat} &  & \sB_6:= Hom(\sQ^{\flat}, \sN)
  \\ \sA_6:= \wedge^2 (\sU \otimes \sQ^{\flat}) &  & \sB_7:= \sU^* \otimes \sQ^{\flat}\\ &  &
\sB_8:= \sU^* \otimes \sN   \\ &  &  \sB_9:= \Sym^2 \sU^*. \end{array} $$
Then $\Hom( \sA_i, \sB_j) = 0$ in one of the following two cases.
\begin{itemize} \item[(a)] $1 \leq i \leq 3$ and $0 \leq j \leq 9;$ \item[(b)] $4 \leq i \leq 6$ and $1 \leq j \leq 9$. \end{itemize}
\end{proposition}

\begin{proof}
In each of the cases listed in (a) or (b), we can rewrite
$\Hom(\sA_i, \sB_j)$ in the form $\Hom(\sA', \sV \otimes \sB')$ for a quotient bundle $\sA'$ of
$\bigotimes^k \sU$, a vector bundle $\sV$ of nonpositive type and  a vector bundle $\sB'= \Sym^{\ell} \sU$ with $\ell < k$.
For example,
$$\Hom(\sA_3, \sB_3) = \Hom(\Sym^2 \sU \otimes \sU \otimes \sQ^{\flat}, \sU^* \otimes \sU)
= \Hom(\sA', \sV \otimes \sB')$$ with $\sA' = \Sym^2 \sU \otimes \sU,$
$\sV = (\sQ^{\flat})^* \otimes \sU^*,$ and $\sB' = \sU$.
Thus the vanishing follows from Proposition \ref{p.MU} (1). \end{proof}

\subsection{Cartan connection in the setting of Subsection \ref{ss.positive}}

\begin{notation}\label{n.Cpositive}
Let $(\fg_0, \fg_-)$ be as in Notation \ref{n.positive} and let $G_0 \subset {\rm gr}\Aut(\fg_-)$ be the connected algebraic group whose Lie algebra is $\fg_0 \subset {\rm gr}\aut(\fg_-)$. The kernel of the natural surjective homomorphism
$$ \iota:  \GL(U) \times \Sp_{\omega}(Q) \to G_0$$ arising from $\fgl(U) \oplus \spo(Q) \cong \fg_0$ is $\{ \pm 1\}$. Let $\fg$ be the universal prolongation of $(\fg_0, \fg_-)$ and let $G^0 \subset G \subset \GL(\fg)$ be the connected algebraic subgroups in Proposition \ref{p.spgrade}.

Let $M$ be a complex manifold with a filtration $(F^{-k}M, k \in \Z)$ of type $\fg_-$.
Assume that we are given a $G_0$-structure subordinate to the filtration, i.e., a $G_0$-principal subbundle $E_o \subset {\rm gr}{\rm Fr}(M)$. By the natural representation of $G_0$ on
$\fg_{-1}$, we have  the associated vector bundle $\sD:= E_0 \times^{G_0} \fg_{-1}$ with a natural inclusion $\sD \subset TM$.\end{notation}

\begin{remark}\label{r.kernel}
Because of the kernel of $\iota$, the vector spaces $U$ and $Q$ do not have natural $G_0$-module structures. This is why we need to require the condition (1) in the next theorem. On the other hand, vector spaces like $U \otimes Q$ and $\Sym^2 Q$ have natural $G_0$-module structures because ${\rm Ker}(\iota)$ acts trivially on them. \end{remark}

\begin{theorem}\label{t.Cpositive}
In the setting of Notation \ref{n.Cpositive}, assume that there exist vector bundles
$\sU$ and $\sQ$ on $M$ such that
 \begin{itemize} \item[(1)] there are  isomorphisms  $\sD \cong \sU \otimes \sQ$ and  $TM/\sD \cong \Sym^2 \sU$ as vector bundles on $M$; \item[(2)] the pair $(M, \sU)$ satisfies SAF-condition and $\sQ$ is of trivial type. \end{itemize} Then there exists a  Cartan connection $(E \to M, \omega)$ of type $(G, G^0)$ on $M$ related to $E_0$ in the way described in Theorem \ref{t.Cartan}. \end{theorem}

\begin{proof}
 By the condition (2), we can apply Proposition \ref{p.symbol} with $\varphi$ given by the Frobenius bracket tensor of the distribution $\sD \subset TM$. Since we are assuming that the symbol of $\sD$ is isomorphic to $\fg_-$ at every point of $M$, we have $M=M_o$. Consequently, we can apply Proposition \ref{p.grvanish} to  the vector bundles arising from $\sU$ and $\sQ$ on $M$.

We use this to check
$$H^0(M, C^{\ell, 2}(E_0) / \partial (C^{\ell,1}(E_0)) ) =0 \mbox{ for any }  \ell \geq 1,$$ the condition in Theorem \ref{t.Cartan}, from which the existence of the desired Cartan connection follows.

Note that the nilradical of $\fg_0$ is ${\mathfrak r} = \Hom(Q^{\flat}, {\rm Null}_{\omega})$.
We will use the following $\fg_0$-reductive filtrations
\begin{eqnarray*} \fg_{-2}&=& \Sym^2 U  \supset 0 \\
 \fg_{-1}&=&  U \otimes Q \supset U \otimes {\rm Null}_{\omega} \supset 0\\
\fg_0&=& \spo(Q) \oplus \fgl(U) \supset \spo(Q) = {\mathfrak r} \sd (\fgl({\rm Null}_{\omega}) \oplus \fsp_{\omega^{\flat}}(Q^{\flat}))
\\ & & \supset {\mathfrak r} \sd \fgl({\rm Null}_{\omega}) \supset {\mathfrak r} \supset 0\\
\fg_1 &=&\Hom(U, Q) \supset  \Hom(U, {\rm Null}_{\omega})  \supset 0 \\
\fg_2&=& \Sym^2U^* \supset 0.\end{eqnarray*}
Denoting the unipotent radical of $G_0$ by $R$,
the associated graded objects are the following $G_0/R$-modules
\begin{eqnarray*} \fg_{-2} &:& B_0:= \Sym^2 U, \\
\fg_{-1} &:& B_1:=U \otimes Q^{\flat}, \ B_2:= U \otimes {\rm Null}_{\omega}, \\
\fg_0 &:& B_3:= \fgl(U), \ B_4:= \fsp_{\omega^{\flat}}(Q^{\flat}), \\ & &  B_5:= \fgl({\rm Null}_{\omega}), \ B_6:= \Hom(Q^{\flat}, {\rm Null}_{\omega}), \\
\fg_1 &:& B_7:= \Hom(U, Q^{\flat}), \ B_8:=  \Hom(U, {\rm Null}_{\omega}),\\
\fg_2 &:& B_9:= \Sym^2 U^*.\end{eqnarray*}
By Lemma \ref{l.reductive} (4) and (5), they induce $\fg_0$-reductive filtrations of $\fg_{-1}^*$ and $\fg_{-2}^*$ whose graded objects are dual $G_0/R$-modules of the graded objects of $\fg_{-1}$ and $\fg_{-2}$ respectively. By Lemma \ref{l.reductive} (1), (2) and (3), they induce a $\fg_0$-reductive multi-filtration on $\wedge^2 \fg_-^*$ whose graded objects are one of the dual $G_0/R$-modules of the following
$$A_1:= \wedge^2 \Sym^2 U, \ A_2:= \Sym^2 U \otimes U \otimes {\rm Null}_{\omega}, \ A_3:= \Sym^2 U \otimes U \otimes Q^{\flat}, $$ $$A_4:= \wedge^2(U \otimes {\rm Null}_{\omega}), \ A_5:= U \otimes {\rm Null}_{\omega} \otimes U \otimes Q^{\flat}, \ A_6:= \wedge^2(U \otimes Q^{\flat}).$$
For each $\ell \geq 1$, they induce $\fg_0$-reductive multi-filtrations on factors of
$$C^{\ell,2}(\fg_0) = \bigoplus_{i,j\in \N} (\fg_{-i}^* \wedge \fg_{-j}^*) \otimes \fg_{-i-j+\ell}$$
whose graded objects are of the form $\Hom(A_i, B_j)$ with $ 1 \leq i \leq 6$ and $1 \leq j \leq 9$, or $\Hom(A_1, B_0), \Hom(A_2, B_0), \Hom(A_3, B_0)$.
Using Notation \ref{n.Cpositive}, we see that
 $$H^0(M, E_0 \times^{G_0} \Hom(A_i, B_j)) = \Hom(\sA_i, \sB_j)$$
where the bundles on the right hand side are as defined in Proposition \ref{p.grvanish}.
Thus the vanishing results in Proposition \ref{p.grvanish} and Lemma \ref{l.trivial} (i) verify
the condition $H^0(M, C^{\ell, 2}(E_0)/\partial(C^{\ell, 1}(E_0))) =0$ for all $\ell \geq 1$. \end{proof}

\subsection{Cartan connection in the setting of Subsection \ref{ss.rank0}}

\begin{notation}\label{n.Crank0}
Let $(\fg_0, \fg_-)$ be as in Notation \ref{n.rank0} and let $G_0 \subset \GL(\fg_-)$ be the connected algebraic group whose Lie algebra is $\fg_0 \subset \fgl(\fg_-)$. The kernel of the natural surjective homomorphism
$$ \iota:  \Hom(U,Q) \sd (\GL(U) \times \GL(Q)) \to G_0$$ arising from $\Hom(U,Q) \sd (\fgl(U) \oplus \fgl(Q)) \cong \fg_0$ is $${\rm Ker}(\iota) = \{ \pm 1 \in \GL(U) \times \GL(Q)\}.$$ Let $\fg$ be the universal prolongation of $(\fg_0, \fg_-)$ and let $G^0 \subset G \subset \GL(\fg)$ be the connected algebraic subgroups in Proposition \ref{p.spgraderank0}.

Let $M$ be a complex manifold of dimension equal to $\dim \fg_{-1}$ equipped with a $G_0$-structure, i.e., a $G_0$-principal subbundle $E_0 \subset {\rm Fr}(M)$.
 By the natural representation of $G_0$ on
$U \otimes Q$, we have the associated vector bundle
$\sD:= E_0 \times^{G_0} (U \otimes Q)$ with
a natural inclusion $\sD \subset TM$.
\end{notation}

\begin{theorem}\label{t.Crank0}
In the setting of Notation \ref{n.Crank0}, assume that there exist vector bundles
$\sU$ and $\sQ$ on $M$ such that
 \begin{itemize} \item[(1)] there are  isomorphisms  $\sD \cong \sU \otimes \sQ$ and  $TM/\sD \cong \Sym^2 \sU$ as vector bundles on $M$; \item[(2)] the pair $(M, \sU)$ satisfies SAF-condition and $\sQ$ is of trivial type. \end{itemize}
 Assume furthermore that the distribution $\sD$ is integrable. Then there exists a  Cartan connection $(E \to M, \omega)$ of type $(G, G^0)$ on $M$ related to $E_0$ in the way described in Theorem \ref{t.Cartan}. \end{theorem}

\begin{proof}
 By the condition (2), we can apply Proposition \ref{p.symbol} with $\varphi$ given by the Frobenius bracket tensor of the distribution $\sD \subset TM$. Since we are assuming that $\sD$ is integrable, we have $\varphi =0$ and $M=M_o$. Consequently, we can apply Proposition \ref{p.grvanish} to the vector bundles arising from $\sU$ and $\sQ$ on $M$.

To obtain the Cartan connection, we will use Theorem \ref{t.Cartan'}. The condition (i) of Theorem \ref{t.Cartan'} is just our assumption that $\sD$ is integrable. Let us check the conditions (ii) and (iii) of Theorem \ref{t.Cartan'}.

Note that the nilradical of $\fg_0$ is ${\mathfrak r} = \Hom(U, Q)$.
We will use the following $\fg_0$-reductive  filtrations
\begin{eqnarray*}
\fg_{-1} &=& (U \otimes Q) \oplus \Sym^2 U \supset U \otimes Q \supset 0
\\ \fg_0 & = & {\mathfrak r} \sd (\fgl(Q) \oplus \fgl(U)) \supset {\mathfrak r} \sd \fgl(Q) \supset {\mathfrak r}  \supset 0 \\ \fg_1 &=& \Sym^2U^* \supset 0. \end{eqnarray*}
Denoting the unipotent radical of $G_0$ by $R$, the associated graded objects are the following $G_0/R$-modules,
\begin{eqnarray*} \fg_{-1} &:& B_0:= \Sym^2 U, \ B_2:=U \otimes Q, \\
 \fg_0 &:& B_3:= \fgl(U),  \ B_5:= \fgl(Q), \  B_8:= \Hom(U, Q),  \\
\fg_1 &:& B_9:= \Sym^2 U^*.\end{eqnarray*}
where we have borrowed the notation from the proof of Theorem \ref{t.Cpositive} by putting
${\rm Null}_{\omega} = Q$.
The filtration on $\fg_{-1}$ induces a $\fg_0$-reductive filtration on $\fg_{-1}^*$. The latter induces a $\fg_0$-reductive multi-filtration on $\wedge^2 \fg_{-1}^*$ whose graded objects are one of the dual $G_0/R$-modules of the followings
$$A_1:= \wedge^2 \Sym^2 U, \ A_2:= \Sym^2 U \otimes U \otimes Q, \ A_4:= \wedge^2(U \otimes Q).$$ For each $\ell \geq 1$, these induce a $\fg_0$-reductive multi-filtration on
$$C^{\ell,2}(\fg_0) = (\wedge^2 \fg_{-1}^*) \otimes \fg_{\ell-2}$$
whose graded objects are of the form $\Hom(A_i, B_j)$ with $$ i=1,2,4 \mbox{ and } j =0,2,3,5,8, 9.$$
Using Notation \ref{n.Crank0}, we see that
 $$H^0(M, E_0 \times^{G_0} \Hom(A_i, B_j)) = \Hom(\sA_i, \sB_j)$$
where the bundles on the right hand side are as defined in Proposition \ref{p.grvanish}.
Thus the vanishing results in Proposition \ref{p.grvanish} show that all of them vanish
except $$\Hom(\sA_4, \sB_0) = \Hom(\wedge^2(\sU \otimes \sQ), \Sym^2 \sU).$$

  Lemma \ref{l.trivial} (ii) gives an injective homomorphism
$$H^0(M, C^{1, 2}(E_0)/\partial(C^{1, 1}(E_0))) \to  \Hom(\sA_4, \sB_0).$$  It is easy to see that this homomorphism arises from the homomorphism $$\varrho: H^0(M, C^{1, 2}(E_0)/\partial(C^{1, 1}(E_0))) \to \Hom(\wedge^2 \sD, TM/\sD) = \Hom(\sA_4, \sB_0)$$
from Notation \ref{n.Gst} with $D = U \otimes Q \subset \fg_{-1}$.
Thus the condition (ii) of Theorem \ref{t.Cartan'} is satisfied.

  Lemma \ref{l.trivial} (i) gives $H^0(M,  C^{\ell, 2}(E_0)/\partial(C^{\ell, 1}(E_0))) =0$ for all $\ell \geq 2$. Thus the condition (iii) of Theorem \ref{t.Cartan'} is satisfied.
\end{proof}

\section{Geometry of the projective variety $Z \subset \BP W$}\label{s.Z}

Throughout Section \ref{s.Z} and Section \ref{s.vmrt}, we will use the terminology  of Definition \ref{d.Z} below.

\subsection{Basic properties of $Z \subset \BP W$}

\begin{definition}\label{d.Z}
Let $U$ and $Q$ be  vector spaces with $\dim U = m \geq 2$ and $\dim Q = n \geq 1$.
Write $W := (U \otimes Q) \oplus \Sym^2 U$.
 \begin{itemize}
 \item[(1)] Define $\widehat{Z} \subset W$ by  $$\widehat{Z} :=\{ u \otimes q + t u^2 \in (U \otimes Q) \oplus \Sym^2 U, \ u \in U, q \in Q, t \in \C \}$$ and let $Z \subset \BP W$ be the corresponding projective variety. Denote by $\sO_Z(1)$ the restriction of the hyperplane line bundle $\sO_{\BP W}(1)$. It is easy to see that the embedding $Z \subset \BP W$ is linearly normal, i.e.,
      $$ W^* = H^0(Z, \sO_Z(1)).$$  \item[(2)]  Let $E \subset Z$ be the divisor $Z \cap \BP (U \otimes Q)$. Note that the embedding $E \subset \BP (U \otimes Q)$ is isomorphic to the Segre embedding of $\BP U \times \BP Q$. \item[(3)] Let $\widehat{S} \subset \Sym^2 U$ be the set of pure symmetric tensors, which is the affine cone of $S \subset \BP \Sym^2 U$, the 2nd Veronese image of $\BP U$. \item[(4)]  Let $\epsilon: W \to \Sym^2 U$ be the projection modulo $U \otimes Q$. It gives rise to a $\BP^n$-bundle structure $\alpha: Z \to \BP U \cong S$ such that $$\alpha|_E : E \cong \BP U \times \BP Q \to \BP U$$ is the trivial $\BP^{n-1}$-bundle.
      \item[(5)] Let ${\bf G} \subset {\rm GL}(W)$ be the group of linear automorphisms of $\widehat{Z}$, in other words,  $${\bf G} := \{ \varphi \in \GL(W), \ \varphi(\widehat{Z}) = \widehat{Z} \}.$$ \end{itemize} \end{definition}

We skip the proof of the following  elementary lemma ((i) is essentially from Proposition 4.2.1 of \cite{HM05} or Proposition 3.8 of \cite{FH12}).

\begin{lemma}\label{l.bfG}
\begin{itemize} \item[(i)] The group ${\bf G} \subset {\rm GL}(W)$ is equal to $G_0$ in Notation \ref{n.Crank0}, i.e., the image of a natural homomorphism
$$\iota: \Hom(U,Q) \sd ({\rm GL}(U) \times {\rm GL}(Q)) \to \GL(W)$$ with ${\rm Ker}(\iota) = \{ \pm 1 \in \GL(U) \times \GL(Q)\}.$
%
\item[(ii)] The subspace $U \otimes Q \subset W$ is preserved by the ${\bf G}$-action and
the projection $\epsilon: W \to \Sym^2 U$ is ${\bf G}$-equivariant.
\item[(iii)] The ${\bf G}$-action on $Z$ has two orbits: the divisor $E$ and its complement $Z \setminus E$. \end{itemize} \end{lemma}

The following is straight-forward to check.

 \begin{lemma}\label{l.blowup}
 Define a morphism  $\beta: Z \to \BP (Q \oplus U)$ by
  $$ \beta([ u \otimes q + t u^2]) = [(q, tu)] \mbox{ for }  u \otimes q + t u^2 \in \widehat{Z}.$$   \begin{itemize}
  \item[(a)] The morphism $\beta$ is the blowup of $\BP (Q \oplus U)$ along the subspace
  $\BP Q$ and $E$ is the exceptional divisor of the blowup.
  \item[(b)] There is a natural ${\bf G}$-action on $\BP (Q \oplus U)$ where $\varphi \in \Hom(U, Q)$ acts by
  $$\varphi\cdot [q + u] = [q + \varphi(u) + u].$$ The morphism $\beta$ is ${\bf G}$-equivariant under this action.
  \item[(c)] There is an isomorphism of line bundles on $Z$ $$\sO_Z(1) \ \cong \ \sO_Z(-E) \otimes   \beta^*\sO_{\BP(Q \oplus U)}(2).$$ In particular, the morphism $\beta$ sends conics on $Z \subset \BP W$ disjoint from $E$  to lines in $\BP(Q \oplus U)$
      disjoint from $\BP Q$ and $\alpha$-fibers to $n$-dimensional linear subspaces in
      $\BP (Q \oplus U)$ containing $\BP Q$. \end{itemize} \end{lemma}

The following is from Lemma 3.4.1 of \cite{HM05}.

\begin{lemma}\label{l.tanZ}
For a point $w \in \widehat{Z}\setminus (U \otimes Q),$ let $u \in U$ be a vector such that
$u^2 = \epsilon (w) \in \widehat{S} \subset \Sym^2 U$. Then the affine  tangent space $T_w \widehat{Z} \subset W$ of $\widehat{Z}$ at $w$ satisfies
\begin{eqnarray*} (U \otimes Q) \cap T_w \widehat{Z} &=& u \otimes Q \ \subset U \otimes Q \\   \zeta(T_w \widehat{Z}) &=& u \odot U \ \subset \Sym^2 U \end{eqnarray*} yielding a short exact sequence $$0 \to u \otimes Q \to T_w \widehat{Z} \stackrel{\epsilon}{\to} u \odot U \to 0.$$
\end{lemma}

\begin{definition}\label{d.II}
Let $Y \subset \BP V$ be a projective variety and let  $y \in Y$ be a nonsingular point.
 Let $\rho_Y: V^* \to H^0(Y, \sO_Y(1))$ be the restriction homomorphism on $Y$. Denote by $H^0(Y, {\bf m}^k_y(1)) \subset H^0(Y, \sO_Y(1))$ the subspace consisting of sections vanishing at $y$ to order $k \in \N$.  \begin{itemize}
\item[(1)] The  {\em second normal space} $N^{(2)}_{Y,y} \subset N_{Y,y}$ is the subspace of the normal space of $Y \subset \BP V$ at $y$ annihilated by $\rho_Y^{-1}(H^0(Y, {\bf m}^3_y (1))).$
    \item[(2)] The $2$-jets of elements of $H^0(Y, {\bf m}^2_y (1))$ defines an injection
    $${\rm II}_{Y,y}: N^{(2)}_{Y,y} \to \Sym^2 T^*_y Y$$ called the {\em 2nd fundamental form } of $Y$ at $y$.
    \item[(3)] If $N^{(2)}_{Y,y} = N_{Y,y}$, or equivalently $H^0(Y, {\bf m}^3_y (1)) =0$, we say that the second fundamental form of $Y$ {\em surjects} at $y$. \end{itemize}
    \end{definition}

    \begin{lemma}\label{l.IIZ} \begin{itemize}
\item[(1)]    For any point $z \in Z \setminus E$, the second fundamental form of $Z$ surjects at $z.$ \item[(2)] The image
    of ${\rm II}_{Z,z}$ in $\Sym^2 T^*_z Z$ is the system of quadrics vanishing on
     the subspace $T^{\alpha}_z \subset T_z Z$  tangent to the fiber through $z$ of the morphism   $\alpha: Z \to \BP U$ in Definition \ref{d.Z} (4). \item[(3)] For a point $e =[u \otimes q] \in E$, the base locus of
     the system of quadrics ${\rm II}_{Z,e}$ on $T_e(Z)$ includes
     both $T^{\alpha}_e$ and $T^{\beta}_e$, where $\beta$ is the morphism in
     Lemma \ref{l.blowup}. Consequently, the second fundamental form
     ${\rm II}_{Z,e}$ is not isomorphic to ${\rm II}_{Z,z}$ in (2) as systems of quadrics
     on a vector space of dimension $n+m-1 = \dim Z$. \end{itemize}    \end{lemma}

\begin{proof}
Suppose there exists a section $s \in H^0(Z, {\bf m}^3_z(1))$.
Then $s$ vanishes along any conic on $Z$ through $z$. But Lemma \ref{l.blowup} (c) says that such conics cover $Z \setminus E$. Thus $s=0$. This proves (1).

If $s \in H^0(Z, {\bf m}_z^2(1))$, then $s$ vanishes along all lines through $z$ on $Z$.
Thus the base locus of ${\rm II}_{Z,z}$ contains $T^{\alpha}_z$. But the space of
all quadrics on $T_z Z$ vanishing on $T^{\alpha}_z$ has the same dimension as $\dim N_{Z,z}$,
which is equal to the dimension of ${\rm II}_{Z,z}$ by (1). This proves (2).

If $s \in H^0(Z, {\bf m}_e^2(1))$, then $s$ vanishes along all lines through $e$ on $Z$.
Thus the base locus of ${\rm II}_{Z,e}$ contains both $T^{\alpha}_e$ and $T^{\beta}_e$, proving (3). \end{proof}

\subsection{A rigidity property of  $Z \subset \BP W$}

The goal of this subsection is to prove the following  rigidity property of $Z \subset \BP W$ under a projective degeneration.

\begin{theorem}\label{t.rigid}
Let $\sW$ be a vector bundle of rank equal to $\dim W$ over the unit disc $\Delta :=\{ t \in \C, |t| <1\}$ and
let $W_t$ be its fiber at $t \in \Delta$.
Let $\sZ_{\Delta} \subset \BP \sW$ be an irreducible closed subvariety in the projectivized bundle $\BP \sW \to \Delta$ and let $p: \sZ_{\Delta} \to \Delta$ be the natural projection. Denote by $Z_t$ the set-theoretic fiber of $p$ over $t \in \Delta.$ Assume that there exists a closed algebraic subset $A \subset Z_0, A \neq Z_0,$ with  the following properties.
\begin{itemize} \item[(i)] The complement $Z'_0:= Z_0 \setminus A$ is an irreducible nonsingular subvariety of $\BP W_0$ and the morphism $p$ is smooth (i.e. submersive) on $\sZ_{\Delta} \setminus A$. \item[(ii)]
 For each $t \in \Delta, t\neq 0$, the fiber $Z_t \subset \BP W_t$ is projectively equivalent to $Z \subset \BP W$. 
 \item[(iii)]
 There exists a smooth morphism $\alpha^0: Z'_0 \to B$ to a nonsingular variety $B$ such that the closure of each fiber of $\alpha^0$ is a linear subspace of dimension $n = \dim Q$ in $\BP W_0$. \item[(iv)] For each point $z \in Z'_0$, if there exists a line $\ell \subset \BP W$ satisfying $z \in \ell$ and $\ell \subset Z_0$, then the germ of $\ell $ at $z$ is contained in the fiber of $\alpha^0$.
 \item[(v)] The submanifold $Z'_0 \subset \BP W_0$ is linearly nondegenerate, i.e., not contained in any hyperplane in $\BP W_0$.
 \end{itemize}
 Then $Z_0 \subset \BP V_0$ is projectively equivalent to $Z \subset \BP W$.
 \end{theorem}

Our proof of Theorem \ref{t.rigid} is an adaptation of Mok's proof of Proposition 2.3 in \cite{Mo}, replacing G-structures of Hermitian symmetric spaces by projective connections.

\begin{definition}\label{d.projective}
Let $Y$ be a complex manifold of dimension $N$ and let $\pi: \BP TY \to Y$ be the projectivization of the tangent bundle of $Y$. For each $\eta \in \BP T_y Y$, let $\widehat{\eta} \subset T_y Y$ be the 1-dimensional subspace corresponding to $\eta $ and let $${\rm d}_{\eta} \pi: T_{\eta} (\BP TY) \to T_y Y$$ be the derivative of $\pi$.
\begin{itemize} \item[(1)]  Setting   $\sJ_{\eta} \subset T_{\eta}( \BP TY )$ to be
$({\rm d}_{\eta} \pi)^{-1} (\widehat{\eta})$, we obtain a subbundle $\sJ \subset T (\BP TY)$, called the {\em tautological distribution } on $\BP TY$.
\item[(2)] A {\em projective connection} on $Y$ is a line subbundle $\sF \subset \sJ$ splitting the exact sequence $$ 0 \to T^{\pi} \to \sJ \to \sJ/T^{\pi} \to 0$$ where $T^{\pi}$ is the vertical tangent bundle of $\pi$ and $\sJ/T^{\pi}$ is isomorphic to  the tautological line bundle of the projectivization $\BP TY$. The curves on $Y$ that are the images of the  leaves of $\sF$ under $\pi$ are called the {\em geodesics} of the projective connection.
    \item[(3)] Given a projective connection $\sF$ on $Y$ and an open subset $O \subset Y$, the restriction $\sF|_O$ is the projective connection on $O$ induced by
        $\sF$ on the open subset $\BP TO \subset \BP TY$.
   \item[(4)]  When $Y = \BP^N$, there is a natural
    projective connection $\sF^{\rm flat}$  called the {\em flat projective connection} whose geodesics are exactly lines on $\BP^N$.
    \item[(5)]
        A projective connection $\sF$ on $Y$ is {\em locally flat} if each point $y \in Y$  has an open neighborhood $O \subset Y$ equipped with an open immersion
        $\varphi: O \to \BP^N$   which sends germs of geodesics of $\sF$ to germs of lines on $\BP^N$.
        \item[(6)] For a connected open subset $O' \subset \BP^N$,
       a foliation $\sA' \subset TO'$ of rank $k <N$ is said to be {\em linear} if there exists a linear subspace $\BP^{k-1} \subset \BP^N$ such that leaves of $\sA'$ are contained in linear subspaces of dimension $k$ in $\BP^N$ containing $\BP^{k-1}$.
\item[(7)]        Given a locally flat projective connection $\sF$ on a complex manifold $Y$, a foliation $\sA \subset T Y$ is said to be {\em linear} if the immersion $\varphi: O \to \BP^N$ on an open neighborhood of a point in $Y$ in (5) sends $\sA$ to a linear foliation $\sA'$ on $O' = \varphi(O) \subset \BP^N$. It is easy to see that this property  does not depend on the choice of $\varphi$.
\end{itemize}        \end{definition}

\begin{example}\label{e.Z}
The map $\beta: Z \to \BP(Q \oplus U)$ in Lemma \ref{l.blowup} induces a locally flat projective connection on
$Z \setminus E.$ The fibers of  $\alpha$ define a linear foliation on $Z \setminus E$. \end{example}

\begin{lemma}\label{l.fundamental}
Let $\sF$ be a locally flat projective connection  on a simply-connected complex manifold $M$ with $\dim M = m+n-1 = \dim Z$ and let $\sA \subset TM$ be a linear foliation of rank $n$. Let $O \subset M$ be a connected open subset and let $\varphi: O \to \BP^{m+n-1}$ be an open immersion sending $\sF$ to $\sF^{\rm flat}|_{\varphi(O)}$.
\begin{itemize} \item[(i)] There exist a unique linear subspace $\BP^{n-1}_{\varphi} \subset \BP^{m+n-1}$ of dimension $n-1$ and an open immersion
$\widetilde{\varphi}: M \to \BP^{m+n-1} \setminus \BP^{n-1}_{\varphi}$ such that
$\varphi = \widetilde{\varphi}|_O$ and leaves of $\sA$ are sent into linear spaces of dimension $n$ containing $\BP^{n-1}_{\varphi}$.
\item[(ii)] If $M = Z \setminus E$ and $(\sF, \sA)$ is as in Example \ref{e.Z},
  there is a biregular morphism to the blowup of $\BP^{m+n-1}$ along $\BP_{\varphi}^{n-1}$ $$\Phi: Z \to {\rm Bl}_{\BP_{\varphi}^{n-1}} \BP^{m+n-1} $$ such that $\varphi= \Phi|_O$. \end{itemize} \end{lemma}

\begin{proof}
Since $\sA$ is linear, there exists $\BP^{n-1}_{\varphi} \subset \BP^{m+n-1}$ such that
leaves of $\sA$ are sent by $\varphi$ into linear subspaces of dimension $n$ containing $\BP^{n-1}_{\varphi}$.
As $M$ is simply-connected,  the developing map construction as in the proof of Proposition \ref{p.develop} gives an extension of $\varphi$ to an open immersion $\widetilde{\varphi}: M \to \BP^{m+n-1}$ which sends $\sF$ to $\sF^{\rm flat}|_{\widetilde{\varphi}(M)}.$ Since  $\widetilde{\varphi}$ sends leaves of $\sA$ into
 into linear subspaces of dimension $n$ containing $\BP^{n-1}_{\varphi}$, its image lies in $\BP^{m+n-1} \setminus \BP^{n-1}_{\varphi}$, proving (i).

  When $M= Z \setminus E$, the map $\widetilde{\varphi}$
  induces  a holomorphic map $$\BP (Q \oplus U) \setminus \BP Q \to \BP^{m+n-1} \setminus \BP^{n-1}_{\varphi}$$ which sends germs of lines to germs of lines.
  Thus it must extend to a biregular morphism $\BP (Q \oplus U) \to \BP^{m+n-1}$ which sends $\BP Q$ to $\BP_{\varphi}^{n-1}$.
 Then it lifts to a biregular morphism $$\Phi: Z = {\rm Bl}_{\BP Q} \BP (Q \oplus U) \to {\rm Bl}_{\BP^{n-1}} \BP^{m+n-1},$$  proving (ii). \end{proof}

\begin{lemma}\label{l.Weyl}
Let $\sY$ be a complex manifold and $q: \sY \to \Delta$ be a smooth morphism (submersion).
Assume that we have a holomorphic family of projective connections $\sF^t$  and foliations of fixed rank $\sA^t$ on the fiber $Y_t:= q^{-1}(t)$ for all $t \in \Delta$. \begin{itemize}
\item[(i)] If $\sF^t$ is locally flat for all $t \neq 0$, then so is $\sF^0.$
\item[(ii)] In the setting of (i), if $\sA^t$ is linear for all $t \neq 0$, then so is $\sA^0$. \end{itemize}
\end{lemma}

\begin{proof}
A classical result due to E. Cartan and T. Y. Thomas (see e.g. Theorem 15 in \cite{KN} or Theorem 4.4.3 of \cite{CSl}),
 says that a projective connection has a naturally defined holomorphic curvature tensor (i.e.  a 2-form with values in a naturally defined  vector bundle: the type of the vector bundle is different depending on whether the dimension of the base manifold is 2 or bigger than 2)  such that the curvature tensor vanishes if and only if the projective connection is locally flat.  This proves (i).

 To prove (ii), using (i) and shrinking $\sY$ if necessary, we can regard $\sY$ as an open subset in
  $\BP^N \times \Delta, N = \dim Y,$ such that $\sF^t$ come from the flat projective connection on $\BP^N$.
 As the leaves of $\sA^t, t \neq 0$ are open subsets of linear subspaces in $\BP^N$, so are leaves of $\sA^0.$
  Let $k$ be the rank of $\sA$ and define a holomorphic map $\mu: \sY \to {\rm Gr}(k+1, \C^{N+1})$ sending $y \in \sY$ to  the linear space in $\BP^N$ containing the leaf of $\sA^t, t = q(y),$ through $y \in Y^t.$
 Since $\sA^t, t \neq 0,$ is linear, the image  $\mu (Y_t)$ is an open subset in a linear subspace of ${\rm Gr}(k+1, \C^{N+1})$ under the Pl\"ucker embedding. Thus so is the image $\mu(Y_0)$, which implies  that $\sA^0$ is linear.
   \end{proof}

Now we turn to the proof of Theorem \ref{t.rigid}. First, note that
the assumption in Theorem \ref{t.rigid} has the following consequence on the conic curves of $Z_0$

\begin{lemma}\label{l.rigid}
In the setting of Theorem \ref{t.rigid},
let $z_0 \in Z'_0$ be a general point and let $v_0 \in T_z Z_0$ be a general tangent vector, not tangent to the fiber of $\alpha^0$. Using the property (i), choose a holomorphic section $$\{ z_t \in Z_t, \ t \in \Delta\}$$ of $p: \sZ_{\Delta} \to \Delta$ passing through $z_0$ and a family of tangent vectors $$\{v_t \in T_{z_t} Z_t, \ t \in \Delta\}$$ converging to $v_0$. By Lemma \ref{l.blowup} (c), we have a unique conic curve $C_t \subset Z_t$ tangent to $v_t$ for each $t \neq 0$. Then their limit $C_0
\subset Z_0$ as algebraic cycles is a nonsingular conic curve tangent to $v_0$. \end{lemma}

\begin{proof}
Suppose not. Then $C_0$ is the union of two distinct lines or a double line passing through $z_0.$ The limit of the planes $\langle C_t \rangle$ spanned by $C_t, t \neq 0$, gives a plane $\langle C_0 \rangle$ in $\BP W_0$ which contains $C_0$ and
includes $v_0$ as its tangent vector.

If $C_0$ is the union of two distinct lines, both lines are contained in the fiber of $\alpha^0$ by the condition (iv) of Theorem \ref{t.rigid}. Thus $\langle C_0\rangle$ lies in the fiber of $\alpha_0$ and cannot have $v_0$ as its tangent vector, a contradiction.

Thus we can assume that  $C_0$ is a double line. By the generality of the choice of $z_t$ and $v_t$, we may assume that the limits of conic curves on $Z_t$ through $z_t$ are double lines in  $Z_0$. Choose a point $z'_0 \neq z_0$ in a neighborhood of $z_0$ in $Z'_0$ such that
$z'_0$ is not in the fiber of $\alpha^0$ through $z_0$. Choose a section $$\{z'_t \in Z_t, \ t \in \Delta\}$$ of $p$ through $z'_0 \in Z'_0$. By Lemma \ref{l.blowup} (c), there exists  a unique conic $C'_t \subset Z_t$ joining $z_t$ and $z'_t$
for each $t \neq 0$. Their limit $C'_0$ is a curve joining $z_0$ to $z'_0$ and, by our assumption, it is a double line contained in the $\alpha^0$-fiber through $z_0$. This is a contradiction to our choice of $z'_0$. \end{proof}

  We borrow the following   lemma from  Lemma 2.3 in \cite{Mo}.

\begin{lemma}\label{l.Mok2.3}
Let $\sW$ be a vector bundle over $\Delta$. Let $Z$ be a nonsingular projective variety
and let $\Psi$ be a meromorphic map from $Z \times \Delta$ into $\BP \sW$ respecting the natural projections to $\Delta$. Assume that \begin{itemize}
\item[(a)] the restriction of $\Psi$ at $t \in \Delta\setminus \{0\}$ gives a linearly normal embedding
$Z \times \{ t \} \to \BP W_t$; and
\item[(b)] there is a point $z \in Z \times \{ 0 \}$ such that $\Psi$ is holomorphic at $z$ and $\Psi (O_z) \subset \BP W_0$ is linearly nondegenerate for some open neighborhood $O_z$ of $z$ in $Z \times \{0 \}.$\end{itemize}
    Then $\Psi$ is a holomorphic embedding. \end{lemma}

\begin{proof}[Proof of Theorem \ref{t.rigid}]
For $t\neq 0$, let $E_t \subset Z_t$ be the hypersurface corresponding to $E \subset Z$ under the biholomorphism $Z_t \cong Z$. Then $E_t$ is the exceptional divisor of a contraction $\beta_t: Z_t \to \BP^{m+n-1}$. We have the projective connection $\sF^t$ on $Z_t \setminus E_t$ arising from the flat projective connection on $\BP^{m+n-1}$ via $\beta_t: Z_t \to \BP^{m+n-1}$ as in Example \ref{e.Z}. Taking limits of leaves of $\sF^t$ on $\BP T (Z_t \setminus E_t)$ as $t \rightarrow 0$, we obtain a holomorphic foliation $\sF^0$ of rank 1 on a Zariski open subset
in $\BP T Z'_0$. The leaves of $\sF^0$ are tangent to the tautological distribution $\sJ^0$ of $\BP T Z'_0$.
By Lemma \ref{l.rigid}, general leaves of $\sF^0$ are obtained by the tangent vectors to conics on $Z_0$. Thus $\sF^0$ is transversal to $T^{\pi_0}$ where $\pi_0: \BP T Z'_0 \to Z'_0$ is the natural projection. By Hartogs extension, we see that $\sF^0$ gives a projective connection on $Z'_0$. By Lemma \ref{l.Weyl}, the projective connection is locally flat.

By the property (iii), the foliations $\sA^t$ on $Z_t \setminus E_t, t \neq 0,$ given by $\alpha$ on $Z \setminus E$ converges to a foliation $\sA^0$ on $Z'_0$. By Lemma \ref{l.Weyl}, this is a family of linear foliations with respect to the locally flat projective connections $\sF^t$, including $t=0$.

By the local flatness of the projective connections, we can find an open subset
 $O \subset \sZ_{\Delta} \setminus A$ with $O_t := O \cap p^{-1} (t)$ connected and a family of open immersions $\varphi_t: O_t \to \BP^{m+n-1}$ preserving the projective connections for all $t \in \Delta$.
By Lemma \ref{l.fundamental}, the  holomorphic map
$\varphi_t$ can be extended to a family of biregular morphisms $Z_t \cong Z$ for each $t \neq 0$. Their  inverse maps give  a biholomorphic map $$ \psi: Z \times (\Delta \setminus \{0\})  \to \sZ_{\Delta} \setminus Z_0$$ which has a  holomorphic extension   across a nonempty open subset of
$Z \times \{0\}$.
By Hartogs extension of meromorphic maps, we can extend $\psi$ to a meromorphic map
$$\Psi: Z \times \Delta \dasharrow \sZ_{\Delta}$$ which satisfies the conditions of Lemma  \ref{l.Mok2.3}.  Theorem \ref{t.rigid} follows by Lemma \ref{l.Mok2.3}.
\end{proof}

We have the following corollary of Theorem \ref{t.rigid}.

\begin{corollary}\label{c.rigid}
Let $\sW$ be a vector bundle of rank equal to $\dim W$ over the unit disc $\Delta :=\{ t \in \C, |t| <1\}$ and
let $W_t$ be its fiber at $t \in \Delta$.
Let $\sZ_{\Delta} \subset \BP \sV$ be an irreducible closed subvariety and let $p: \sZ_{\Delta} \to \Delta$ be the natural projection. Denote by $Z_t$ the set-theoretic fiber of $p$ over $t \in \Delta.$ Assume that there exists a closed algebraic subset $A \subset Z_0, A \neq Z_0,$ with  the following properties.
\begin{itemize} \item[(a)] The morphism $p$ is smooth on $\sZ_{\Delta} \setminus A$. \item[(b)]
 For each $t \in \Delta, t\neq 0$, the fiber $Z_t \subset \BP W_t$ is projectively equivalent to $Z \subset \BP W$. \item[(c)] At each point $z \in Z_0 \setminus A$, the second fundamental form ${\rm II}_{Z_0, z}$ as a system of quadrics is isomorphic  to ${\rm II}_{Z, x}$ at a point $x \in Z\setminus E$.  \end{itemize}
 Then $Z_0 \subset \BP V_0$ is projectively equivalent to $Z \subset \BP W$.
 \end{corollary}

 \begin{proof}
 It suffices to verify the conditions (i)-(v) of Theorem \ref{t.rigid}.

 The conditions (i) and (ii) of Theorem \ref{t.rigid} are identical to the conditions (a) and (b).

The base locus of ${\rm II}_{Z,x}$ defines the fibers of $\alpha: Z \to \BP U$ by Lemma \ref{l.IIZ} (2). By the condition (c), we have corresponding fibration on a nonempty Zariski open subset of $Z_0 \setminus A$. Thus by enlarging $A$ if necessary, we can assume that the condition (iii) of Theorem \ref{t.rigid} holds, i.e., we have a fibration $\alpha^0: Z_0 \setminus A \to B$ the closure of whose fibers are linear subspaces of dimension $n$ in $\BP W_0$.

To check the condition (iv), note that for any line on $Z_0$ passing through $z \in Z'_0$, its tangent vector must be in the base locus of the system of quadrics ${\rm II}_{Z_0,z}$ by Definition \ref{d.II}. By Lemma \ref{l.IIZ} (2) and the assumption (c), this must be inside $T^{\alpha^0}$, implying the condition (iv).

 The condition (c) implies that $N_{Z_0, z}^{(2)} = N_{Z_0, z}$, which shows that the germ of $Z_0$ at $z$ is linearly nondegenerate in $\BP W_0$. Thus the condition (v) of Theorem \ref{t.rigid} holds.

\end{proof}

\section{Proof of Theorems \ref{t.Moksp}, \ref{t.recog} and \ref{t.deform}}\label{s.vmrt}

 We start with recalling briefly the basic terminology of the VMRT theory. For more details,  we refer the reader to Section 1 of \cite{Hw01} or Section 3 of \cite{HH} (or any of the sections introducing VMRT in the surveys \cite{Hw12}, \cite{Hw14} or \cite{M16}).

\begin{definition}\label{d.vmrt}
Let $X$ be a uniruled projective manifold. We fix a family $\sK$ of minimal rational curves, i.e. an irreducible component of the space of rational curves on $X$ such that for a general point $x \in X$, the subscheme $\sK_x \subset \sK$ parametrizing  members of $\sK$ passing through $x$ is nonempty and projective.
 For a general point $x \in X$, the {\em variety of minimal rational tangents} (VMRT) at $x$, to be denoted by $\sC_x \subset \BP T_x X,$ is  the closure of the union of the tangent directions to members of $\sK_x$ which are nonsingular at $x.$ Denote by
$\widehat{\sC}_x \subset T_x X$ the affine cone of $\sC_x$.
Let $\sC \subset \BP TX$ be the {\em VMRT-structure}, i.e., the closure of the union of VMRT's $$\{ \sC_x \subset \BP T_x X, \ \mbox{ general } x \in X\}$$ determined by $\sK$.
\end{definition}

The following results are standard.

\begin{theorem}\label{t.standard}
In the setting of \ref{d.vmrt}, assume that  $p: = \dim \sC_x \geq 2$.  Let $C \subset X$ be a general member of $\sK$ passing through a general point $x \in X$. \begin{itemize}
\item[(1)] The rational curve $C$ is nonsingular.
\item[(2)] The natural projection $\sC \to X$ is smooth (submersive)  in a neighborhood the curve $\BP TC \subset \sC.$
\item[(3)] For each $y \in C$, let $y':= \BP T_y C \in \sC_y \subset \BP T_y X$ be the corresponding point.
Then the second fundamental form ${\rm II}_{\sC_y, y'}$ of the projective variety
$\sC_y \subset \BP T_y X$, which is nonsingular at $y'$ by (2), is isomorphic to ${\rm II}_{\sC_x, x'}$.

\item[(4)] The restriction of the tangent bundle to $C$ splits as
$$TX|_C \cong \sO(2) \oplus \sO(1)^{\oplus p} \oplus \sO^{\oplus (\dim X -p-1)}.$$
In particular, the degree of $T(X)|_C$ is $p +2$.

\item[(5)] For each $y \in C$, the affine tangent space $T_{y'} \widehat{\sC}_y \subset T_y X$
is equal to the factor $(\sO(2) \oplus \sO(1)^{\oplus p})_y \subset T_y X$ in the splitting of $TX|_C$ in (4). \end{itemize} \end{theorem}

\begin{proof}
That $C$ is nonsingular in (1) is a consequence of the assumption $p \geq 2$ and Theorem 3.3 of \cite{Ke}. (2) is a consequence of the freeness of $C$, no obstruction in deforming $C$ with one point fixed (e.g. Proposition 1.4 in \cite{Hw01}). (3) is by Proposition 2.2 of \cite{Mo} or
Proposition 3.1 of \cite{HH}.
(4) is a well-known consequence of  the bend-and-break argument (e.g. Corollary 2.9 in Chapter IV of \cite{Ko} or Theorem 1.2 in \cite{Hw01}). (5) is from the basic deformation theory of rational curves (e.g. Proposition 2.3 of \cite{Hw01}). \end{proof}

The following is well-known to experts.

\begin{proposition}\label{p.ratconn}
In Theorem \ref{t.standard}, assume that $\widehat{\sC}_x $ spans $T_x X$ for a general $x \in X$. Then for any open subset $O \subset X$ containing $C$,  there exists a holomorphic map $f: \BP^1
\to O$ such that $f^* T O$ is ample.
\end{proposition}

\begin{proof}
Let  $\sK^{\rm free} \subset \sK$ be the Zariski open subset  parametrizing members of $\sK$ which are free.
  For each positive integer $k$, let ${\rm Chain}^k(\sK^{\rm free})$ be the space of 1-dimensional cycles which are connected chains  of $k$  members of $\sK^{\rm free}$.
  The assumption that  $\widehat{\sC}_x $ spans $T_x X$  implies (e.g. by Theorem 4.13 in Chapter IV of \cite{Ko}) that there exists some positive integer $k,$  such that for two general points $x_1, x_2 \in X$, we can find a member of ${\rm Chain}^k(\sK^{\rm free})$ containing both $x_1$ and $x_2$.
  Then by Theorem 7.6 in Chapter II of \cite{Ko}, such a member of ${\rm Chain}^k(\sK^{\rm free})$ can be deformed  to a family of (irreducible) free rational curves on $X$ containing $x_1$. Fixing $x_1$ and varying $x_2$, we see that deformations of such members of
   ${\rm Chain}^k(\sK^{\rm free})$ give a family of  free rational curves on $X$ through $x_1$ which cover a nonempty open subset of $X$.
 Under the normalization $f: \BP^1 \to X$ of a general member of such a family, the vector bundle $f^*T X$ is ample. We conclude that  a general member of ${\rm Chain}^k(\sK^{\rm free})$ can be deformed to a rational curve under whose normalization $f: \BP^1 \to X$, the vector bundle $f^*T X$ is ample.

 Since $C$ is a member of $\sK^{\rm free}$,
the assumption that  $\widehat{\sC}_x $ spans $T_x X$  implies
 that  the subset of ${\rm Chain}^k(\sK^{\rm free})$ consisting of chains contained in $O$ includes a Euclidian open subset in
  ${\rm Chain}^k(\sK^{\rm free})$. Thus some chains in $O$ can be deformed to a rational curve
   contained in $O$ under whose normalization $f: \BP^1 \to O$, the vector bundle $f^*T O$ is ample.
 \end{proof}

The VMRT of lines through a general point of  a symplectic Grassmannian or an odd-symplectic Grassmannian is isomorphic to $Z \subset \BP W$ in Definition \ref{d.Z}, as explained in
 (3.2) of \cite{HM05}.
For the rest of the section, we will work with a uniruled projective manifold with the following property.

\begin{assumption}\label{a.setup}  Let $X$ be  a uniruled projective manifold with a family  $\sK$ of minimal rational curves. We assume that
  $$X^Z:= \{ x \in X, \ (\sC_x \subset \BP T_x X) \mbox{ is projectively equivalent to }
 (Z \subset \BP W).\}$$  is  a nonempty Zariski open subset in $X$.
 Note $\dim \sC_x = \dim Z = m+ n -1 \geq 2.$ Thus the assumptions in Theorem \ref{t.standard}
 and Proposition \ref{p.ratconn} hold.
  \end{assumption}

\begin{notation}\label{n.Sp}
Let ${\bf G} \subset \GL(W)$ be the group  in
 Definition \ref{d.Z} (5).
In Assumption \ref{a.setup}, let $\sP \subset {\rm Fr}(X^Z)$ be the principal ${\bf G}$-bundle whose fiber at $x \in X^Z$ is $$\sP_x:= \{ \varphi \in {\rm Fr}_x(X) = {\rm Isom}( W, T_x X), \varphi(\widehat{Z}) = \widehat{\sC}_x \}.$$ We have  fiber bundles on $X^Z$
$$ \sZ:= \sP \times^{\bf G} Z, \  \sE:= \sP \times^{\bf G} E, \  \sR:= \sP \times^{\bf G} \BP Q, $$ $$ \sS :=  \sP \times^{\bf G} S \ \cong \sP \times^{\bf G} \BP U $$
associated to the varieties $Z, E, \BP Q, \BP U \cong S$ (biregular)  with ${\bf G}$-actions from Definition \ref{d.Z}.  We  identify $\sZ$ with $\sC|_{X^Z}$ and regard $\sE$ as a fiber subbundle of $\sZ$.
  Let $D \subset W$ be the ${\bf G}$-stable subspace corresponding to $U \otimes Q$ in Lemma \ref{l.bfG} (ii) and let $B = W/D = \Sym^2 U$ be the quotient ${\bf G}$-module.
  We have vector bundles $$\sD := \sP \times^{\bf G} D \mbox{ and }  \sB:= \sP \times^{\bf G} B$$ on $X^Z$. There  are natural inclusions $$
  \sR \times_{X^Z} \sS \subset \BP \sD \mbox{ and }  \sS \subset \BP \sB $$ which come from the Segre embedding $\BP Q \times \BP U \subset \BP (Q \otimes U)$ and  the 2nd Veronese embedding $\BP U \cong S \subset \BP \Sym^2 U$, respectively.
  Let $$ \varepsilon: T X^Z \to \sB \mbox{ (resp. } \overline{\varepsilon}: \sZ \setminus \sE \to \sS )$$ be the vector bundle homomorphism (resp. holomorphic map) induced by
  $\epsilon : W \to \Sym^2 U$ in Definition \ref{d.Z} (4).

\end{notation}

\begin{remark}
As mentioned before, because of the kernel of $\iota$ in Lemma \ref{l.bfG} (i), the vector spaces $U$ and $Q$ do not have natural ${\bf G}$-module structures.  \end{remark}

\begin{proposition}\label{p.extensionsp}
In Notation \ref{n.Sp}, let $C \subset X$ be a general member of $\sK_x$ at a general point $x \in X^Z \subset X$ and let $\BP TC \subset \sZ$ be the corresponding curve in the VMRT-structure. Then $C \subset X^Z$ and $\BP TC \cap \sE = \emptyset$.  \end{proposition}

\begin{proof}
 At the point $x' (:= \BP T_xC)  \in \sC_x,$ the second fundamental form ${\rm II}_{\sC_x, x'}$ is isomorphic to ${\rm II}_{Z,z}$ in Lemma \ref{l.IIZ} (2). For any point
$y \in C$ and the corresponding point  $y' := \BP T_{y} C \in \ \sC_{y},$  the second fundamental form ${\rm II}_{\sC_{y}, y'}$
 is isomorphic to ${\rm II}_{\sC_x, x'}$ by Theorem \ref{t.standard} (3).
Thus by Corollary \ref{c.rigid}, the subvariety $(\sC_y \subset \BP T_y X)$ is isomorphic
to  $(Z \subset \BP W)$ for any $y \in C$. It follows that $C \subset X^Z$. Furthermore, by Lemma \ref{l.IIZ} (3), the point $y' \in \sC_y$ cannot correspond to a point in $E \subset Z$.
Thus $\BP TC \cap \sE = \emptyset.$ \end{proof}

\begin{lemma}\label{l.Brauer}
Let $C \subset Y$ be a nonsingular rational curve in a complex manifold.
Then there exists a neighborhood $ M \subset Y$ of $C$ such that
\begin{itemize}
\item[(1)] any $\BP^k$-bundle on $M$ is the projectivization of a vector bundle of rank $k+1$ on $M$; and
    \item[(2)]  for a line bundle $\xi$ on $M,$ the degree of $\xi|_C$ is  even if and only if there exists a line bundle $\zeta$ on $M$ such that $\xi = \zeta^{\otimes 2}.$ \end{itemize} \end{lemma}

        \begin{proof}
        Choose $M$ such that $H^i(M, \Z) \cong H^i(C, \Z)$ for all $i \in \Z$.
     It is well-known that (1) follows from  $H^3(M, \Z) = 0$ (e.g. by Proposition 1.1 in \cite{Sr}).
     (2) follows from the  surjective homomorphism $${\rm Pic}(M) = H^1(M, \sO_M^*)  \to H^2(M, \Z) \cong
     H^2(C, \Z) \cong \Z$$ arising from the exponential sequence.
     \end{proof}

\begin{proposition}\label{p.germ}
  In Proposition \ref{p.extensionsp}, fix  a general member $C_0 \subset X^Z$ of $\sK$.
  Then
  there exist vector bundles $\sU$ and $\sQ$ on  a Euclidean neighborhood $M \subset X^Z$ of $C_0 \subset X^Z$ admitting isomorphisms of projective bundles  $$\BP \sU \cong \sS|_M \mbox{ and } \BP \sQ \cong \sR$$ and  isomorphisms of vector bundles
$$\sD|_M  \cong \sU \otimes \sQ \mbox{ and } \sB|_M \cong \Sym^2 \sU.$$ \end{proposition}

\begin{proof}
Fix a neighborhood $M\subset X^Z$ of $C_0 \subset \BP^1$ as in Lemma \ref{l.Brauer}. There exists a vector bundle
$\sU'$  on $M$ such that $\sS|_M \cong  \BP \sU'$.
Then the two projective bundles $\BP \sB|_M$ and $\BP \Sym^2 \sU'$ are isomorphic.
Thus there is a line bundle $\xi$ on $M$ such that $\sB|_M \cong \xi \otimes \Sym^2 \sU'.$

 We claim that $\xi$ has even degree on $C_0$.
 By Proposition \ref{p.extensionsp}, the curve $\BP T C_0 \subset \sZ$ is disjoint from $\sE$.
Thus $\varepsilon: TM \to \sB|_M$ sends $T C_0$ isomorphically to a line subbundle  $\eta \subset \sB|_{C_0}$ such that $$\BP \eta = \overline{\varepsilon}(\BP TC_0)  \subset \sS \subset \BP \sB.$$
Since $\eta$ is an isomorphic image of $TC_0 \subset TM|_{C_0} \to \sB |_{C_0},$ we see that $\eta|_{C_0}$ has degree 2. From $$\BP \eta \subset \sS|_M  \subset \BP \sB|_M \cong  \BP \Sym^2 \sU'$$ where the second inclusion comes from the second Veronese embedding,  there exists a line subbundle $\lambda \subset \sU'|_{C_0}$ such that there is an isomorphism of line subbundles on $C_0$
$$( \eta \subset \sB|_{C_0} ) \ \cong \ (\xi \otimes \lambda^{\otimes 2}  \subset \xi \otimes \Sym^2 \sU'|_{C_0} \cong \sB|_{C_0}). $$
Thus $\xi \otimes \lambda^{\otimes 2}$ has degree 2. It follows that $\xi|_{C_0}$ has even degree, proving the claim.

 By the claim and Lemma \ref{l.Brauer}, we can find a
 line bundle $\zeta$ on $M$ such that $\zeta^{\otimes 2} = \xi$. Then $\sU := \sU' \otimes \zeta$ satisfies $\Sym^2 \sU \cong \sB|_M$.

 Applying Lemma \ref{l.Brauer} to the projective bundle $\sR|_M$, we obtain a vector bundle
 $\sQ'$ on $M$ such that $\BP \sQ' \cong \sR|_M$.
   From $$\BP \sQ' \times_M \BP \sU \cong \sR \times_M \sS \subset \BP \sD$$ where the inclusion comes from the fiberwise Segre embedding, we see that  $\BP (\sQ' \otimes \sU)
   \cong \BP \sD|_M$. Thus there exists a line bundle $\theta$ on $M$ such that
$\theta \otimes \sQ' \otimes \sU \cong \sD|_M$. Set $\sQ:= \theta \otimes \sQ'$ to finish the proof.   \end{proof}

\begin{proposition}\label{p.SAF}
 The pair $(M, \sU)$ in Proposition \ref{p.germ} satisfies SAF-condition  and $\sQ$ is of trivial type in the sense of Definition \ref{d.MU}. \end{proposition}

\begin{proof}
By abuse of notation, we will use the same symbols for the fiber bundles
on $X^Z$ in Notation \ref{n.Sp} and their restrictions on $M \subset X^Z$.

For a general member $C \subset M$ of $\sK$, we have $\BP TC \subset \sZ \setminus \sE$ from
Proposition \ref{p.extensionsp}. The image $\overline{\varepsilon}(\BP TC) \subset \sS$
 is a section of $\sS|_C \to C$ and the corresponding section of $\BP \sU |_C \to C$ under the isomorphism $\sS|_M \cong \BP \sU$ of Proposition \ref{p.germ}  determines  a line subbundle $\sL_C \subset \sU|_C$.  Since the square map $u \in \sU \mapsto u^2 \in \Sym^2 \sU$
sends $\sL_C$ to $\varepsilon(TC) \subset \Sym^2 \sU$, we see that $\sL_C$ is a line bundle of degree 1 on $C$.

For a general point $x \in M$, the tangent vectors of general members of $\sK_x$ lying on $M$
span a Euclidean open subset $\sU_x^o \subset \sU_x \setminus \{0\}.$
From the construction of $\sL_C \subset \sU|_C$, we see that  for  each $u \in \sU^o_x$, there exist a holomorphic map $f_u: \BP^1 \to M$ sending the base point $o \in \BP^1$ to $x$; and   a line subbundle  $\sL_u \subset f_u^*\sU$ which has degree 1 and the fiber of  which at $o$ is $\C  u$.

Set $a := {\rm deg}(f_u^* \sU)$ and $b := {\rm deg}(f_u^* \sQ)$. The  exact sequence $$0 \to f_u^*(\sU \otimes \sQ) \to f_u^* TM \stackrel{\varepsilon}{\to} f_u^* \Sym^2 \sU \to 0$$ gives
$$\deg f_u^*(\sU \otimes \sQ) + \deg f_u^* \Sym^2 \sU = \deg f_u^* TM.$$ Using
 $\deg f_u^* TM = 2 + \dim \sC_x = m+ n+1$ from Theorem \ref{t.standard} (4), we obtain
\begin{equation}\label{e.bm} b m + (m+n+1)(a-1) =0. \end{equation}
By Theorem \ref{t.standard} (5), the collection of affine tangent spaces  $$\{ T_{y'} \widehat{\sC}_y \subset T_y M, y \in f_u(\BP^1)\} $$  determines a subbundle $T^{+} \subset f_u^* TM$ which
corresponds to the sum of the positive factors $$\sO(2) \oplus \sO(1)^{m+n-1} \subset \sO(2) \oplus \sO(1)^{m+n-1} \oplus \sO^{\dim X -m-n} \ \cong f_u^* TM.$$
  The natural exact sequence in  Lemma \ref{l.tanZ} induces an exact sequence of vector bundles
$$ 0 \to \sL_u \otimes f_u^* \sQ \to T^+   \to \sL_u \otimes f_u^*\sU \to 0.$$
It follows that $$ 0 \to f_u^* \sQ \to T^+ (-1) \to f_u^* \sU \to 0,$$ which implies that
$$\deg (f_u^* \sQ) + \deg (f_u^* \sU)= \deg (T^+ (-1)).$$ This reads $b + a =1$. Combining it
with (\ref{e.bm}), we obtain $b=0$ and $a=1$. From the surjection
 $$\sO(2) \oplus \sO(1)^{m+n-1} \oplus \sO^{\dim X -m-n} \cong f_u^* TM  \stackrel{\varepsilon}{\to} f_u^* \Sym^2 \sU,$$ we see that
 $f_u^* \sU \cong \oplus_{i=1}^m \sO(a_i), a_1 \geq \cdots \geq a_m \geq 0.$ So
 $1 = a = \sum_i a_i$ implies that $a_1 =1$ and $a_2 = \cdots = a_m=0$. It follows that
$f_u^* \sU = \sL_u \oplus \sO^{m-1}$. This shows that $(M, \sU)$ satisfies SAF-condition.

 Note that in the exact sequence
$$ 0 \to \sL_u \otimes f_u^* \sQ \to T^+  \to \sL_u \otimes f_u^*\sU \to 0,$$ the intersection of
the line subbundle $({\rm d} f_u)T\BP^1 \subset T^+$ and the subbundle $\sL_u \otimes f_u^* \sQ$ is  0,
from $\BP TC \cap \sE = 0$ in Proposition \ref{p.extensionsp}.
Thus the image of $\sL_u \otimes f_u^* \sQ$ in $T^+ \cong \sO(2) \oplus \sO(1)^{m+n-1}$ is contained in $\sO(1)^{m+n-1}$.
This means that $f_u^* \sQ \cong \sO(b_1) \oplus \cdots \oplus \sO(b_n)$ with
$b_i \leq 0$. But $\deg f_u^*\sQ = b =0$.
Thus $f_u^* \sQ$ must be trivial. This proves that $\sQ$ is of trivial type.
\end{proof}

As a direct consequence of Proposition \ref{p.SAF} and Proposition \ref{p.symbol}, we obtain the following.

\begin{corollary}\label{c.varpi}
In Proposition \ref{p.SAF}, let $$\varpi \in \Hom( \wedge^2 \sD, TM/\sD) = \Hom(\wedge^2 (\sU \otimes \sQ), \Sym^2 \sU)$$ be the Frobenius bracket tensor of $\sD$.
 Then there exist a nonnegative integer ${\rm \bf n}_X$, a nonempty Zariski-open subset $M_o \subset M$ and an element $\omega^M \in H^0(M, \wedge^2 \sQ^*)$ with the following properties.
 \begin{itemize} \item[(1)] $(M_o, \sU|_{M_o})$ satisfies SAF-condition.
 \item[(2)]
For any $x \in M_o$, the nullity of $\omega^M_x:= \omega^M|_{\sQ_x}$ is ${\rm \bf n}_X$ and $$\varpi_x (u \otimes p, v \otimes q) = \omega^M_x(p,q) u \odot v $$ for $u, v \in \sU_x$ and $p, q \in \sQ_x$.\end{itemize} \end{corollary}

Now we prove Theorem \ref{t.Moksp} in the following form.

\begin{theorem}\label{t.VMRT}
In the setting of Corollary \ref{c.varpi},
let $(V, \omega)$ be a presymplectic vector space of dimension $2m+n$ and nullity ${\rm \bf n}_{\omega} = {\rm \bf n}_X$.
 Then a general member $C \subset M_o$ of $\sK$ has a neighborhood
$O_C \subset M_o$ with an open immersion $h: O_C \to \Go(m, V;0)$ such that $h$ sends $C$ to a general  line in the presymplectic Grassmannian.  \end{theorem}

\begin{proof}
We define a Cartan connection of type $(G, G^0)$ on $M_o$ for a suitable pair $(G, G^0)$ with $G/G^0 = \Go(m, V;0)$ as follows.

First consider  the case $n = {\rm \bf n}_X$ ( i.e. $\varpi=0$ ) in Corollary \ref{c.varpi}.
Comparing  Notation \ref{n.Crank0} and Lemma \ref{l.bfG}, we see that ${\bf G} = G_0$ and
our $\sP$ is just a $G_0$-structure on $M$ with integrable $\sD$. Thus by Theorem \ref{t.Crank0} and Proposition \ref{p.SAF}, we have a Cartan connection of type $(G, G^0)$ on $M_o$.

If ${\rm \bf n}_X < n$ (i.e. $\varpi \neq 0$) in Corollary \ref{c.varpi},  we can assume that
$(V, \omega)$ arises from a presymplectic form  $\omega \in \wedge^2 Q^*$ of nullity ${\rm \bf n}_X$ in the way described in Definition 4.10.
 Consider $(\fg_0, \fg_-)$ as in Notation \ref{n.positive} associated with the presymplectic space $(Q, \omega)$.
Then the distribution $\sD \subset TM_o$ is a filtration of type $\fg_-$.
From the ${\bf G}$-principal bundle $\sP \to M_o$, we
 can obtain  a $G_0$-structure subordinate to $\sD$ as follows.
Let ${\bf G}' \subset {\bf G}$ be the image of
$$\Hom(U, Q) \sd ({\rm GL}(U) \times {\rm Sp}_{\omega}(Q)).$$
 Then we have a ${\bf G}'$-principal subbundle $\sP' \subset \sP$ on $M_o$ consisting of
 elements $\varphi \in \sP_x \subset {\rm Isom}(W, T_xM_o)$ such that
 $$\varphi ([w, w']) \equiv  \varpi (\varphi(w), \varphi(w')) \mod \sD_x $$  for all
$w, w' \in U \otimes Q \subset W.$ By $ G_0 = {\bf G}'/ \Hom(U,Q),$ the ${\bf G}'$-bundle $\sP'$ induces a $G_0$-structure $E_0 \subset {\rm grFr}(M)$.
Thus by Theorem \ref{t.Cpositive} and Proposition \ref{p.SAF}, we have a Cartan connection of type $(G, G^0)$ on $M_o$.

In either cases, the Cartan connection of type $(G, G^0)$ on $M_o$ is locally equivalent to that of $G/G^0$  by Proposition \ref{p.Biswas} and Proposition \ref{p.ratconn}.
By Propositions \ref{p.spgrade} and \ref{p.spgraderank0}, the homogeneous space $G/G^0$ is $\Go(m, V;0)$. Proposition \ref{p.develop} gives an immersive holomorphic map $h:O_C \to G/G^0$ defined on a neighborhood $O_C$ of a general member $C$ of $\sK$ contained in $M_o$, which lifts to an equivalence of Cartan connections.

The derivative ${\rm d} h: \BP T O_C \to \BP T \Go(m, V;0)$ sends the VMRT $\sC_x$ at a general point $x \in O_C$ to the VMRT of lines on $\Go(m, V)$ at $h(x)$. This implies that $h$ sends the germ of $C$ to the germ of a line in $\Go(m, V)$ (e.g. by Section 3.2 Step 1 in \cite{Hw01}). As $h$ is immersive, it must send $C$ biregularly to a general line in $\Go(m,V)$.     \end{proof}

We are ready to prove Theorems \ref{t.recog} and \ref{t.deform}.

\begin{proof}[Proof of Theorem \ref{t.recog}]
Now assume that $X$ is a Fano manifold of Picard number 1 in Theorem \ref{t.VMRT}.
Regard $\fg$ as the Lie algebra of holomorphic vector fields on $\Go(m, V;0).$
In the neighborhood $O_C$ of Theorem \ref{t.VMRT}, we have a Lie algebra $h^*\fg$ of holomorphic vector fields arising from $\fg$ via $h$. The elements of $h^*\fg$ generate local biholomorphisms sending germs of elements of $\sK$ to germs of elements of $\sK$.
Thus by Theorem \ref{t.CF} they can be extended to
global holomorphic vector fields on $X$. It follows that there exists a  connected algebraic group
$\widetilde{G}$ isogenous to $G$  acting on $X$ with an open orbit $O \subset X$ such that
$h$ can be extended to an \'etale morphism $ O \to G/G^0$.
 Since $G/G^0$ is simply connected by Proposition \ref{p.cell} (1), we conclude that $O \to G/G^0$ is biregular. In other words, the Fano manifold  $X$ of Picard number 1 contains a Zariski-open subset biregular to $G/G^0$. Thus $X$ is biregular to a symplectic Grassmannian or an odd-symplectic Grassmannnian by Theorem \ref{t.SpPic}. \end{proof}

\begin{proof}[Proof of Theorem \ref{t.deform}]
The central fiber $X:= \pi^{-1}(0)$ is a Fano manifold of Picard number 1.
By Theorem \ref{t.recog}, it suffices to show that for a suitable choice of $\sK$ on $X$,  the VMRT $\sC_x \subset \BP T_x X$ at a general point  $x \in X$ is isomorphic to $\sC_s \subset \BP T_s S$.
  This is exactly Proposition 3.5.2 of \cite{HM05}, which was stated for a symplectic Grassmannian $S$, but the proof given there works verbatim for an odd-symplectic Grassmannian. \end{proof}

\bigskip
{\bf Acknowledgment} We would like to thank Andi \v{C}ap, Boris Doubrov, Jaehyun Hong, Shin-Young Kim and Tohru Morimoto for valuable discussions on Tanaka theory. We are grateful to Baohua Fu and Nagiming Mok for all the help and  encouragement over the years.

\bigskip
Korea Institute for Advanced Study,
Seoul, Republic of Korea

jmhwang@kias.re.kr

qifengli@kias.re.kr
\end{document}